\documentclass[12pt,oneside,reqno]{amsart}
\usepackage{etex}
\usepackage{pifont}
\usepackage{soul}
\usepackage[dvips]{graphicx}
\usepackage{amsmath,latexsym,amsthm,cmap,indentfirst,setspace,ccaption,amssymb,color,amscd,mathrsfs,euscript}

\usepackage[normalem]{ulem}
\usepackage[svgnames]{xcolor}

\newcommand{\nicecolor}{Navy}
\usepackage[pdfauthor={J. Blanc, J. Schneider \& E. Yasinsky}, backref=page, colorlinks, linktocpage, citecolor=\nicecolor, linkcolor=\nicecolor, urlcolor=\nicecolor]{hyperref}
\usepackage{array}

\usepackage{amssymb,lipsum}
\usepackage{float,xypic}
\usepackage{cmap,longtable,booktabs}
\usepackage[all,cmtip]{xy}

\usepackage{geometry}
\usepackage{caption, subcaption}

\usepackage[shortlabels]{enumitem}
\setlist[1]{wide}
\setlist[2]{leftmargin=15mm}
\setlist[enumerate]{label=\rm{(\arabic*)}}
\setlist[enumerate,2]{label=\rm({\it\roman*}), }
\setlist[itemize]{label=\raisebox{0.25ex}{\tiny$\bullet$}}
\definecolor{grisclair}{rgb}{0.9,0.9,0.9}

\usepackage{tikz}
\usetikzlibrary{cd,arrows,positioning}
\tikzset{>=stealth}
\tikzcdset{arrow style=tikz}
\tikzset{link/.style={column sep=1.8cm,row sep=0.16cm}}
\tikzset{map/.style={row sep=0em, column sep=0em}}

\DeclareFontFamily{U}{mathb}{\hyphenchar\font45}
\DeclareFontShape{U}{mathb}{m}{n}{
	<5> <6> <7> <8> <9> <10> gen * mathb
	<10.95> mathb10 <12> <14.4> <17.28> <20.74> <24.88> mathb12
}{}
\DeclareSymbolFont{mathb}{U}{mathb}{m}{n}
\DeclareMathSymbol{\bigast}{1}{mathb}{"06}

\DeclareFontFamily{U}{mathx}{\hyphenchar\font45}
\DeclareFontShape{U}{mathx}{m}{n}{<-> mathx10}{}
\DeclareSymbolFont{mathx}{U}{mathx}{m}{n}
\DeclareMathAccent{\widebar}{0}{mathx}{"73}

\makeatother 

\usepackage{verbatim}
\hypersetup{
    bookmarks=true,         
    unicode=true,          
    pdftoolbar=true,        
    pdfmenubar=true,        
    pdffitwindow=false,     
    pdfnewwindow=true,      
    colorlinks=true,       
    linkcolor=blue,          
    citecolor=blue,        
    filecolor=magenta,      
    urlcolor=cyan           
}

\DeclareFontFamily{U}{mathx}{\hyphenchar\font45}
\DeclareFontShape{U}{mathx}{m}{n}{<-> mathx10}{}
\DeclareSymbolFont{mathx}{U}{mathx}{m}{n}
\DeclareMathAccent{\widebar}{0}{mathx}{"73}

\renewcommand{\to}{ \, \tikz[baseline=-.6ex] \draw[->,line width=.5] (0,0) -- +(.5,0); \, }
\renewcommand{\rightarrow}{ \, \tikz[baseline=-.6ex] \draw[->,line width=.5] (0,0) -- +(.5,0); \, }
\newcommand{\longto}{ \, \tikz[baseline=-.6ex] \draw[->,line width=.5] (0,0) -- +(.9,0); \, }

\newcommand{\rat}{ \, \tikz[baseline=-.6ex] \draw[->,densely dashed,line width=.5] (0,0) -- +(.5,0); \, }
\renewcommand{\dasharrow}{\rat}
\renewcommand{\dashrightarrow}{\rat}

\newcommand{\ps}{ \, \tikz[baseline=-.6ex] \draw[->,dotted,line width=.6] (0,0) -- +(.5,0); \,}

\renewcommand{\mapsto}{ \, \tikz[baseline=-.6ex] \draw[|->,line width=.5] (0,0) -- +(.5,0); \, }
\newcommand\iso{\stackrel{\sim}{\to}}
\newcommand{\tto}{ \, \tikz[baseline=-.6ex] \draw[->>,line width=.5] (0,0) -- +(.5,0); \, }
\renewcommand{\twoheadrightarrow}{\tto}
\newcommand{\hookto}{ \, \tikz[baseline=-.6ex] \draw[right hook->,line width=.5] (0,0) -- +(.5,0); \, }
\renewcommand{\hookrightarrow}{\hookto}

\newcommand{\Ra}{2.6cm}
\newcommand{\Rb}{1.70cm}
\renewcommand{\Rc}{0.85cm}

\usepackage{tikz}

\usetikzlibrary{cd,arrows,positioning,calc}

\usetikzlibrary{shapes.geometric}

\tikzset{map/.style={row sep=0em, column sep=0em}}

\usetikzlibrary{intersections, through}

\DeclareMathOperator{\Bir}{Bir}
\DeclareMathOperator{\Spec}{Spec}
\DeclareMathOperator{\rk}{rk}
\DeclareMathOperator{\Pic}{Pic}

\DeclareMathOperator{\Aut}{Aut}

\DeclareMathOperator{\PGL}{PGL}
\DeclareMathOperator{\GL}{GL}

\DeclareMathOperator{\Sym}{Sym}

\DeclareMathOperator{\Br}{Br}

\DeclareMathOperator{\Gal}{Gal}

\DeclareMathOperator{\Mat}{Mat}

\DeclareMathOperator{\Proj}{Proj}

\DeclareMathOperator{\charact}{char}

\DeclareMathOperator{\indexx}{ind}

\DeclareMathOperator{\Cohom}{H}

\DeclareMathOperator{\BirMori}{BirMori}

\DeclareMathOperator{\cg}{cov.gen}
\DeclareMathOperator{\covgon}{cov.gon}
\DeclareMathOperator{\gon}{gon}

\theoremstyle{plain}
\newtheorem{thm}{Theorem}[subsection]
\newtheorem{lem}[thm]{Lemma}
\newtheorem{cor}[thm]{Corollary}
\newtheorem{prop}[thm]{Proposition}

\newtheorem{maintheorem}{Theorem}

\newtheorem{maincorollary}[maintheorem]{Corollary}

\theoremstyle{definition}
\newtheorem{mydef}[thm]{Definition}
\newtheorem{rem}[thm]{Remark}

\newtheorem{ex}[thm]{Example}

\theoremstyle{definition}
\newtheorem*{Question}{Question}

\makeatletter
\g@addto@macro{\endabstract}{\@setabstract}
\newcommand{\authorfootnotes}{\renewcommand\thefootnote{\@fnsymbol\c@footnote}}%
\makeatother

\title[Severi-Brauer surfaces]{Birational maps of Severi-Brauer surfaces, with applications to Cremona groups of higher rank}

\author{J\'er\'emy Blanc}
\address{J\'er\'emy Blanc, Universit\'e de Neuch\^atel,
Institut de Math\'ematiques,
Rue Emile-Argand 11,
Switzerland \href{mailto:jeremy.blanc@unine.ch}{jeremy.blanc@unine.ch}}

\author{Julia Schneider}
\address{Julia Schneider, 
Universit\"at Z\"urich, Institut f\"ur Mathematik
Winterthurerstrasse 190,
CH-8057 Z\"urich, \href{mailto:julia.schneider@math.ch}{julia.schneider@math.ch}}

\author{Egor Yasinsky}
\address{
Egor Yasinsky, IMB, Universit\'{e} de Bordeaux, 351 Cours de la Lib\'{e}ration,
	33405 Talence Cedex, France, \href{mailto:egor.yasinsky@u-bordeaux.fr}{egor.yasinsky@u-bordeaux.fr}}

\thanks{
J.S.~was first supported by the Swiss National Science Foundation project P2BSP2\_200209 and hosted by the Institut de Math\'ematiques de Toulouse, and then partially supported by the ERC starting grant $\#$804334.}


\newcommand{\I}{\ensuremath{\mathrm{I}}}
\newcommand{\II}{\ensuremath{\mathrm{II}}}
\newcommand{\III}{\ensuremath{\mathrm{III}}}
\newcommand{\IV}{\ensuremath{\mathrm{IV}}}

\newcommand{\CC}{\mathbb C}
\newcommand{\QQ}{\mathbb Q}

\newcommand{\id}{\mathrm{id}}
\newcommand{\NN}{\mathbb N}

\newcommand{\kk}{K}
\newcommand{\RR}{\mathbb R}
\renewcommand{\AA}{\mathbb A}
\newcommand{\PP}{\mathbb P}
\newcommand{\ZZ}{\mathbb Z}

\newcommand{\pt}{\mathrm{pt}}

\makeatletter
\def\@tocline#1#2#3#4#5#6#7{\relax
	\ifnum #1>\c@tocdepth 
	\else
	\par \addpenalty\@secpenalty\addvspace{#2}%
	\begingroup \hyphenpenalty\@M
	\@ifempty{#4}{%
		\@tempdima\csname r@tocindent\number#1\endcsname\relax
	}{%
		\@tempdima#4\relax
	}%
	\parindent\z@ \leftskip#3\relax \advance\leftskip\@tempdima\relax
	\rightskip\@pnumwidth plus4em \parfillskip-\@pnumwidth
	#5\leavevmode\hskip-\@tempdima
	\ifcase #1
	\or\or \hskip 3em \or \hskip 4em \else \hskip 5em \fi%
	#6\nobreak\relax
	\hfill\hbox to\@pnumwidth{\@tocpagenum{#7}}\par
	\nobreak
	\endgroup
	\fi}
\makeatother

\begin{document}

	\maketitle

\begin{abstract}

	We prove that any group of cardinality at most the one of $\mathbb{C}$ is a quotient of any Cremona group of rank at least $4$. This provides a definitive answer to the question of what the quotients of Cremona groups can be. As a consequence, this gives a negative answer to the question of I. Dolgachev of whether Cremona groups of all ranks are generated by involutions. As another application, we show that higher Cremona groups do not enjoy some classical group-theoretic properties (namely, the Hopfian property) which are satisfied by Cremona groups of rank 2. Finally, we discover that the $3$-torsion of the Cremona group of rank at least $4$ is not countable.

	To deduce these properties of higher Cremona groups, we first describe the group of birational transformations of a non-trivial Severi-Brauer surface $S$ over a perfect field, proving in particular that if $S$ contains a point of degree $6$, then its group of birational self-maps is not generated by elements of finite order as it admits a surjective group homomorphism to~$\mathbb{Z}$. We then use this result to study Mori fibre spaces over the field of complex numbers, for which the generic fibre is a non-trivial Severi-Brauer surface.

\end{abstract}

\setcounter{tocdepth}{2}

{
	\hypersetup{linkcolor=blue}
	\tableofcontents
}

\section{Introduction}\label{sec:intro}

\subsection{Higher Cremona groups} The question of the (non)simplicity of Cremona groups, i.e., groups of birational transformations of the projective space $\PP_\kk^n$, dates back at least to F.~Enriques \cite[p. 116]{Enriques}. It is now known that, contrary to Enriques' own expectations, these groups do contain proper normal subgroups. This was first shown by S. Cantat and S. Lamy in dimension $n=2$ over $\kk=\CC$ using methods from geometric group theory \cite{CantatLamy}, and then for any dimension $n\geqslant 3$ and $\kk\subseteq\CC$ by the first author, S. Lamy, and S. Zimmermann \cite{BLZ}, using methods from modern birational geometry: the minimal model program, the Sarkisov program, and the boundedness of Fano varieties, established by C.~Birkar \cite{Birkar}. Recently, there have been numerous results about the structure of this group, see e.g.~\cite{CantatLamy,Lonjou,Zimmermann,LamyZimmermann,BlancYasinsky,Zikas,LamySchneider,ShinderLin,BLZ,Schneider,DGO17} and references therein. At the same time, all these articles leave the following question open: \emph{what are possible quotients of Cremona groups?} In this paper, we provide the \emph{ultimate} answer for any dimension $\geqslant 4$:

\begin{maintheorem}\label{thm: free product}
	For every integer $m\geqslant 4$, there is a surjective group homomorphism
	\[
	\Phi\colon \Bir_{\CC}(\PP^m)\twoheadrightarrow \mathcal{F}(\CC),
	\]
	where $\mathcal{F}(\CC)$ is the free group over the set $\CC$. In particular, $\Bir_{\CC}(\PP^m)$ is not generated by elements of finite order and admits a surjective homomorphism to any group whose cardinality is at most the cardinality of $\CC$.
\end{maintheorem}

As an example, for each complex algebraic variety $X$, there is an abstract surjective group homomorphism $\Bir_{\CC}(\PP^{m})\twoheadrightarrow \Bir_{\CC}(X)$, where $m\geqslant 4$.

\medskip

Theorem \ref{thm: free product} implies a negative answer to the following question of I. Dolgachev:

\begin{Question}[I. Dolgachev]
	Is the group $\Bir_{\CC}(\PP^m)$ generated by involutions?
\end{Question}

For $m=2$ the answer to Dolgachev's question is positive and follows from the classical Noether-Castelnuovo theorem which states that $\Bir_{\CC}(\PP^2)$ is generated by $\Aut_{\CC}(\PP^2)\simeq\PGL_3(\CC)$ and the standard Cremona involution $[x:y:z]\mapsto [yz:xz:xy]$, see~\cite{CerveauDeserti}. More generally, as was recently proven in \cite{LamySchneider}, the group $\Bir_{\kk}(\PP^2)$ is generated by involutions over any perfect field $K$.

\medskip

Another application of Theorem \ref{thm: free product} concerns the \emph{Hopfian property}. Recall that a group $\Gamma$ is called {\it Hopfian} if any surjective endomorphism $\Gamma\twoheadrightarrow\Gamma$ is an isomorphism.  Equivalently, a group is Hopfian if and only if it is not isomorphic to any of its proper quotients. Examples of Hopfian groups include all finite groups, all  finitely-generated free groups, the group of rationals $\QQ$, and many more. In \cite{Deserti} it was proven that the complex plane Cremona group $\Bir_{\CC}(\PP^2)$ is Hopfian. By contrast, Theorem~\ref{thm: free product} implies that this is wrong in higher dimensions:

\begin{maincorollary}\label{CorHopfian}
	For each $m\geqslant 4$, the group  $\Bir_{\CC}(\PP^{m})$  is not Hopfian.
\end{maincorollary}

Finally, our techniques allow to get some insights into the structure of the abelianisation of higher Cremona groups. Recall that the abelianisation of a group $G$ is the quotient $G/[G,G]$, where $[G,G]$ denotes the commutator subgroup. The abelianisation of $\Bir_{\CC}(\PP^2)$ is trivial, by the Noether-Castelnuovo theorem, but the one of $\Bir_{\RR}(\PP^2)$ is isomorphic to $\bigoplus_{\RR} \ZZ/2\ZZ$ by the main result of \cite{Zimmermann}. For each perfect field $K$, as $\Bir_{\kk}(\PP^2)$ is generated by involutions, its abelianisation is a $2$-torsion group. Theorem~\ref{thm: free product} implies that the abelianisation of $\Bir_{\CC}(\PP^{ m})$ is not a torsion group for $m\geqslant 4$. The results of \cite{BLZ} imply however that its $2$-torsion is not trivial. Here we prove

\begin{maintheorem}\label{3Torsion}
	For each $m\geqslant 4$, the $3$-torsion of the abelianisation of  $\Bir_{\CC}(\PP^m)$  is  not countable.
\end{maintheorem}
This raises the following question:
\begin{Question}[possible orders of elements in the abelianisation for Cremona groups]
	Let $N\geqslant 4$ be an integer. Does there exist $m\geqslant 1$ such that the $N$-torsion of the abelianisation of  $\Bir_{\CC}(\PP^m)$ is not trivial? For $m>2$ fixed, what are the possible orders of elements of the abelianisation of $\Bir_\CC(\PP^m)$?
\end{Question}

\subsection{Severi-Brauer surfaces}

The crucial insight allowing for the construction of arbitrary quotients of Cremona groups in Theorem \ref{thm: free product} is that projective spaces of dimension $\geqslant 4$ admit birational models fibered in \emph{non-trivial Severi-Brauer surfaces}. In turn, the groups of birational self-maps of these surfaces themselves possess quite exotic properties.

\medskip

Let $\kk$ be a field. A {\it Severi-Brauer} variety over $\kk$ is a projective algebraic variety $X$ such that $X_{\overline{\kk}}=X\times_{\Spec\kk}\Spec\overline{\kk}\simeq\PP^{m}_{\overline{\kk}}$. By the classical Ch\^{a}telet's theorem, a Severi-Brauer variety $X$ over $\kk$ is isomorphic to $\PP_\kk^{m}$ if and only if $X(\kk)\ne\varnothing$; we say that in this case the Severi-Brauer variety is \emph{trivial}. So, brational transformations of a trivial Severi-Brauer variety $X=\PP_\kk^{m}$ consitute the Cremona group $\Bir_{\kk}(\PP^{m})$ of rank~$m$. At the same time, the group $\Bir_K(X)$ of birational transformations of a {\it non-trivial} Severi-Brauer variety $X$ over a perfect field $K$ got a lot less attention. The surface case was first studied by Felix Weinstein in his Master's thesis \cite{Weinstein89}, republished recently as \cite{Weinstein22}, then more recently in \cite{Shramov1,Shramov2,ShramovVologodsky,TrepalinSB,Vikulova}. In particular, for a non-trivial Severi-Brauer surface~$S$, generators and finite subgroups of $\Bir_K(S)$ were studied. Here, we go one step further and describe quotients of $\Bir_K(S)$. It turns out that this group exhibits some properties which are very different from the properties of the plane Cremona group (which was already observed by C. Shramov in the context of finite group actions on $S$ and the Jordan property).

\medskip

Let $S$ be a non-trivial Severi-Brauer surface over a perfect field $K$. Although $S$ does not contain any $K$-rational point, it always contains points of degree $3$. Furthermore, if $S$ contains a closed point of degree $d\geqslant 3$, which corresponds to $d$ $\overline{K}$-points forming one orbit of $\Gal(\overline{K}/K)$, then $d$ is always a multiple of $3$, see Corollary~\ref{cor: SB contains a point of degree p}. Our second main result is the following.

\begin{maintheorem}\label{theoremSB}
	Let $S$ be a Severi-Brauer surface over a perfect field $K$, with $S(K)=\varnothing$. We fix an algebraic closure $\overline{K}$ of $K$ and denote for $d\in \{3,6\}$ by $\mathcal{P}_d$ the set of degree $d$ points of $S$ up to the action of $\Aut_K(S)$. Then, $\lvert\mathcal{P}_3\rvert \geqslant 2$ and for each $p\in \mathcal{P}_3$, there is a surjective group homomorphism
	\[
	\Psi\colon \Bir_K(S)\to \bigoplus\limits_{\mathcal{P}_3\setminus \{p\}} \ZZ/3\ZZ \ast \left ( \bigast_{\mathcal{P}_6} \mathbb{Z} \right).
	\]
	In particular, $\Bir_K(S)$ is not a perfect group $($and is thus not simple$)$ and is not generated by involutions. Moreover, if $\mathcal{P}_6\not=\varnothing$, then $\Bir_K(S)$ is not generated by elements of finite order.
\end{maintheorem}

Theorem \ref{theoremSB} gives the first example of del Pezzo surfaces (or more generally geometrically rational surfaces) whose groups of birational transformations are not generated by elements of finite order.

\medskip

To deduce Theorem \ref{thm: free product} from Theorem  \ref{theoremSB}, we study Mori fibre spaces $X\to B$ over the field~$\CC$ of complex numbers, such that the generic fibre is a non-trivial Severi-Brauer surface over~$\CC(B)$. Such spaces were previously studied by T.~Maeda \cite{MaedaMaps,MaedaModels}, A.~Kresch and Yu.~Tschinkel \cite{KreschTschinkelModels,KreschTschinkel} (in the context of rationality problems), and others. Note that one needs $\dim(B)\geqslant 2$, as otherwise $\CC(B)$ has the $C_1$ property and hence there are no non-trivial Severi-Brauer surfaces over $\CC(B)$; in particular, $\dim(X)\geqslant 4$. We show that some birational maps of such fibrations appear in very few relations, and obtain in particular a group homomorphism from $\Bir_{\CC}(X)$ to products of free groups with $\ZZ/3\ZZ$, as in Theorem~\ref{theoremSB} (see Theorem~\ref{Theorem:SBMfs}).

\subsection{Relation to other work}

\subsubsection{Relation with \cite{BLZ} and SQ-universality of Cremona groups}
As explained above, \cite{BLZ} was the first article that proved the non-simplicity of $\Bir_{\CC}(\PP^m)$ for $m\geqslant 3$. The quotients constructed there are of the form $\bigoplus_{\CC}\ZZ/2\ZZ$ or $\bigast_{\CC}\ZZ/2\ZZ$, and thus are quite large but always generated by involutions.

The fact that every group of the right cardinality is a quotient of  $\Bir_{\CC}(\PP^m)$ for $m\geqslant 4$ is much stronger than all previously known results in this direction. As a comparison, it is known that groups $\Bir_{\CC}(\PP^2)$ and $\Bir_{\CC}(\PP^m)$, $m\geqslant 3$, are \emph{SQ-universal}, see \cite{CantatLamy,DGO17} and \cite[Corollary 8.4]{BLZ}, respectively. SQ-universality means that every countable group can be embedded in a quotient of $\Bir_{\CC}(\PP^m)$. For $\Bir_\CC(\PP^m)$ with $m\geqslant 4$, we can replace ``countable'' by ``having at most the cardinality of $\CC$'' and ``be embedded in'' with ``being equal to'', which is of course much stronger. Note that in particular we get a far reaching generalisation of all the results from \cite[8.A]{BLZ}.

\subsubsection{Motivic invariants of birational maps}

A completely different approach to this problem, via motivic invariants of birational maps, was recently developed by H.-Y. Lin and E.~Shinder, who constructed a non-trivial homomorphism $\Bir_{\kk}(\PP^m)\to\ZZ$ for $m\geqslant 4$ and $\kk=\CC$, or for $m=3$ and $\kk$ a number field (see \cite[Theorem 1.2]{ShinderLin} for a more precise statement); this answers negatively Dolgachev's question for $m\geqslant 4$, the remaining case is then $m=3$. One might actually derive from \cite{ShinderLin} a surjective group homomorphism to $\bigoplus\limits_{\CC} \ZZ$, but their technique cannot be used to obtain a non-abelian quotient, and thus cannot give Theorem~\ref{thm: free product}.

\medskip

Note that the construction of Theorem \ref{thm: free product} is {\it fundamentally} different from that of \cite{BLZ} and \cite{ShinderLin}: in fact, the simplest birational maps which generate the image of $\mathcal{F}(\CC)$ lie in the kernel of the homomorphisms constructed in these two articles.

\medskip

{\bf Acknowledgements.} We would like to thank Asher Auel, Marcello Bernardara, Olivier Benoist, Fabio Bernasconi, Julie D\'eserti, St\'ephane Druel, Diego Izquierdo, St\'ephane Lamy, Konstantin Loginov, Immanuel van Santen, Evgeny Shinder, and Constantin Shramov for interesting discussions during the preparation of this text.

\section{Preliminaries on Severi-Brauer surfaces}\label{BirSB}

\subsection{Central simple algebras}\label{sec: first definitions}

A {\it Severi-Brauer variety} over a field $\kk$ is a projective algebraic variety $X$ such that $X_{\overline{\kk}}=X\times_{\Spec\kk}\Spec\overline{\kk}\simeq\PP^{n}_{\overline{\kk}}$. For their basic properties we refer to \cite{Artin}, \cite{Kollar_SB} and \cite{Gille}. Severi-Brauer varieties of dimension $n-1$ over $\kk$ are in one-to-one correspondence with central simple $\kk$-algebras of degree~$n$. More precisely, one can observe that central simple algebras of dimension $n^2$ over $\kk$ and Severi-Brauer varieties of dimension $n-1$ over $\kk$ are both described by the continuous cohomology set $\Cohom^1(\Gal(\kk^{\rm sep}/\kk),\PGL_n(\kk^{\rm sep}))$. In fact, this correspondence can be made into an equivalence of categories, see \cite[\S 6]{Jahnel}. A Severi–Brauer variety $X$ always splits over a finite Galois extension $L$ of the base field $K$, i.e.~one has $X_L\simeq\PP_L^n$, see \cite[Corollary 5.1.5]{Gille}.

By contrast with the one-dimensional case, i.e. the case of conics, in dimensions $\geqslant 2$ it is not that easy to give an {\it explicit} equation of a non-trivial Severi-Brauer variety, as the canonical embedding gives a lot of equations in a projective space of large dimension; for some results in this direction, see \cite{BadrBars1,BadrBars2}. Let $X$ be a non-trivial Severi-Brauer surface over $\kk$. By the general theory, it corresponds to a non-trivial central simple algebra $A$ of degree $3$. From the Wedderburn's structural theorem \cite[Theorem 2.1.3]{Gille} it immediately follows that $A$ is a division algebra. On the other hand, by the classical result of Wedderburn and Albert \cite[15.6]{Pierce}, every division algebra of degree 3 is a {\it cyclic algebra}, i.e.~there is a
strictly maximal subfield $L$ of $A$ such that $L/K$ is a cyclic Galois extension. Let $g$ be a generator of $\Gal(L/K)$. Then there is an element $e\in A^*$ such that $e^3=\xi\in K^*$, $A$ is spanned by $1$, $e$ and $e^2$ as an $L$-vector space and the multiplication law is given by $\lambda\cdot e=e\cdot g(\lambda)$ for all $\lambda\in L$. We write $(L/\kk,g,\xi)$ for such an algebra. The image of $X$ (or the corresponding cyclic algebra) in $\Cohom^1(\Gal(\kk^{\rm sep}/\kk),\PGL_n(\kk^{\rm sep}))$ is then given by the cocycle
\[
\nu(g)=\left(\begin{smallmatrix}
	0 & 0 & \xi\\
	1 & 0 & 0 \\
	0 & 1 & 0
\end{smallmatrix}\right)
\]
The cyclic algebra $(L/\kk,g,\xi)$ is trivial if and only if $\xi$ is the norm of an element of $L$, see e.g.~\cite[15.1]{Pierce}. In \cite{BadrBars1} the authors give an explicit biregular model of a Severi-Brauer surface associated with each cyclic algebra $(L/\kk,g,\xi)$. In our paper, we provide much easier {\it birational} models given by explicit cubic equations, see Section \ref{sec: examples}.

\subsection{Twisted linear subvarieties}\label{subsec: twisted} The {\it index} of a proper $\kk$-scheme $X$ is the greatest common divisor of the degrees of all $0$-cycles on $X$. We will denote it by $\indexx(X)$. We say that a closed $\kk$-subvariety $Y\hookrightarrow X$ is a {\it twisted linear subvariety} of $X$ if $Y$ is a Severi-Brauer variety and moreover over $\overline{\kk}$ the inclusion $Y_{\overline{\kk}}\subseteq X_{\overline{\kk}}$ becomes isomorphic to the inclusion of a linear subvariety of the corresponding projective space.
We recall the following result:
\begin{prop}[{\cite[Lemma 27, Theorem 28, Corollary 30 and Theorem 53]{Kollar_SB}}]\label{prop:KollarSB}
	Let $X$ be a Severi-Brauer variety of dimension $d\geqslant 1$ over a field $\kk$. Let  $m\geqslant 0$ be the minimal dimension of a twisted linear subvariety of $X$. Then, the following hold:
	\begin{enumerate}
		\item\label{KollarSB1}
		$m+1$ divides $d+1$;
		\item\label{KollarSB2}
		For each integer $e$ with $m\leqslant e\leqslant d$, there is a twisted linear subvariety of $X$ of dimension $e$ if and only if $m+1$ divides $e+1$;
		\item\label{KollarSB3}
		The twisted linear subvarieties of $X$ of dimension $m$ form a unique isomorphism class. Moreover if $P$ is such a twisted linear subvariety, then $\indexx(P)=\indexx(X)=m+1$ and $X$ is birational to $P\times \mathbb{P}^{d-m}$.
		\item\label{KollarSB4} $X$ contains smooth, $0$-dimensional subschemes of degree $\indexx(X)$.
	\end{enumerate}
\end{prop}
The following is a direct consequence, well-known to specialists, that we will mostly use in the case when $p=3$, i.e. for Severi-Brauer surfaces.
\begin{cor}\label{cor: SB contains a point of degree p}
	Let $p\geqslant 2$ be a prime number and let $X$ be a non-trivial Severi-Brauer variety of dimension $p-1$ over a field $K$. Then, the minimal dimension of a twisted linear subvariety of $X$ is $p-1$. Moreover, $X$ does not contain points of degree $d$ not divisible by $p$. On the other hand, $X$ always contains a closed point of degree $p$.
\end{cor}
\begin{proof}
	Let $m$ be the minimal dimension of a twisted linear subvariety of $X$. By Proposition~\ref{prop:KollarSB}\ref{KollarSB1}, $m+1$ divides $\dim(X)+1=p$, so $m\in \{0,p-1\}$. The case $m=0$ is impossible, as otherwise $X$ would be trivial by Ch\^{a}telet's theorem. Hence, $m=p-1$, which implies that $\indexx(P)=p$ by  Proposition~\ref{prop:KollarSB}\ref{KollarSB3}, and gives the result.
\end{proof}

In particular, every non-trivial Severi-Brauer surface contains a point of degree $3$. We will see in Section \ref{subsec: points 6} that it may fail to have points of degree 6.

\subsection{Points of degree $3$}
Let $X$ be a smooth projective variety over a perfect field $K$. A~point $p\in X$ of degree $d$ is a closed point that corresponds to $d$ $\overline{K}$-points $p_1,\ldots,p_d$, forming one orbit of $\Gal(\overline{K}/K)$. These are the irreducible components of $p_{\overline{K}}$. By the \emph{splitting field of~$p$}, we mean the smallest subfield $L$ of $\overline{K}$ over which the points $p_1,\ldots,p_d$ are defined. Note that  $L/K$ is a finite Galois extension.

\begin{ex}\label{ex: splitting field}
	Let $K=\QQ$, $q$ be a prime and $\zeta$ be a primitive third root of unity. Consider the polynomial $f_q(x)=x^3-q\in K[x]$. Recall that its splitting field is $L=\QQ(\sqrt[3]{q},\zeta)$ and its Galois group $G=\mathrm{Gal}(L/K)$ is the symmetric group $\Sym_3$ generated by
	\[
	\sqrt[3]{q}\mapsto\zeta\sqrt[3]{q},\ \zeta\mapsto\zeta\ \ \ \text{and}\ \ \ \sqrt[3]{q}\mapsto\sqrt[3]{q},\ \zeta\mapsto\zeta^{-1}.
	\]
	The $G$-orbit of the point $[0:1:\sqrt[3]{q}]\in \PP^2(\overline{K})$ consists of 3 points
	\[
	p_{\overline{K}}=\big \{[0:1:\sqrt[3]{q}],[0:1:\zeta\sqrt[3]{q}],[0:1:\zeta^2\sqrt[3]{q}]\big \}\subseteq\PP^2(\overline{K}),
	\]
	corresponding to a point $p\in \PP^2$ of degree $3$. Note that the splitting field of $p$ is $L$ and does not coincide with its residue field, which is simply $\QQ[y,z]/(y-1,z^3-q)\simeq\QQ(\sqrt[3]{q})$ and is not Galois over $\QQ$. Note that by changing the prime $q$, we get infinitely many degree $3$ points with different splitting fields.
\end{ex}

As we already mentioned in Section \ref{sec: first definitions}, Severi-Brauer surfaces correspond to cyclic algebras of degree $3$. Let us put this result into a slightly more geometric form that will be convenient for us in the future.

\begin{lem}\label{lem:auto}
	Let $S$ be a non-trivial Severi-Brauer surface over a perfect field $K$, and let $p\in S$ be a point of degree $3$. Denote by $L$ the splitting field of $p$ over $K$, and by $p_1,p_2,p_3\in S(L)$ the three irreducible components of $p_{L}$.
	\begin{enumerate}

		\item\label{it:auto--Gal} The Galois group $\Gal(L/K)$ acts faithfully and transitively on $\{p_1,p_2,p_3\}$. In particular, it  is either isomorphic to $\Sym_3$ or to $\ZZ/3\ZZ$ and contains a unique element $g\in \Gal(L/K)$ of order $3$ such that $g(p_1)=p_2, g(p_2)=p_3, g(p_3)=p_1$.
		\item\label{it:auto--g} Choosing $g$ as in \ref{it:auto--Gal}, there exists an $L$-isomorphism $\varphi\colon S_L\iso\PP^2_L$ such that $\varphi(p_1)=[1:0:0]$, $\varphi(p_2)=[0:1:0]$, $\varphi(p_3)=[0:0:1]$, and such that $\varphi\circ g\circ \varphi^{-1}=A_g\circ g$, where $A_g=\left(\begin{smallmatrix}0 & 0 &\xi\\ 1&0&0\\ 0&1&0\end{smallmatrix}\right)\in\GL_3(L)$ for some $\xi \in L^g$.
		\item\label{autopts}
		If $q$ is a $3$-point of $S$ having the same splitting field as $p$, there is $\alpha\in\Aut_K(S)$ such that $\alpha(p)=q$.
	\end{enumerate}
\end{lem}

\begin{proof}
	\ref{it:auto--Gal}: Since $p$ is a point of degree $3$, the Galois group $\Gal(L/K)$ acts transitively on $\{p_1,p_2,p_3\}$. This action is moreover faithful, as $L$ is the splitting field of $p$. Hence, $\Gal(L/K)$ is isomorphic to a transitive subgroup of $\Sym_3$, so it is either isomorphic to $\Sym_3$ or to the alternating group $\mathrm{Alt}_3\simeq \ZZ/3\ZZ$, and the statement follows.

	\ref{it:auto--g}: Since $p_1\in S(L)$, there exists an isomorphism $\varphi\colon S_L\iso \PP^2_L$.
	The preimage of the smallest linear subspace of $\PP^2$ containing the image of the three points $p_1,p_2,p_3$ is defined over $K$, and is thus of dimension $2$ by Corollary~\ref{cor: SB contains a point of degree p}. There is then no line containing the three points, and thus we may choose that $\varphi$ sends $p_1,p_2,p_3$ onto $[1:0:0], [0:1:0],[0:0:1]$ respectively.

	Since $g$ permutes the points $p_i$ cyclically, $\varphi\circ g\circ \varphi^{-1}= A_g\circ g$ with $A_g=\left(\begin{smallmatrix}0&0&\lambda \\ \mu &0&0\\0&\nu &0\end{smallmatrix}\right)$ and $\lambda,\mu,\nu\in L^*$.
	Consider the map $\alpha\colon[x:y:z]\mapsto \left [x: \mu y:g(\mu)\nu z \right ]$. Replacing $\varphi$ by $\alpha\circ\varphi$, we may assume that $A_g$ is as in the statement, namely with $\xi=\lambda g^2(\mu)g(\nu)\in L^*$.
	Finally, since $g^3$ is the identity, so is $(\varphi\circ g\circ\varphi^{-1})^3$, and thus $A_g\cdot g(A_g)\cdot g^2(A_g)$ is the identity in $\PGL_3(L)$. This implies that $\xi=g(\xi)=g^2(\xi)$, i.e.~$\xi\in L^g$.

\ref{autopts}: Let $q_1,q_2,q_3\in S(L)$ be the components of $q$. We may assume $p\ne q$. Order the points $q_1,q_2,q_3$ so that the action of $\Gal(L/K)$ on $(p_1,p_2,p_3)$ and $(q_1,q_2,q_3)$ is the same. Suppose first that one point $\varphi(q_j)$ lies on the coordinate lines (the lines containing two of the $\varphi(p_i)$). As $\varphi(q_j)\notin \{\varphi(p_1),\varphi(p_2),\varphi(p_3)\}$, it belongs to exactly one of the three lines and applying $\Gal(L/K)$ we obtain that $\varphi(q_1),\varphi(q_2),\varphi(q_3)$ lie on different coordinate lines. Hence, $\varphi(p_1)$ is on no line passing through two of the points $\varphi(q_i)$, and the same holds for $\varphi(p_2)$ and $\varphi(p_3)$. Exchanging the roles of $p$ and $q$, we may thus assume that  $\varphi(q_1)=[1:a:b]$ for some $a,b\in L^*$. Consider a column vector $v_1\in L^3$ that represents this point and a matrix $M\in \Mat_{3\times 3}(L)$ associated to it, by
\[
v_1=\left(\begin{smallmatrix} 1 \\ a \\ b\end{smallmatrix}\right), \quad M=\left(\begin{smallmatrix} 1 & \xi g(b)  & \xi g^2(a)\\ a& 1 &\xi g^2(b) \\ b & g(a) & 1 \end{smallmatrix}\right)=\left(\begin{smallmatrix} v_1 & v_2 & v_3\end{smallmatrix}\right)
\]
where $v_2=A_g\cdot g(v_1)$ and $v_3=A_g\cdot g(v_2)$. The points $q_2$ and $q_3$ are sent by $\varphi$ onto the classes of $v_2$ and $v_3$ in $\mathbb{P}^2(L)$ respectively. Since $\varphi(q_1),\varphi(q_2),\varphi(q_3)$ are not collinear, we have $M\in \GL_3(L)$. We obtain an element  $\alpha\in \Aut_L(S)$  given by $\varphi^{-1}\circ M\circ \varphi$ that sends $p_i$ onto $q_i$ for each $i\in \{1,2,3\}$ and it remains to see that $\alpha\in \Aut_K(S)$, or equivalently that the action of $M$ on $\mathbb{P}^2$ commutes with $\varphi \Gal(L/K)\varphi^{-1}$. We  compute
\[
M\cdot A_g=\left(\begin{smallmatrix} v_2 & v_3 & \xi v_1\end{smallmatrix}\right)=A_g\cdot g(M),
\]
which implies that $\alpha$ commutes with $g$, as $x\mapsto M\cdot x$ commutes with $\varphi\circ g\circ  \varphi^{-1}=A_g\circ g$.
If $\Gal(L/K)=\langle g\rangle=\mathrm{Alt}_3$, we are done. The remaining case is when $\Gal(L/K)\simeq \Sym_3$. There is then $h\in \Gal(L/K)$ that fixes $p_1,q_1$, exchanges $p_2$ and $p_3$, and exchanges $q_2$ and $q_3$.

We obtain $\varphi\circ h\circ \varphi^{-1}= A_h\circ h$ with $A_h=\left(\begin{smallmatrix}1&0&0 \\ 0 &0&\mu_2 \\0&\mu_3 &0\end{smallmatrix}\right)$ and $\mu_2,\mu_3\in L^*$. As $\varphi\circ h\circ \varphi^{-1}$ fixes $\varphi(q_1)$ and permutes $\varphi(q_2)$ and $\varphi(q_3)$, there are $\lambda_1,\lambda_2,\lambda_3\in L^*$ such that  \[A_h\cdot  h(v_1)=\lambda_1 v_1,\quad  A_h\cdot h(v_2)=\lambda_3 v_3, \quad A_h \cdot h(v_3)=\lambda_2 v_2\] for $i\in \{1,2,3\}$. The first coordinate of $v_1$ being $1$, we get $\lambda_1=1$. Similarly, the second coordinate of $v_2$ and the third coordinate of $v_3$ are equal to $1$, so $\lambda_2=\mu_2$ and $\lambda_3=\mu_3$. This gives
\[
M\cdot A_h=\left(\begin{smallmatrix} v_1 & \mu_3v_3 & \mu_2v_2 \end{smallmatrix}\right) =A_h\cdot h(M),
\]
which implies that $\alpha$ commutes with $h$, as $x\mapsto M\cdot x$ commutes with $\varphi\circ h \circ \varphi^{-1}=A_h\circ h$.
\end{proof}

\begin{lem}\label{lem:SBdegree3}
	Let $K$ be a perfect field, $L/K$ be a Galois extension of degree $3$ and let $g$ be a generator of the group $\ \Gal(L/K)$. We denote by
	$
	N=N_{L/K}(L^*)=\big \{a g(a) g^2(a)\mid a\in L^*\big \}
	$
	the norm group of this extension.
	\begin{enumerate}
		\item\label{SB31}
		For each element $\xi \in K^*$, there is a unique Severi-Brauer surface $S_\xi$ over $K$, such that there is an isomorphism $\varphi\colon S_L\iso \mathbb{P}^2_L$ satisfying
		\[
		\varphi\circ g\circ \varphi^{-1}\colon [x:y:z]\mapsto [\xi g(z) :g(x):g(y)].
		\]
		\item\label{SB32}
		Every surface $S$ defined over $K$ that is $L$-isomorphic to $\PP^2_L$ is $K$-isomorphic to $S_\xi$, for some $\xi\in K^*$.
		\item\label{SB33}
		For each $\xi\in K^*$, the surface $S_\xi$ has a $K$-point if and only if $\xi\in N$.
		\item\label{SB34}
		Given $\xi,\xi'\in K^*$, the surfaces $S_\xi$ and $S_{\xi'}$ are $K$-isomorphic if and only if $\xi/\xi' \in N$.
	\end{enumerate}
\end{lem}
\begin{proof}
	Recall \cite[Chapter III]{Serre} that the set of $K$-surfaces $S$ that are $L$-isomorphic to $\PP^2_L$, up to $K$-isomorphism, is parametrised by the cohomology set $\Cohom^1(\Gal(L/K),\PGL_3(L))$. As $\Gal(L/K)$ is cyclic and generated by $g$, a cocycle corresponds to an element $\nu\in \PGL_3(L)$ such that $\nu g(\nu) g^2(\nu)=\mathrm{id}$ and two cocycles $\nu,\nu'$ are equivalent if and only if there exists $\alpha\in \PGL_3(L)$ such that $\nu'=\alpha^{-1}  \nu  g(\alpha)$. Moreover, each cocycle $\nu$ corresponds to a surface $S$ defined over $K$, with an isomorphism $\varphi\colon S_L\iso \mathbb{P}^2_L$ such that $\varphi\circ g\circ \varphi^{-1}=\nu \circ g$.

	\ref{SB31}: For each $\xi \in K^*$, the matrix
	$
	\nu_\xi=\left(\begin{smallmatrix} 0 & 0 & \xi  \\  1 & 0 & 0 \\ 0 & 1 & 0\end{smallmatrix}\right)\in \PGL_3(K)
	$
	yields a cocyle, as $g(\nu)=\nu$ and $\nu^3=\mathrm{id}$. This proves the claim.

	\ref{SB34}: If $\xi/\xi'\in N$, there is $a\in L^*$ such that $\xi/\xi'=a g(a) g^2(a)$. We then choose
	$
	\alpha=\left(\begin{smallmatrix} 1 & 0 & 0  \\  0 & a & 0 \\ 0 & 0 & a g(a)\end{smallmatrix}\right)\in\PGL_3(L)
	$
	and obtain that $\alpha \nu_\xi=\nu_{\xi'} g(\alpha)$, as elements of $\PGL_3(L)$, proving that the two cocycles $\nu_\xi$ and $\nu_{\xi'}$ are equivalent. Conversely, we suppose that $\nu_\xi$ and $\nu_{\xi'}$ are equivalent. Consider the matrices
	$
	R_\xi=\left(\begin{smallmatrix} 0 & 0 & \xi  \\  1 & 0 & 0 \\ 0 & 1 & 0\end{smallmatrix}\right)\in \GL_3(K)$, $R_{\xi'}=\left(\begin{smallmatrix} 0 & 0 & \xi'  \\  1 & 0 & 0 \\ 0 & 1 & 0\end{smallmatrix}\right)\in \GL_3(K),
	$
	and obtain a matrix $A\in \GL_3(L)$ and an element $a\in L^*$ such that
	$a  R_{\xi'}= A^{-1} R_{\xi} g(A)$. Hence, we find
	\[A = a^{-1}R_\xi g(A)R_{\xi'}^{-1}.\]
	Applying $g$ and $g^2$ gives $g(A) ={g(a)^{-1}}R_\xi g^2(A)R_{\xi'}^{-1}$ and $g^2(A)={g^2(a)^{-1}}R_\xi AR_{\xi'}^{-1}$. We then replace successively $g(A)$ and $g^2(A)$ and obtain
	\[
	ag(a)g^2(a)A = g(a)g^2(a)R_\xi g(A)R_{\xi'}^{-1} = g^2(a)R_\xi^2 g^2(A)R_{\xi'}^{-2} =R_\xi^3 AR_{\xi'}^{-3} = \xi\xi'^{-1} A.
	\]
	This yields $\xi/\xi'=a g(a) g^2(a)$, as required.

	\ref{SB33}: If $\xi\in N$, then $S_\xi$ is $K$-isomorphic to $S_1$, by \ref{SB34}. There is thus an isomorphism  $\varphi\colon S_L\iso \mathbb{P}^2_L$ such that $\varphi\circ g\circ \varphi^{-1}\colon[x:y:z]\mapsto [g(z) :g(x):g(y)].$ Hence, $[1:1:1]$ is fixed by $\varphi\circ g\circ \varphi^{-1}$, and thus $\varphi^{-1}([1:1:1])$ is a $L$-point of $S$ fixed by $g$, which is thus a $K$-point of $S$. Conversely, let us assume that $S_\xi(K)\not=\varnothing$. Take an isomorphism  $\varphi\colon S_L\iso \mathbb{P}^2_L$ such that $\varphi\circ g\circ \varphi^{-1}\colon[x:y:z]\mapsto [\xi g(z) :g(x):g(y)].$ The image of a $K$-point is a point $[a:1:b]$ with $a,b\in L^*$, with $[a:1:b]=[\xi g(b) : g(a):1]$, so $b=\frac{1}{g(a)}$ and $ag(a)=\xi g(b)=\xi \frac{1}{g^2(a)}$, which implies that $\xi=a g(a)g^2(a)\in N$.

	\ref{SB32}: Let $S$ be a surface defined over $K$ that is $L$-isomorphic to $\PP^2_L$. If $S(K)\not=\varnothing$, then it is already $K$-isomorphic to $\PP^2$, and is then isomorphic to $S_1$ by \ref{SB33} and by Ch\^{a}telet's theorem. If $S(K)=\varnothing$, we take an $L$-point $p\in S(L)$, and consider its orbit under $\Gal(L/K)$, which corresponds to three $L$-points $p_1=p$, $p_2=g(p)$ and $p_3=g^2(p)$, and thus to a $3$-point whose splitting field is $L$. By Lemma~\ref{lem:auto} and \ref{SB31}, $S$ is $K$-isomorphic to $S_\xi$.
\end{proof}

In fact, every non-trivial Severi-Brauer surface is as in Lemma~\ref{lem:SBdegree3}:

\begin{lem}\label{lem:CyclicExtension}
	Let $S$ be a non-trivial Severi-Brauer surface over a perfect field $K$. Then there exists a Galois extension $L/K$ of degree $3$ such that $S_L\iso \PP^2_L$.
\end{lem}
\begin{proof}
	Let $A$ be the non-trivial central simple algebra of degree $3$ over $K$ that corresponds to $S$. Then $A$ is a cyclic division algebra, see Section \ref{sec: first definitions}. By definition of a cyclic algebra, there exists a maximal subfield $L$ of $A$ such that $L/K$ is a cyclic Galois extension of degree~$3$. By \cite[Proposition 4.5.3]{Gille} or \cite[Lemma 11.8.3]{Stacks}, the field $L$ is a splitting field for $A$.
\end{proof}

\begin{rem}\label{rem: field conditions}

	Let $K$ be a perfect field. By Lemma~\ref{lem:CyclicExtension} and Lemma~\ref{lem:SBdegree3}\ref{SB33}, non-trivial Severi-Brauer surfaces over $K$ exist if and only if there is a Galois extension $L/K$ of degree $3$ whose norm map is not surjective.

	\begin{enumerate}
		\item If $\charact(K)=3$, there is no non-trivial Severi-Brauer surface over $K$. Indeed, since $K$ is perfect, the cube map $\phi\colon K\to K$, $\phi(t)=t^3$, is surjective. Hence the norm map $N_{L/K}\colon L^*\to K^*$ is also surjective, as for every $s\in K^*$ one can write $s=t^3$ for some $t\in K^*$ and then $N_{L/K}(t)=t^3=s$.
		\item If $K$ has the $C_1$ property, then the norm of every finite separable Galois extension is surjective \cite[X \S7, Propositions 10 and 11]{SerreLocalFields}. Hence, every Severi-Brauer variety over $K$ is trivial. In particular, if a morphism $X\to B$ between complex varieties has a generic fibre being a non-trivial Severi-Brauer surface over $\CC(B)$, then we have $\dim(B)\geqslant 2$ and thus $\dim(X)\geqslant 4$ (recall that algebraic function fields of dimension $1$ over algebraically closed fields are $C_1$ by Tsen's theorem).
	\end{enumerate}
\end{rem}

Note that Lemma~\ref{lem:SBdegree3} indeed allows to construct non-trivial Severi-Brauer surfaces for many fields: for instance, by \cite{Stern} the norm $N\colon L^*\to K^*$ is not surjective for any Galois extension $L/K$ of number fields. We can also construct a non-trivial Severi-Brauer surfaces over $\CC(t_1,\ldots,t_n)$ for each $n\geqslant 2$, which will be useful in the sequel, as it will be the generic fibre of a Mori fibre space over $\mathbb{P}^n_{\CC}$.

\begin{cor}\label{cor: non-trivial SB over function field}
	Let $n\geqslant 2$ and let $K=\CC(t_1,\ldots,t_n)=\CC(\mathbb{A}^n)$ be the complex rational function field in $n$ variables. Choosing $\lambda=t_1$ and $L=K[\sqrt[3]{\lambda}]$, the  field extension $L/K$ is Galois of degree~$3$, with Galois group generated by $g\colon \sqrt[3]{\lambda}\mapsto \zeta\sqrt[3]{\lambda}.$ Moreover, $\xi=t_2$ is not the norm of an element of $L$. There is then a unique Severi-Brauer surface $S$ over $K$ which admits an isomorphism $\varphi\colon S_L\iso \PP^2_L$ such that \[\varphi\circ g\circ \varphi^{-1}\colon[x:y:z]\mapsto [g(z)t_2 :g(x):g(y)],\] and this Severi-Brauer is a non-trivial one: $S(K)=\varnothing$.
\end{cor}
\begin{proof}
	As $\lambda$ is not a cube in $K$ and $K$ contains the root of unity $\zeta=e^{2\mathbf{i} \pi/3}$, the field extension $L/K$ is Galois, with Galois group generated by $g.$ The existence of $S$ is then given by Lemma~\ref{lem:SBdegree3}\ref{SB31}. We now prove that $\xi=t_2$  is not a norm of $L/K$, which is equivalent to $S(K)=\varnothing$ by Lemma~\ref{lem:SBdegree3}\ref{SB32}. Assume for contradiction that there exists $a\in L^*$ such that $t_2=ag(a)g^2(a)$. Writing $u_1= \sqrt[3]{\lambda}= \sqrt[3]{t_1}$, the field $L=\CC(u_1,t_2,\ldots,t_n)$ is a function field in $n$ variables. We then write $a=b/c$ for two polynomials $b,c\in\CC[u_1,t_2,\ldots,t_n]\setminus\{0\}$ and obtain that $t_2cg(c)g^2(c)=bg(b)g^2(b)$ is of degree $3\deg(b)=3\deg(c)+1$, which is impossible.
\end{proof}

In Section \ref{subsec: twisted} we already saw that some basic properties of Severi-Brauer surfaces imply that there is always a point of degree $3$. In fact, we have a stronger statement:

\begin{prop}\label{L3atleast2}
	Let $S$ be a non-trivial Severi-Brauer surface over a perfect field $K$. There are two points $p,q$ of degree $3$ of $S$ with different splitting fields. In particular, there is no $\alpha\in \Aut_K(S)$ such that $\alpha(p)=q$.
\end{prop}

\begin{proof}
	By Lemma~\ref{lem:CyclicExtension}, there exists a Galois extension $L/K$ of degree $3$ and an isomorphism  $\varphi\colon S_L\iso \mathbb{P}^2_L$. Let $g\in \Gal(L/K)$ be a generator. By Lemma~\ref{lem:SBdegree3}, we may choose $\varphi$ such that $\varphi\circ g\circ \varphi^{-1}\colon[x:y:z]\mapsto [\xi g(z) :g(x):g(y)],$ for some $\xi \in K^*$, which is not the norm of an element of $L$, and, in particular, not a cube in $K$. The polynomial $X^3-\xi$ then has no roots in $K$. We now prove that $X^3-\xi$ also has no root in $L$. Indeed, otherwise one of the root being in $L$ implies that all root belong to $L$, as $L/K$ is Galois. Hence, $L$ is the splitting field of $X^3-\xi$ over $K$ and $g$ acts cyclically on the roots $\xi_1,\xi_2,\xi_3$ of the polynomial $X^3-\xi$ and $\xi$ is then the norm of $\xi_1$, a contradiction.

	Let $F$ be the splitting field of $X^3-\xi$ over $K$, that is, $F=K[\xi_1,\xi_2,\xi_3]$. The field extensions $K\subseteq L\cap F \subseteq L$ gives $3=[L:K]=[L:L\cap F]\cdot [L\cap F:K]$, so either $L\cap F=L$ or $L\cap F=K$. A $\xi_1\not\in L$, we find that $L\cap F=K$. By  \cite[Chapter VI, Theorem~1.14]{lang-algebra},  $L'/K$ is Galois, where $L'=LF$ is the splitting field of $X^3-\xi$ over $L$, and we have a group isomorphism
	\[ \begin{array}{ccc}
		\Gal(L'/K) &\iso & \Gal(L/K)\times \Gal(F/K)\\
		\sigma & \mapsto & (\sigma|_L,\sigma|_F).
	\end{array}\]
	As $\Gal(F/K)$ is either of order $3$ or $6$ and acts transitively on the three roots $\xi_1$, $\xi_2$, $\xi_3$, we can extend $g$ to an element of $\Gal(L'/K)$, that we again write $g$, such that $g(\xi_1)=\xi_2$, $g(\xi_2)=\xi_3$ and $g(\xi_3)=\xi_1$.

	Consider the three points $p_i=\left [\xi_i:1:\xi_i^{-1}\right ]\in \PP^2(F)$, with $i\in \{1,2,3\}$. Note that $\{p_1,p_2,p_3\}\subseteq \PP^2(F)$ form one $\Gal(F/K)$-orbit.
	We calculate that $\varphi\circ g\circ \varphi^{-1}\colon[x:y:z]\mapsto [\xi g(z) :g(x):g(y)]$ satisfies
	\[
	\varphi\circ g\circ \varphi^{-1}(p_i)=\left [\frac{\xi}{g(\xi_i)}:g(\xi_i):1\right ]=\left [g(\xi_i):1:\frac{1}{g(\xi_i)} \right ]=g(p_i),
	\]
	using $\xi=(g(\xi_i))^3$, so the set $\{p_1,p_2,p_3\}$ is a $\varphi\circ g\circ \varphi^{-1}$-orbit. As $g$ and $\Gal(F/K)$ generate $\Gal(L'/K)$ (using the isomorphism above),  the set $\{q_1,q_2,q_3\}$ with $q_i=\varphi^{-1}(p_i)\in S(L')$ for $i=1,2,3$ forms a $\Gal(L'/K)$-orbit on $S$, and hence a $3$-point $q$ on $S$, since it is $\Gal(F/K)$-invariant and also invariant by $g$. The splitting field of $q$ is an intermediate field $K\subseteq L''\subseteq L'$, of degree $3$ or $6$, different from the splitting field $L$ of $p$, because $p_i\notin \PP^2(L)$ for each $i\in\{1,2,3\}$.
\end{proof}

\subsection{Non-trivial Severi-Brauer surfaces with no points of degree 6}\label{subsec: points 6}

Corollary \ref{cor: SB contains a point of degree p} implies that every non-trivial Severi-Brauer surface contains a point of degree $3$. In this section, we explain that it may fail to have points of degree $6$; in particular, for some fields $K$ the set $\mathcal{P}_6$ appearing in Theorem \ref{theoremSB} is empty. Since we were not able to find the corresponding example in the literature, we give it in the Proposition \ref{prop: no points of deg 6} below.

For each field $\mathrm{k}$, we denote as usual by $\mathrm{k}(\!(x)\!)$ the field of formal Laurent series. By the Newton-Puiseux theorem \cite[Chapter IV, Proposition 8]{SerreLocalFields}, the field $\bigcup_{a\in \NN} \mathrm{k}(\!(x^{1/a})\!)$ of Puiseux series is an algebraic closure of $\mathrm{k}(\!(x)\!)$ if $\mathrm{k}$ is algebraically closed of characteristic zero. Our field $K$ will be some subfield of the algebraic closure of $\CC(\!(x)\!)(\!(y)\!)$.

\begin{prop}\label{prop: no points of deg 6}
	The field $K=\bigcup_{a\in \NN\setminus 3\NN} \CC(\!(x^{1/a})\!)(\!(y^{1/b})\!)$ satisfies the following properties:
\begin{enumerate}
 \item\label{LKseries}
Each finite Galois extension $L/K$ is of degree $3^n$ for some $n\geqslant 1$.
\item\label{SBseries}
There is a non-trivial Severi-Brauer surface defined over $K$ that does not have any point of degree $6$.
\end{enumerate}
\end{prop}

\begin{proof}
Applying the Newton-Puiseux theorem to $\mathrm{k}=\CC$, we get the algebraically closed field $F=\bigcup_{a\in \NN} \CC(\!(x^{1/a})\!)$. Doing this again with a new variable $y$, an algebraic closure of $F(\!(y)\!)$ is the field
\[
\bigcup_{a,b\in \NN} \CC(\!(x^{1/a})\!)(\!(y^{1/b})\!)=\bigcup_{a\in \NN} \CC(\!(x^{1/a})\!)(\!(y^{1/a})\!),
\]
which is then also an algebraic closure of $\CC(\!(x)\!)(\!(y)\!)$ and of $K$.

\ref{LKseries}: Let  $L/K$ be a finite Galois extension. We prove that $[L:K]$ is a power of $3$ by proving that the order of every element  $\sigma\in \mathrm{Gal}(L/K)$ is a power of $3$.  We take $a\in \NN$ such that $L\subseteq \CC(\!(x^{1/a})\!)(\!(y^{1/a})\!)$ (for instance take $a$ to be the common multiple of all denominators of exponents in a generating set of $L/K$), and extend $\sigma$ to an element $\widehat\sigma\in  \mathrm{Gal}(\overline{K}/K)$. Choose the largest $m\geqslant 1$ such that $3^m$ divides $a$ and obtain that $(x^{1/a})^{3^m}$ belongs to $K$. Hence, $\widehat\sigma(x^{1/a})=\xi x^{1/a}$ for some $\xi\in \CC$ with $\xi^{3^m}=1$. This proves that $\widehat\sigma^{3^m}$ fixes $x^{1/a}$. Similarly, $\widehat\sigma^{3^m}$ fixes $y^{1/a}$, so the order of $\sigma$ divides $3^m$.

\ref{SBseries}: We define $L=K(\sqrt[3]{x})$. Note that $L/K$ is a Galois extension with $\Gal(L/K)\simeq\ZZ/3\ZZ$ generated by $\sqrt[3]{x}\mapsto\zeta\sqrt[3]{x}$, where $\zeta$ is a primitive third root of unity. As $y$ is not in the image of the norm map $N_{L/K}\colon L^*\to K^*$, there exists a non-trivial Severi-Brauer $S$ surface over $K$ by Lemma \ref{lem:SBdegree3}. This surface does not contain any point of degree $6$ as its splitting field would be a Galois extension $L'/K$ of degree divisible by $6$, which does not exist since there is no finite Galois extension of $K$ of even degree by~\ref{LKseries}. \end{proof}

\section{Sarkisov links between Severi-Brauer surfaces}

\subsection{Sarkisov links and relations}\label{subsec: Sarkisov} In this section, we will be interested in the structure of the group $\Bir(S)$, where $S$ is a Severi-Brauer surface. Following \cite{LamyZimmermann} and \cite{BLZ}, we denote by $\BirMori(S)$ the groupoid of all birational maps between Mori fibre spaces (see Definition \ref{def: Mfs}) birational to $S$ over $K$. The generators of this groupoid $\BirMori(S)$ were described in 1970 by Felix Weinstein in his Master's thesis \cite{Weinstein89}, republished recently as \cite{Weinstein22}. Today, Weinstein's results can be viewed as a part of {\it Sarkisov program} for 2-dimensional Mori fibre spaces, which we now recall. In our exposition, we follow \cite{LamySchneider}.

\begin{mydef}\label{def: Mfs}
Let $K$ be a perfect field, let $S$ be a smooth projective surface defined over~$K$, and let $r \geqslant 1$ be an integer.

\begin{enumerate}
\item
A morphism $S\to B$ (often written $S/B$) defined over $K$ is called  a \emph{rank~$r$ fibration} if it is a surjective morphism with geometrically connected fibres, $-K_S$ relatively ample, $\rk\Pic_{K}(S/B)=r$, and if $B$ is a point or a smooth curve.
\item
A rank $1$ fibration $S\to B$ is said to be a \emph{Mori fibre space}.\end{enumerate}
\end{mydef}

In Definition~\ref{def: Mfs}, there are only two possibilities for $B$:
either $B = \pt$ is a point and $S$ is a del Pezzo surface of Picard rank $r$ over $K$, or $S \to B$ is a conic bundle over a smooth curve $B$ with $r-1$ $\Gal(\overline{K}/K)$-orbits of singular fibres.

\begin{mydef}
Let $K$ be a perfect field and let $S/ B$ and $S'/ B'$ be two fibrations of ranks $r$ and $r'$, respectively.

\begin{enumerate}
\item
We say that $S/B$ \emph{factorizes through} $S'/B'$, or that $S'/B'$ is \emph{dominated} by $S/B$, if there is a birational morphism $\varphi\colon S\to S'$ and a morphism $\beta\colon B' \to B$, both defined over $K$, such that the following diagram commutes (if it exists then $r\geqslant r'$):
\[
\begin{tikzcd}[row sep=small]
	S \ar[rrr] \ar[dr,"\varphi",swap] &&& B \\
	& S' \ar[r] & B' \ar[ur,"\beta",swap]
\end{tikzcd}
\]
\item
In the above definition, $\varphi$ is called \emph{isomorphism of rank $r$ fibrations} if $r=r'$ and $S\to S'$, $B'\to B$ are both isomorphisms, and is moreover also called \emph{isomorphism of Mori fibre spaces} if $r=r'=1$.
\end{enumerate}
\end{mydef}
By additivity of the relative Picard rank, we get the following easy result:

\begin{lem}[{\cite[Lemma 2.1]{LamySchneider}}]\label{l:2_domination_type}
	Assume that $S/B$ is a rank $r+1$ fibration that factorizes through a rank~$r$ fibration $S'/B'$. Then one of the following holds:
	\begin{enumerate}
		\item \label{domination:blowup} Either $B \simeq B'$, and $S \to S'$ is the blow-up of a $d$-point $p\in S'$;
		\item \label{domination:base} Or $B$ is a curve, $B'$ is a point, and $S \to S'$ is an isomorphism.
	\end{enumerate}
\end{lem}

The \emph{piece} \label{def:piece} of a rank $r$ fibration $S/B$ is the $(r-1)$-dimensional combinatorial polytope constructed as follows:
Each rank $d$ fibration dominated by $S/B$ is a $(d-1)$-dimensional face, and for each pair of faces $S_i/B_i$, $i=1,2$, $S_2/B_2$ lies in $S_1/B_1$ if and only if $S_1/B_1$ dominates $S_2/B_2$. We write $(r-1)$-piece when we want to emphasize the dimension of the piece (associated to a rank $r$ fibration).

We now consider more closely the case of $1$-pieces, which turn out to encode Sarkisov links. Let $Z/B$ be a rank 2 fibration.
As a consequence of the two-rays game, there are exactly two rank~$1$ fibrations $X_1/B_1$ and $X_2/B_2$ dominated by $Z/B$ (see \cite[Lemma~3.7]{BLZ}). This gives a birational map $\chi\colon X_1 \rat X_2$ that is called \emph{Sarkisov link}. Note that this one is unique up to replacing a link with the inverse and up to composing at the source and target by automorphisms of Mori fibre spaces. Using Lemma~\ref{l:2_domination_type} we can distinguish between four types of Sarkisov links, depending if the domination of $X_1/ B_1$ (resp. $X_2/B_2$) by $Z/B$ is a blow-up or a  base change. This description of Sarkisov links is entirely classical and can be found in Iskovskikh's seminal paper \cite[\S 2.2]{Isk1996}. We recall it in Figure~\ref{fig:SarkisovTypesSurfaces} (one can later compare with the higher dimensional case described in Figure~\ref{fig:SarkisovTypes}).
\begin{figure}[ht]
\[
{
\def\arraystretch{2.2}
\begin{array}{cc}
\begin{tikzcd}[ampersand replacement=\&,column sep=1.3cm,row sep=0.14cm]
\ar[dd,"\rm div",swap] Z \ar[rr,"\simeq"] \&\& X_2 \ar[dd,"\rm fib"] \\ \\
X_1 \ar[uurr,"\chi",dashed,swap] \ar[dr,"\rm fib",swap] \&  \& B_2 \ar[dl] \\
\& B_1 = B=\pt \&
\end{tikzcd}
&
\begin{tikzcd}[ampersand replacement=\&,column sep=.8cm,row sep=0.14cm]
  \&Z\ar[ddl,"\rm div",swap]\ar[ddr,"\rm div"]\&  \\ \\
X_1 \ar[rr,"\chi",dashed,swap] \ar[dr,"\rm fib",swap] \&  \& X_2 \ar[dl,"\rm fib"] \\
\& B_1 = B = B_2 \&
\end{tikzcd}
\\
\I & \II
\\
\begin{tikzcd}[ampersand replacement=\&,column sep=1.3cm,row sep=0.14cm]
X_1 \ar[ddrr,"\chi",dashed,swap] \ar[dd,"\rm fib",swap]  \ar[rr,"\simeq"] \&\& Z \ar[dd,"\rm div"] \\ \\
B_1 \ar[dr] \& \& X_2 \ar[dl,"\rm fib"] \\
\& B = B_2 \&
\end{tikzcd}
&
\begin{tikzcd}[ampersand replacement=\&,column sep=1.3cm,row sep=0.14cm]
X_1=Z \ar[rr,"\chi=\mathrm{id}_Z",swap] \ar[dd,"\rm fib",swap]  \&\& Z=X_2 \ar[dd,"\rm fib"] \\ \\
B_1 \ar[dr] \& \& B_2 \ar[dl] \\
\& B \&
\end{tikzcd}
\\
\III & \IV \vspace{-0.2cm}
\end{array}
}
\]
\caption{The four types of Sarkisov links for dimension $2$ over perfect fields. Here ``div'' and ``fib'' mean divisorial contraction and fibration, respectively.}
\label{fig:SarkisovTypesSurfaces}
\end{figure}
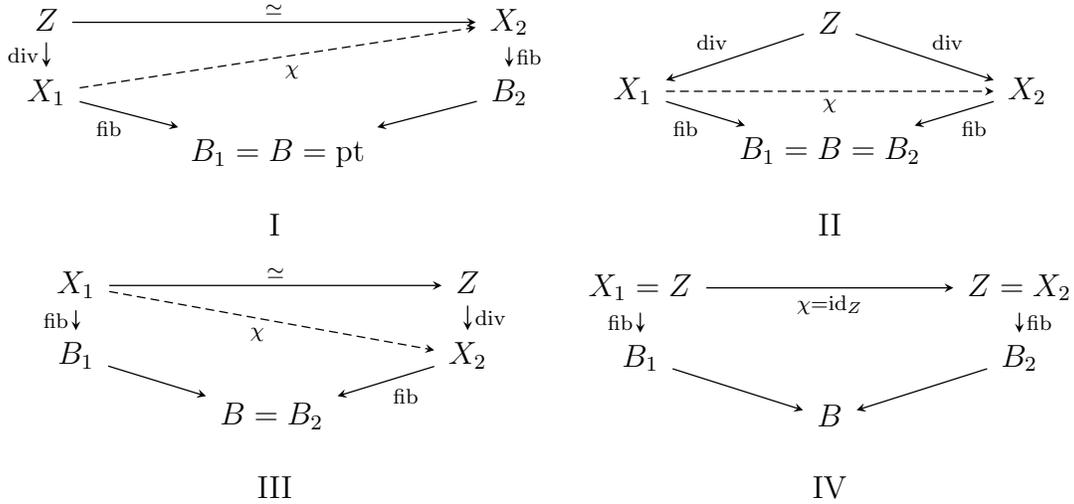

As we shall use only Sarkisov links of type II over a point, in Figure \ref{fig:link2} we give on the left the classical diagrammatic notation as in Figure~\ref{fig:SarkisovTypesSurfaces} and on the right its counterpart from \cite{LamySchneider} showing the $1$-piece with the rank~$2$ fibration in the centre of the edge, dominating the two rank 1 fibrations $S/ B$ and $S'/B'$ on the left and right vertices. An arrow $\stackrel{_d}{\to}$ refers to the blow-up of a $d$-point.

\begin{figure}[ht]
	\begin{tabular}{ccc}
		 \text{Diagram} & \text{Piece} \\
		\\
		\begin{tikzcd}[column sep=1.3cm,row sep=0.14cm]
			& Z\ar[dl,"d",swap]\ar[dr,"d'"] \\ X_1\ar[dr] && X_2\ar[dl] \\ & \pt
		\end{tikzcd} &
		\begin{tikzcd}
			X_1/\pt & Z/\pt\ar[l,-,ultra thick,"d",swap]\ar[r,-,ultra thick,"d'"] & X_2/\pt
		\end{tikzcd}
	\end{tabular}
	\caption{Sarkisov links of type II over a point}
	\label{fig:link2}
\end{figure}
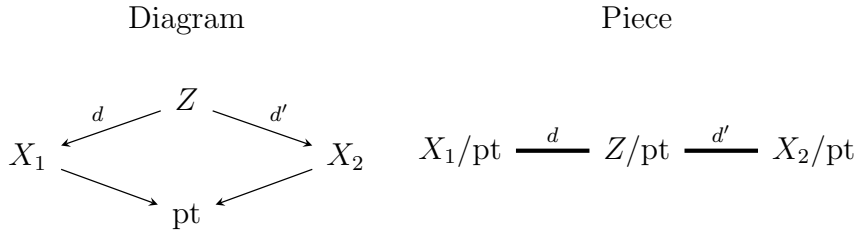

In \cite[Proposition 2.6]{LamyZimmermann} it was shown that a $2$-piece (or rather its geometric realisation) is homeomorphic to a disk, and following \cite{LamySchneider} we will draw it as a regular polygon. The boundary of the polygon corresponds to a sequence of Sarkisov links whose product is an automorphism: we say that the piece encodes an \emph{elementary relation} between Sarkisov links.
Since the composition of a Sarkisov link with an automorphism of Mori fibre space is a Sarkisov link, and the inverse of a Sarkisov link is a Sarkisov link, we obtain \emph{trivial relations} given by $\chi\circ \chi^{-1}=\mathrm{id}$ and $\alpha'\circ \chi\circ \alpha=\chi'$, where $\chi,\chi'$ are Sarkisov links and $\alpha,\alpha'$ are automorphisms of Mori fibre spaces. The Sarkisov program can be formulated in the following form:

\begin{thm}[{\cite[Theorem 3.1]{LamyZimmermann}, see also \cite{Isk1996}} ] \label{p:Sarkisov}
	Let $K$ be a perfect field, and let $S$ be a smooth projective surface defined over $K$, that is birational to a Mori fibre space. The groupoid $\BirMori(S)$ is generated by Sarkisov links and isomorphisms of Mori fibre spaces.
	Any relation between Sarkisov links is generated by trivial relations and elementary relations.
\end{thm}

\subsection{$d$-links on Severi-Brauer surfaces}

Unlike the projective plane, a non-trivial Severi-Brauer surface $S$ admits very few types of Sarkisov links. First, by the Lang-Nishimura lemma (see e.g.~\cite[Proposition IV.6.2]{Kollar_rational}), a non-trivial Severi-Brauer surface $S$ is not birational to any smooth projective surface which admits a $K$-rational point. Moreover, as the degree of any closed point of $S$ is a multiple of $3$, one directly obtains exactly two types of Sarkisov links starting from $S$. We first describe these in Example~\ref{Ex36links} (also see \cite[Theorem 2.6, ii, (c) and (e)]{Isk1996}), then prove that these are the only ones in Lemma~\ref{lem:ClassificationRankr}. We will then describe elementary relations in Lemma~\ref{lem:OnlyOnePiece}.

\begin{mydef}\label{Def36links}
	Let $K$ be a perfect field and $S_1, S_2$ be two del Pezzo surfaces with Picard rank $1$. Let $d_1,d_2\geqslant 1$ be integers.
	We say that a link $\varphi\colon S_1\dashrightarrow S_2$ of type \II\ over $K$ is a \emph{$(d_1,d_2)$-link} (or \emph{$d$-link} if $d=d_1=d_2$) if the base-point of $\varphi$ has degree $d_1$ and the base-point of $\varphi^{-1}$ has degree $d_2$.
\end{mydef}

\begin{ex}\label{Ex36links}
Let $S$ be a non-trivial Severi-Brauer surface over a perfect field $K$ and let $p\in S$ be a closed point of degree $d\in \{3,6\}$. Then the blow-up  $T\to S$ at $p$ is a del Pezzo surface by \cite[Lemma 2.8]{Shramov1} or \cite[Lemma 2.1]{Weinstein22}, which is of degree $9-d$, and thus $T/\pt$ is a rank $2$ fibration. Fix an isomorphism  $S_{\overline{K}}\simeq\PP_{\overline{K}}^2$ that sends $p_{\overline{K}}$ onto $d$-points $p_1,\ldots,p_d\in \PP^2(\overline{K})$.
\begin{enumerate}
	\item If $d=3$, we may blow-down the proper transforms of the three lines on $\PP_{\overline{K}}^2$ passing through the pairs of points of $\{p_i\}$, and get a birational map $\tau_p\colon  S\dashrightarrow S'$ where $S'$ is another non-trivial Severi-Brauer surface. Hence, $\tau_p$ is a $3$-link, and all $3$-links starting from $S$ are of this form.
	\item If $d=6$, we may blow-down the proper transforms of the six conics on $\PP_{\overline{K}}^2$ passing through the quintuples of points of $\{p_i\}$, and get a birational map $\eta_p\colon S\dashrightarrow S'$ where $S'$ is another non-trivial Severi-Brauer surface. Hence, $\eta_p$ is a $6$-link, and all $6$-links starting from $S$ are of this form.
\end{enumerate}
\end{ex}

\begin{lem}\label{lem:splittingfieldinverse}
Let $S$ be a non-trivial Severi-Brauer surface over a perfect field $K$ and let $p\in S$ be a point of degree $d\in \{3,6\}$. Let $\tau\colon S\dasharrow S'$ be a $d$-link associated to $p$, as in Example~$\ref{Ex36links}$. The inverse of $\tau$ is a birational map $\tau^{-1}\colon S'\dasharrow S$, that is a $d$-link associated to a point $q\in S'$ of degree $d$, and the splitting fields of $p$ and $q$ are equal.
\end{lem}
\begin{proof} Over the algebraic closure $\overline{K}$ of $K$, the map $\tau$ becomes a birational map $\PP^2_{\overline{K}}\dasharrow \PP^2_{\overline{K}}$, that is the blow-up of $d$ points $p_1,\ldots,p_d$ followed by the contraction of the strict transforms of $d$ irreducible curves $C_1,\ldots,C_d\subseteq \PP^2_{\overline{K}}$. If $d=3$ (respectively $d=6$) the curve $C_i$ is the line (respectively the conic) passing through the $d-1$ points $\{p_1,\ldots,p_d\}\setminus \{p_i\}$. The action of $\Gal(\overline{K}/K)$ on the curves $C_1,\ldots,C_d$ is then exactly the same as the action on the points $p_1,\ldots,p_d$. The images $q_1,\ldots,q_d$ of the curves contracted correspond then to a point $q\in S'$ of degree $d$, and the splitting fields are the same as the action of $\Gal(\overline{K}/K)$ on the $p_i$'s and the $q_i$'s is the same.
\end{proof}

\begin{lem}[see also {\cite[Corollaries 2.4 and 2.5]{Shramov1}}]\label{lem:ClassificationRankr}
	Let $S$ be a non-trivial Severi-Brauer surface defined over a perfect field $K$.
	Let $T/B$ be a rank $r$ fibration that dominates $S/\pt$. Then $B=\pt$ and one of the following holds:
	\begin{enumerate}
		\item\label{rank1} $r=1$ and $T$ is a non-trivial Severi-Brauer surface;
		\item\label{rank2} $r=2$ and $T=T_d$ is a del Pezzo surface of degree $d\in\{3,6\}$, encoding a $(9-d)$-link;
		\item\label{rank3} $r=3$ and $T=T_3$ is a del Pezzo surface of degree $3$ obtained by the blow-up $T_3\to T_6\to S$ of two points of degree $3$.
	\end{enumerate}
\end{lem}
\begin{proof}
Let $T/B$ be a rank $r$ fibration dominating $S/\pt$. If $r=1$, then $T/B$ is isomorphic to $S/\pt$, as in \ref{rank1}. Suppose that $r\geqslant 2$. Hence $B=\pt$ and $T$ is therefore a del Pezzo surface, which implies that $T\to S$ is given by the blow-up of $r-1$ points $p$ on $S$ whose sum of degrees it at most $8$. Since $3$ divides the degree of each point on $S$ (Corollary~\ref{cor: SB contains a point of degree p}), the degree of each point blown-up is either $3$ or $6$. Moreover, there are  exactly three possibilities: one point of degree $3$, one point of degree $6$, or two points of degree $3$. This gives Cases~\ref{rank2}-\ref{rank3}.
\end{proof}

Curiously enough, to describe the target of $3$-links and $6$-links completely, we need a short detour to the theory of algebras.

\begin{mydef}
	Let $X$ be a non-trivial Severi-Brauer variety of dimension $n-1$ and let $A$ be the central simple algebra of dimension $n^2$ corresponding to $X$. Define the opposite algebra $A^{\rm op}$ as the ring with the same underlying set and addition as $A$, but with the opposite multiplication, i.e.~$a\times_{A^{\rm op}} b=b\times_A a$ (in fact, this is the inverse of $A$ in the Brauer group $\Br(K)$). This is again a central simple algebra of dimension $n^2$ over $K$. The corresponding {\it opposite Severi-Brauer variety} will be denoted $X^{\rm op}$.
\end{mydef}

\begin{rem}\label{rem: opposite SB}
	It is possible that $X\simeq\ X^{\rm op}$. Indeed, let $A$ be a non-trivial central simple $K$-algebra of degree 2, i.e. the generalised quaternion algebra \cite[Chapter 1]{Gille}, so $X$ is a conic with no $K$-points. Then the period of $A$ (i.e.~its order in the Brauer group $\Br(K)$) is~ $2$, see e.g. \cite[Proposition 4.5.13]{Gille}, hence $A\simeq A^{\rm op}$ and $X\simeq X^{\rm op}$. By contrast, if $A$ is a central simple $K$-algebra of degree 3 and $X$ is a non-trivial Severi-Brauer surface, then the period of $A$ is $3$ by the same Proposition. Thus $X$ is not isomorphic to $X^{\rm op}$.
\end{rem}

\begin{lem}\label{LinkSop}
	Let $S$ be a non-trivial Severi-Brauer surface over a perfect field.

	\begin{enumerate}
	\item
	If $\chi\colon S\dashrightarrow S'$ is a Sarkisov link, then $\chi$ is a $3$-link or a $6$-link, and $S'\simeq S^\mathrm{op}$.
	\item
	Any Mori fibre space birational to $S$ is isomorphic to $S$ or to $S^{\mathrm{op}}$. In particular, $S$ is not birational to a conic bundle.
	\end{enumerate}
\end{lem}
\begin{proof}
	This is a consequence of Lemma~\ref{lem:ClassificationRankr} and \cite[Lemma 1, Lemma 4]{Weinstein22}.
\end{proof}

As a consequence of Lemmas~\ref{lem:ClassificationRankr} and \ref{LinkSop}, the $3$-links and $6$-links, together with $\Aut(S)$ and $\Aut(S^{\rm op})$, generate $\Bir(S)$:

\begin{thm}[{\cite[Theorem 3,1]{Weinstein22}, \cite[Theorem 2.10]{Shramov1}}]\label{thm: Weinstein}
	Let $S$ be a non-trivial Severi-Brauer surface over $K$ and
	$S'$ be a del Pezzo surface over $K$ with $\rk\Pic(S')=1$. Assume there is a birational map $\varphi: S\dashrightarrow S'$. Then either $S\simeq S'$ or $S'\simeq S^{\rm op}$. Moreover,
	\[
	\varphi=\varphi_1\circ\cdots\circ\varphi_n,
	\]
	where each $\varphi_i$ is either an isomorphism, or a $d$-link $S\dasharrow S^{\rm op}$ or $S^{\rm op}\dasharrow S$ with $d\in\{3,6\}$.
\end{thm}

\begin{lem}\label{Lem:SxiOp}
Let $K$ be a perfect field, $L/K$ be a Galois extension of degree $3$ and let $g$ be a generator of the group $\ \Gal(L/K)$. For each $\xi\in K^*$, we denote by $S_\xi$ the Severi-Brauer surface over $K$ given by Lemma~$\ref{lem:SBdegree3}$. Then, we have $S_\xi^{\rm op}\simeq S_{\xi^{-1}}$, for each $\xi \in K^*$.
\end{lem}
\begin{proof}We fix some $\xi\in K^*$. For each $\theta\in \{\xi,\frac{1}{\xi}\}$, we take an isomorphism  $\varphi_\theta\colon (S_\theta)_L\iso \mathbb{P}^2_L$ such that $\varphi_\theta \circ g\circ \varphi_\theta^{-1}=[x:y:z]\mapsto [\theta g(z) :g(x):g(y)].$ We consider the birational involution $\sigma\in \Bir_L(\mathbb{P}^2)$ given by $[x:y:z]\mapsto [yz:xz:xy]$. It has three base-points, which are the coordinates points, and satisfies
\[
\sigma \circ (\varphi_\xi \circ g\circ (\varphi_\xi)^{-1})=(\varphi_{\xi^{-1}} \circ g\circ (\varphi_{\xi^{-1}})^{-1}) \circ \sigma.
\]
Composing at the left with $(\varphi_{\xi^{-1}})^{-1}$ and at the right with $\varphi_\xi$, we find that the birational map $(\varphi_{\xi^{-1}})^{-1}\circ \sigma \circ \varphi_{\xi}\colon S_\xi\dasharrow S_{\xi^{-1}}$ commutes with $g$ and is therefore defined over $K$. It is a $3$-link (see Definition~\ref{Def36links}), and thus goes from a Severi-Brauer surface to its opposite (Lemma~\ref{LinkSop}), i.e.~$(S_\xi)^{\rm op}\simeq S_{\xi^{-1}}$.
\end{proof}

\subsection{Elementary relations between Sarkisov $d$-links}
We will consider the following equivalence relation between Sarkisov links:
\begin{mydef}\label{def:equiLinkDim2}
	Let $\chi\colon X_1\dashrightarrow X_2$, $\chi'\colon X_1'\dashrightarrow X_2'$ be two Sarkisov links between del Pezzo surfaces of Picard rank $1$ over a perfect field $K$. We say that $\chi, \chi'$ are \emph{equivalent} if there is a commutative diagram
	\begin{equation*}
		\xymatrix@R=6pt@C=20pt{
			X_1\ar[d]_{\alpha}\ar@{-->}[rr]^{\chi} && X_2\ar[d]^{\beta}\\
			X_1'\ar@{-->}[rr]^{\chi'} && X_2'
		}
	\end{equation*}
	for some isomorphisms $\alpha\colon X_1\iso X_1',\beta\colon X_2\iso X_2'$.
\end{mydef}

The following result is a simple observation. We leave the proof to the reader.

\begin{lem}\label{lem:equivalencLinksSurfaceBasePt}
Let $\chi\colon X_1\dashrightarrow X_2$, $\chi'\colon X_1'\dashrightarrow X_2'$ be two Sarkisov links between del Pezzo surfaces of Picard rank $1$ over a perfect field $K$. The following conditions are equivalent:
\begin{enumerate}
\item
$\chi$ and $\chi'$ are equivalent;
\item
there exists an isomorphism $\alpha\colon X_1\iso X_1'$ that sends the base-point of $\chi$ onto the base-point of $\chi'$;
\item
there exists an isomorphism $\beta\colon X_2\iso X_2'$ that sends the base-point of $\chi^{-1}$ onto the base-point of $\chi'^{-1}$.
\end{enumerate}
\end{lem}

\begin{rem}\label{rem: link is not equivalent to inverse}
	Lemma~\ref{lem:equivalencLinksSurfaceBasePt} and Remark~\ref{rem: opposite SB} imply that a $d$-link $\chi$ between Severi-Brauer surfaces is never equivalent to its inverse $\chi^{-1}$. This observation will be used in the proof of Theorem \ref{theoremSB}.
\end{rem}
\begin{lem}\label{lem:OnlyOnePiece}
	Let  $X$ be rank $3$ fibration birational to a non-trivial Severi-Brauer surface defined over a perfect field $K$. Then $X$ is a del Pezzo surface of degree $3$ and the piece associated to $X$ is a hexagon as in Figure~$\ref{fig:relation}$, associated to a relation that is equal, modulo trivial relations, to
	\[\chi_6\circ \chi_5\circ\chi_4\circ\chi_3\circ\chi_2\circ\chi_1=\id_S,\]
	where $\chi_{2i-1}\colon S\rat S^{\rm op}$ and $\chi_{2i}\colon S^{\rm op}\rat S$ are $3$-links, for $i\in \{1,2,3\}$, and $S$, $S^{\rm op}$ are the two isomorphism classes of non-trivial Severi-Brauer surfaces birational to $X$. Moreover, $\chi_1,\chi_3,\chi_5$ are equivalent and $\chi_2,\chi_4,\chi_6$ are equivalent.
\end{lem}
\begin{proof}
As $X$ is a rank $3$ fibration, it corresponds to a composition of Sarkisov links $\chi_n\circ \cdots \chi_1=\mathrm{id}_{S}$, where $S$ is a Mori fibre space birational to $S$. By Theorem~\ref{thm: Weinstein}, each of the Mori fibre spaces arising in this composition is isomorphic to $S$ or $S^{\rm op}$, and each of the Sarkisov links is a $3$-link or a $6$-link. Applying isomorphisms, which corresponds to using trivial relations, we may assume that each Mori fibre space is equal to $S$ or to $S^{\rm op}$. By Theorem~\ref{thm: Weinstein}, $n$ is even and each link $\chi_i$ goes from $S$ to $S^{\rm op}$ when $i$ is odd and from $S^{\rm op}$ to $S$ when $i$ is even.
By Lemma~\ref{lem:ClassificationRankr}\ref{rank3}, the surface $X$ is a del Pezzo surface of degree $3$, obtained as the blow-up of two points of degree $3$ on $S$. Denote them by $p$ and $p'$. We use, as in the definition of the links, an isomorphism $S_{\overline{K}}\iso \mathbb{P}^2_{\overline{K}}$ that sends $p$ onto three points $p_1,p_2,p_3\in \mathbb{P}^2(\overline{K})$ and $p'$ onto $p_1',p_2',p_3'\in \mathbb{P}^2(\overline{K})$.

There are exactly $27$ $\overline{K}$-lines on the cubic surface $X_{\overline{K}}$, corresponding to $27$ $(-1)$-curves defined over $\overline{K}$. We may describe them using the birational morphism $X_{\overline{K}}\to \mathbb{P}_{\overline{K}}^2$, that is the blow-up of $p_1,p_2,p_3,p_1',p_2',p_3'$. These $27$ lines contain $6$ orbits of three disjoint lines. We denote by $E_i$ and $E_i'$ the exceptional curves contracted to $p_i$ and $p_i'$ respectively, for $i\in \{1,2,3\}$, denote by $l_i$ (resp.~$l_i'$) the proper transform of the lines through $p_j,p_k$ (resp.~$p_j',p_k'$), where $\{i,j,k\}=\{1,2,3\}$. We then denote by $C_i$ (resp.~$C_i'$) the proper transform of the conic passing through all points except $p_i$ (resp.~$p_i'$). The six disjoint triples are thus given by $E_i,E_i',l_i,l_i',C_i,C_i'$. The nine other lines on $X_{\overline{K}}$ are given by lines through one $p_i$ and one $p_j'$ but no $3$ are disjoint and Galois invariant.

Each of the six disjoint triples is disjoint from exactly two other triples. This gives then $6$ birational morphisms from $X$. We obtain $n=6$ and get the diagram of Figure~\ref{fig:relation}. Each of the set of $3$ disjoint lines appears three times in the diagram. Moreover, the splitting field of a $3$-link and its inverse are the same (Lemma~\ref{lem:splittingfieldinverse}). Hence, the splitting field of the warm links (contractions of $E_i,l_i,C_i$) is the same and the splitting field of the cold links (contractions of $E_i',l_i',C_i'$). This, together with Lemma~\ref{lem:auto} and Lemma~\ref{lem:equivalencLinksSurfaceBasePt} implies that $\chi_1,\chi_3,\chi_5$ are equivalent and $\chi_2,\chi_4,\chi_6$ are equivalent.
\end{proof}
\begin{figure}[ht]
	\includegraphics[width=7cm]{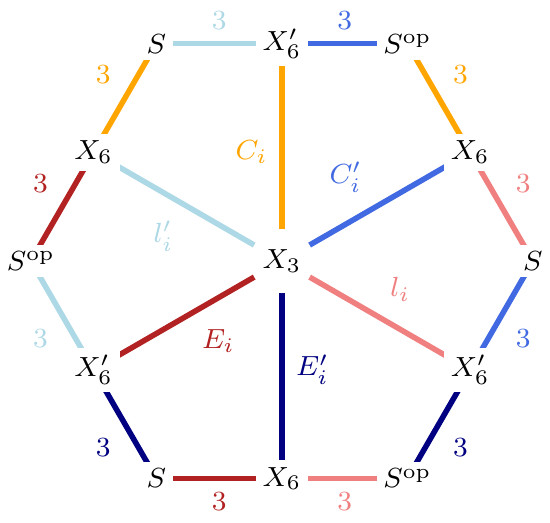}
	\caption{Relation of Lemma~\ref{lem:OnlyOnePiece}. The centre is a del Pezzo surface $X_3$ of degree $3$ of Picard rank $3$. Each segment denotes the blow-up of a point of degree~$3$.}
	\label{fig:relation}
\end{figure}

\subsection{Proof of Theorem \ref{theoremSB}}

Before proving our first main Theorem \ref{theoremSB}, let us provide a slightly more general statement (that will be also used in higher dimensions later).

\begin{thm}\label{thm:SBgroupoidhomo}
	Let $S$ be a non-trivial Severi-Brauer surface over a perfect field $K$. For $d\in \{3,6\}$, denote by $\mathcal{E}_d$ a set of representatives of the equivalence classes of $d$-links $S\dashrightarrow S^{\rm{op}}$.
	Then there is a groupoid homomorphism
	\[
	\Psi\colon \BirMori_K(S)\to \bigoplus\limits_{\mathcal{E}_3} \ZZ/3\ZZ \ast \left ( \bigast_{\mathcal{E}_6} \mathbb{Z} \right)
	\]
	that is defined as follows:
	We send each isomorphism between Mori fibre spaces onto the trivial word. Each $d$-link $\chi\colon S\dashrightarrow S^{\rm op}$ is sent onto the generator $1_{\chi'}$ of the copy of $\ZZ/3\ZZ$ or $\ZZ$, where $\chi'\in\mathcal{E}_d$ is equivalent to $\chi$, and each link $\chi\colon S^\text{op}\dashrightarrow S$ is sent onto $-1_{\chi'}$ where $\chi'\in\mathcal{E}_d$ is equivalent to $\chi^{-1}\colon S\dashrightarrow S^\text{op}$.
\end{thm}
\begin{proof}
	By definition of the homomorphism, trivial relations are sent onto the trivial word.
	By Lemma~\ref{lem:OnlyOnePiece}, all elementary relations are of the form \[\chi_6\circ \chi_5\circ\chi_4\circ\chi_3\circ\chi_2\circ\chi_1=\id_S,\]
	where $\chi_{2i-1}\colon S\rat S^{\rm op}$ and $\chi_{2i}\colon S^{\rm op}\rat S$ are $3$-links, for $i\in \{1,2,3\}$ and where $\chi_1,\chi_3,\chi_5$ are equivalent and $\chi_2,\chi_4,\chi_6$ are equivalent.
	So the elementary relation is sent onto $3_{\chi_1}-3_{\chi_2^{-1}}$, which is trivial.
	We thus obtain a groupoid homomorphism by Theorem~\ref{p:Sarkisov}.
\end{proof}

We are now in position to prove Theorem~\ref{theoremSB}.

\begin{proof}[Proof of Theorem~\ref{theoremSB}]
First, observe that, for $d\in \{3,6\}$, there is a one-to-one correspondence between $\mathcal{E}_d$ and $\mathcal{P}_d$, that sends a link $\chi\colon S\dasharrow S^{\rm op}$ onto its base-point.  This follows from Lemma~\ref{lem:equivalencLinksSurfaceBasePt} and the fact that every $d$-point gives a $d$-link, see Example~\ref{Ex36links}. Theorem~\ref{thm:SBgroupoidhomo} induces a groupoid homomorphism
\begin{equation}\label{eq: BirMori hom Theorem A}
\BirMori(S)\to  \bigoplus\limits_{\mathcal{P}_3} \ZZ/3\ZZ \ast \left (\bigast_{\mathcal{P}_6} \mathbb{Z} \right )
\end{equation}
that sends every isomorphism between Mori fibre spaces onto the identity, a $d$-link is sent onto the generator $1_p$ of the copy of $\ZZ/3\ZZ$ or $\ZZ$ corresponding to the class of the base-point $p\in \mathcal{P}_d$ if the link goes from $S$ to $S^\mathrm{op}$, and onto $-1_p$ if the link goes from $S^\mathrm{op}$ to $S$ (up to isomorphism).
Let us restrict the homomorphism \eqref{eq: BirMori hom Theorem A} to $\Bir(S)$. We then fix an element $p\in \mathcal{P}_3$ (which always exists by Proposition~\ref{prop:KollarSB}), and consider the projection
\[
\bigoplus\limits_{\mathcal{P}_3} \ZZ/3\ZZ \ast \left (\bigast_{\mathcal{P}_6} \mathbb{Z} \right )\to  \bigoplus\limits_{\mathcal{P}_3\setminus \{p\}} \ZZ/3\ZZ \ast \left (\bigast_{\mathcal{P}_6} \mathbb{Z} \right).
\]
The set $\mathcal{P}_3\setminus \{p\}$ is not empty (Lemma~\ref{L3atleast2}). We now prove that the obtained group homomorphism
$$
\Psi\colon \Bir(S)\to  \bigoplus\limits_{\mathcal{P}_3\setminus\{p\}} \ZZ/3\ZZ \ast \left ( \bigast\limits_{\mathcal{P}_6} \mathbb{Z} \right )
$$
is surjective. For this, we simply take an element $q\in \mathcal{P}_d$ for $d\in \{3,6\}$ with $q\not=p$ and choose a $d$-link $\varphi'\colon S\dasharrow S^{\rm op}$ whose base-point is $q$, and a $3$-link $\varphi\colon S\dasharrow S^{\rm op}$ whose base-point is $p$. Then, the image $\Psi(\varphi^{-1}\circ \varphi')$ corresponds to the generator associated to $q$.
\end{proof}

We would like to close this section with the following open problem:

\begin{Question}
	Let $X$ be an algebraic variety. Recall that $\varphi\in\Bir(X)$ is called {\it regularisable} if there is a birational map $\psi\colon X\dashrightarrow Y$ such that $\psi\circ\varphi\circ\psi^{-1}\in\Aut(Y)$ and $Y$ is a smooth projective variety. Let $S$ be a non-trivial Severi-Brauer surface over a perfect field $K$. Is the group $\Bir_K(S)$ always  generated by regularisable elements?
\end{Question}

\section{Explicit birational models of Severi-Brauer surfaces}\label{sec: examples}

\subsection{A singular cubic surface as a birational model} The goal of this section is to give an explicit and simple birational model of a non-trivial Severi-Brauer surface over a perfect field $K$ which contains a primitive third root of unity. It will be used in the proofs of our main results. Note that singular cubics over a non-closed field were studied by Th. Skolem in \cite{Skolem}. The below examples are part of Case 3 in \cite[\S 3]{Skolem}.

\begin{prop}\label{prop: birational cubic model}
Let $K$ be a perfect field that contains a primitive third root of unity $\zeta\in K^*$, let $\lambda\in K^*$ be such that $L/K$ is Galois of degree $3$, with $L=K[\sqrt[3]{\lambda}]$ and Galois group generated by
$g\colon  \sqrt[3]{\lambda}\mapsto \zeta \sqrt[3]{\lambda}.$
For each $\xi\in K^*$, the singular cubic surface $X\subseteq \PP^3$ given by the equation
\begin{equation*}
	\xi w^3=\lambda x^{3}+y^{3}+\lambda^{-1}z^3-3 xyz
\end{equation*}
is birational over $K$ to the $($unique$)$ Severi-Brauer surfaces $S_{\xi}$ and $S_{\xi^{-1}}$, admitting isomorphisms $\varphi_\xi\colon (S_{\xi})_L\iso \mathbb{P}^2_L$ and $\varphi_{\xi^{-1}}\colon (S_{\xi^{-1}})_L\iso \mathbb{P}^2_L$ satisfying \[\varphi_\theta\circ g\circ (\varphi_\theta)^{-1}([x:y:z])= [\theta g(z) :g(x):g(y)],\]
for $\theta=\xi$ and $\theta=\xi^{-1}$ $($recall that such surfaces exist by Lemma~$\ref{lem:SBdegree3}\ref{SB31})$.
\end{prop}
\begin{proof}
Define the map $\widetilde{g}\colon [x:y:z]\mapsto [\xi^{-1} g(z) :g(x):g(y)]$. It suffices to find a birational map $\psi\colon X_L\dasharrow \mathbb{P}^2_L$ such that $\psi\circ g\circ \psi^{-1}=\widetilde{g}$.
Indeed, the birational map $\varphi_{\xi^{-1}}^{-1}\circ\psi\colon X_L\dasharrow (S_{\xi^{-1}})_L$ will commute with $g$ and hence be a birational map $X\dasharrow S_{\xi^{-1}}$.
Moreover, writing $\sigma\in \Bir(\PP^2)$ the birational map $[x:y:z]\mapsto [\frac{1}{x}:\frac{1}{y}:\frac{1}{z}]$, we will have $\sigma\circ \psi\circ g\circ (\sigma\circ \psi)^{-1}=[x:y:z]\mapsto [\xi g(z) :g(x):g(y)]$, so the birational map $\varphi_\xi^{-1}\circ\sigma\circ\psi\colon X_L\dasharrow (S_{\xi})_L$ will commute with $g$ and hence be a birational map $X\dasharrow S_{\xi}$.

Let $\lambda_0=\sqrt[3]{\lambda}$, $\lambda_1=\zeta\sqrt[3]{\lambda}$ and $\lambda_2=\zeta^2\sqrt[3]{\lambda}$ be the three cube roots of $\lambda$ in $L$. Then $F=\lambda x^{3}+y^{3}+\lambda^{-1}z^3-3 xyz\in L[x,y,z]$ is factorized as $F=F_0F_1F_2$, with $F_i=\lambda_i x+y+{\lambda_i}^{-1}z\in L[x,y,z]$ for $i\in \{0,1,2\}$.
The equation of $X$ becomes $\xi w^3=F_0F_1F_2$, and is thus singular at the three points where $w=F_i=F_j=0$, namely $p_k=[0:1:\lambda_k:\lambda_k^2]$ for $\{i,j,k\}=\{0,1,2\}$.
Each of these is a double point of $X$ and thus the projection away from $p_2$ yields a birational map $X_L\dasharrow \PP^2_L$, given by
\[[w:x:y:z]\mapsto [w:F_0(x,y,z):F_1(x,y,z)].\] Composing this map with the birational map of $\PP^2$ given by $[1:a:b]\mapsto [1:a:ab]$, we obtain a birational map
\[
\psi\colon X_L\dasharrow \mathbb{P}^2_L, \quad [w:x:y:z]\mapsto [w^2:wF_0(x,y,z):F_0(x,y,z)F_1(x,y,z)].
\]
It remains to check that $\psi=\widetilde{g}\circ \psi\circ g^{-1}$. For $i\in \{0,1\}$, we have $g(F_i(x,y,z))=g(\lambda_i)g(x)+g(y)+g(\lambda_i^{-1})g(z)=F_{i+1}(g(x),g(y),g(z))$. Using this, we get
\[
\widetilde{g}\circ \psi\circ g^{-1} ([w:x:y:z])=[\xi^{-1} F_1(x,y,z)F_2(x,y,z): w^2: wF_1(x,y,z)]=\psi([w:x:y:z]),
\]
which finishes the proof.
\end{proof}

\subsection{A smooth cubic as a birational model}

In this section, we give a smooth birational model of our Severi-Brauer surfaces and, moreover, construct a composition of two $3$-links that is a birational self-map of order $3$.

\begin{prop}\label{ExplicitForm:SB3}
Let $K$ be a perfect field that contains a primitive third root of unity $\zeta\in K^*$, let $\lambda,\mu\in K^*$ be such that $\widehat{L}=K[\sqrt[3]{\lambda},\sqrt[3]{\mu}]$ is Galois of degree $9$ over $K$, with Galois group generated by
\[g\colon  \sqrt[3]{\lambda}\mapsto \zeta \sqrt[3]{\lambda}, \sqrt[3]{\mu}\mapsto \sqrt[3]{\mu} ,\quad h\colon  \sqrt[3]{\lambda}\mapsto  \sqrt[3]{\lambda}, \sqrt[3]{\mu}\mapsto \zeta\sqrt[3]{\mu} .\]
We set $L=K[\sqrt[3]{\lambda}]$, so that $L/K$ is a Galois extension of order $3$, with Galois group
$\langle g\rangle$.
Let $\nu\in K^*$ be such that $\xi=27\lambda\mu+\nu^3\in K^*$ is not a norm in $L/K$, and let $X\subseteq \PP^3$ be the smooth cubic surface given by
\[\xi w^3=\lambda x^{3}+\mu y^{3}+ z^3+\nu xyz.\]
Then, the following hold:
\begin{enumerate}
\item\label{XtoSxideg3}
There is a $K$-birational morphism $\pi\colon X\to S_\xi$, where $S_{\xi}$ is the Severi-Brauer surface admitting an isomorphism $\varphi\colon (S_{\xi})_L\iso \mathbb{P}^2_L$ satisfying $\varphi\circ g\circ \varphi^{-1}([x:y:z])= [\xi g(z) :g(x):g(y)].$
\item\label{order3with2links}
Writing $\rho\in \Aut_K(X)$ the automorphism of order $3$ given by $[w:x:y:z]\mapsto [\zeta w:x:y:z]$, we can choose $\pi$ as above, such that $\pi \circ \rho \circ \pi^{-1}\in \Bir_K(S_{\xi})$ is the composition of a $3$-link with splitting field $K[\sqrt[3]{\lambda}]$ and a $3$-link with splitting field $K[\sqrt[3]{\mu}]$.
\end{enumerate}
\end{prop}

\begin{proof}
	Note that $X\to \PP^2, [w:x:y:z]\mapsto [x:y:z]$ is a triple covering, corresponding to the quotient by $\rho$, which is a cyclic covering ramified over the cubic $\Gamma\subseteq\PP^2$ given by the Hessian form $\lambda x^{3}+\mu y^{3}+ z^3+\nu xyz=0$. As $\mathrm{char}(K)\not=3$ and $\lambda\mu\xi\not=0$, this curve is smooth. There are then exactly nine flexes over the algebraic closure $\overline{K}$ of $K$, corresponding to the intersection of $xyz=0$ with the curve $\Gamma$. The $27$ lines of $X$ are then the preimages of the $9$ tangent lines at these $9$ flexes. The preimage of each such tangent line splits as a union of three lines on $X$ meeting in an Eckardt point.
	For $i\in\ZZ$, write $\lambda_i=\zeta^i\sqrt[3]{\lambda}$ and $\mu_i=\zeta^i\sqrt[3]{\mu}$, hence $\lambda_i^3=\lambda$ and $\mu_i^3=\mu$, and set
	\[
	\begin{array}{lllll}
		A_i&=&w-\frac{y}{3\lambda_i}, \qquad B_i &=& z + \lambda_i x - \frac{\nu y}{3\lambda_i},\\
		C_i&=&w-\frac{x}{3\mu_i}, \qquad D_i &=& z + \mu_i y - \frac{\nu x}{3\mu_i}.
	\end{array}
	\]
	Define the six lines
	\[
	\begin{array}{lllllll}
	E_i&=&\big \{A_i=B_i=0\big \},& i\in\{0,1,2\}, \qquad E_i&=&\big\{C_{i-1}=D_i=0\big \},& i\in\{3,4,5\},
	\end{array}
	\]
	and observe that $E_0,\ldots,E_5$ are six disjoint $(-1)$-curves in $X$ defined over $\widehat{L}$, which correspond to two orbits $\{E_0,E_1,E_2\}$ and $\{E_3,E_4,E_5\}$ of $\Gal(\hat{L}/K)=\langle g,h\rangle$.

	The contraction of the six lines $E_0,\ldots,E_5$ gives then a birational morphism $\pi\colon X\to S$, defined over $K$, where $S$ is a Severi-Brauer surface. The image of $\{E_0,E_1,E_2\}$ is a $3$-point $p\in S$ with splitting field $L=K[\sqrt[3]{\lambda}]$ and the image of $\{E_3,E_4,E_5\}$ is a $3$-point $q\in S$ with splitting field $K[\sqrt[3]{\mu}]$. There is then an isomorphism $\varphi\colon S_L\iso \PP^2_L$, that sends $\pi(E_0),\pi(E_1),\pi(E_2)\in S(L)$ onto $[1:0:0]$, $[0:1:0]$, $[0:0:1]$ respectively. We will compute $\varphi\circ \pi$ and show that we can choose $\varphi$ such that $\varphi\circ g\circ \varphi^{-1}([x:y:z])= [\xi g(z) :g(x):g(y)],$ which will prove that $S$ is the surface $S_\xi$ of Lemma~\ref{lem:SBdegree3}, associated to $\xi\in K$ and $g\in \mathrm{Gal}(L/K)$. For this, we consider some other lines on $S$, all defined over $\widehat{L}$:
	\[
	\begin{array}{llllllll}
	\ell_{01}&=&\big \{A_1=B_0=0\big \}, \qquad
	\mathcal{C}_0&=&\big \{A_1=B_2=0\big\}, \qquad
	\mathcal{C}_1&=&\big \{A_2=B_0=0\big \}.
	\end{array}
	\]
	We observe that these are three lines of $X$, which intersect the curves $E_0,\ldots,E_5$ as the following table describes.
	\[\begin{array}{|c|c|c|c|c|c|c|}
	\hline
	 &E_0 & E_1 & E_2 & E_3 & E_4 & E_5\\
	\hline
	\ell_{01} & 1 & 1 & 0 & 0 &0 & 0 \\
	\hline
	\mathcal{C}_0 & 0 & 1 & 1 & 1& 1& 1\\
	\hline
	\mathcal{C}_1 & 1 & 0 & 1 & 1& 1& 1\\
	\hline
	\end{array}\]

	In particular, $\ell_{01}$ is the strict transform (via $\varphi\circ \pi\colon X_L\to \PP^2_L$)  of the line $\{z=0\}$ of $\PP^2_L$ passing through $[1:0:0]$, $[0:1:0]$, and $\mathcal{C}_0$ (respectively $\mathcal{C}_1$) is the strict transform of the conic passing through all points blown-up except $[1:0:0]$ (respectively except $[0:1:0]$). The pull-back by $\varphi\circ \pi$ of a general line is then equivalent to $-2K_X-\ell_{01}-\mathcal{C}_0-\mathcal{C}_1$. This implies that $\varphi\circ\pi$ can be described, on the open subset $X\setminus (\ell_{01}\cup\mathcal{C}_0\cup\mathcal{C}_1)$, by the linear system of quadrics of $\PP^3$ containing the three lines $\ell_{01}$, $\mathcal{C}_0$ and $\mathcal{C}_1$.
	The polynomials
	$f_0=\xi A_1 A_2,f_1=B_0B_2,f_2=A_1B_0\in L[w,x,y,z]$ form a basis of the linear system of polynomials of degree $2$ that are zero on $\ell_{01}$, $\mathcal{C}_0$ and $\mathcal{C}_1$. Moreover, $f_1$ and $f_2$ vanish along $E_0$ but  $f_0$ does not. Similarly, $f_0$ and $f_1$ vanish along $E_2$ but $f_2$ does not, and $f_0$, $f_2$ vanish on $E_1$ but $f_1$ does not. Hence, we can choose $\varphi$ such that the birational morphism $\varphi\circ\pi\colon X_L\to\PP^2_L$ is given by
	\[[w:x:y:z]\mapsto [f_0(w,x,y,z):f_1(w,x,y,z):f_2(w,x,y,z)].\]
	To obtain \ref{XtoSxideg3}, we then need to check that $\tilde{g}\circ \varphi=\varphi\circ g$, where $\tilde{g}=[x:y:z]\mapsto [\xi g(z):g(x):g(y)]$. This equality is equivalent to $\varphi=\tilde{g}\circ \varphi\circ g^{-1}$.

	Note that
	\[A_0A_1A_2=w^3-\frac{y^3}{27\lambda}, \quad B_0B_1B_2=\lambda x^3-\frac{\nu^3}{27\lambda}y^3+z^3+\nu xyz,\]
	and hence we find that $\xi A_0A_1A_2-B_0B_1B_2=\xi w^3-\lambda x^{3}-\mu y^{3}- z^3-\nu xyz$ is zero on $X$. Using this, and the fact that for each $i\in \ZZ$ we have
	$g\circ A_i=A_{i+1}\circ g$ and $g\circ B_i=B_{i+1}\circ g$, we get for a general point $p\in X$ (where $A_i(p)B_i(p)\not=0$ for all $i$) the equality
	\[\begin{array}{rcl}
	\tilde{g}\circ \varphi\circ g^{-1} (\pi(p))= \tilde{g}\circ \varphi\circ\pi\circ g^{-1} (p)
	&=&[\xi A_2B_1: \xi A_0A_2: B_0B_1]\\
	&=&[\xi A_1A_2B_1: \xi A_1A_0A_2: A_1B_0B_1]\\
	&=&[\xi A_1A_2B_1:  B_0B_1B_2: A_1B_0B_1]\\
	&=&[\xi A_1A_2:B_0B_2:A_1B_0]=\varphi(\pi(p)),\end{array}
	\]
	writing $A_i,B_i$ instead of $A_i(p), B_i(p)$ for simplicity. This proves $S=S_{\xi}$ and we get \ref{XtoSxideg3}.

	We now prove~\ref{order3with2links}. We fix $\pi$ and $\varphi$ as above. As $\rho\in \Aut_K(X)$ is of order $3$, the element $\hat\rho=\pi\circ \rho\circ\pi^{-1}\in \Bir_K(S_{\xi})$ is a birational map of order $3$. We now prove that it is the composition of a $3$-link with splitting field $K[\sqrt[3]{\lambda}]$ and a $3$-link with splitting field $K[\sqrt[3]{\mu}]$.

	As $\pi\colon X\to S_{\xi}$ is the contraction of the six lines $E_0,\ldots,E_5$, the morphism $\pi\circ \rho\colon X\to S_{\xi}$ is the contraction of the six lines $E_0'=\rho^{-1}(E_0),\ldots,E_5'=\rho^{-1}(E_5)$, which are obtained as
	\[\begin{array}{llll}
	E_i'=\{ A_{i+1}= B_i=0\},\, i\in \{0,1,2\},\quad
	E_i'=\{ C_i=D_i=0\},\,i\in \{3,4,5\}.\end{array}\]

	As $E_0'=\ell_{01}$, the set $\{E_0',E_1',E_2'\}$ is a $\Gal(\widehat L/K)$-orbit that is disjoint from $\{E_3,E_4,E_5\}$, and so the contraction of $E_0',E_1',E_2',E_3,E_4,E_5$ gives then a birational morphism $\pi'\colon X\to S'$, defined over $K$, where $S'$ is a Severi-Brauer surface.
	Decomposing the three birational morphisms $\pi$, $\pi'$, $\rho\pi$ into a contraction of one orbit and then another one, we obtain a commutative diagram of birational morphisms defined over $K$:
	\[
	\begin{tikzcd}
	&&X\ar[dl,"{E_3,E_4,E_5}",swap]\ar[dr,"{E_0',E_1',E_2'}"]\\
	&X_6\ar[dl,"{E_0,E_1,E_2}",swap]\ar[dr,"{E_0',E_1',E_2'}",swap]&&X_6'\ar[dl,"{E_3,E_4,E_5}"]\ar[dr,"{E_3',E_4',E_5'}"]\\
	S_{\xi} \ar[rr,"\pi'\pi^{-1}",dashed,swap] & & S'\ar[rr,"\pi\rho\pi'^{-1}",dashed,swap] && S_{\xi}.
	\end{tikzcd}
	\]
	Moreover, $\pi' \pi^{-1}\colon S\dasharrow S'$ and $\pi \rho \pi'^{-1}\colon S'\dasharrow S$ are $3$-links, with splitting fields  $K[\sqrt[3]{\lambda}]$ and  $K[\sqrt[3]{\mu}]$ respectively, giving the result.
\end{proof}

\subsection{Existence of links on Severi-Brauer surfaces}

The following proposition serves to have enough links on a given Severi-Brauer surface. We will use it in the sequel with $K=\CC(t_1,\ldots,t_n)$ and $n\geqslant 2$, and $\lambda=t_1$.

\begin{prop}\label{prop: existence of links over function fields}
Let $K$ be a perfect field that contains a primitive third root of unity $\zeta\in K^*$, let $\lambda\in K^*$ be such that $L/K$ is Galois of degree $3$, with $L=K[\sqrt[3]{\lambda}]$ and Galois group generated by $g\colon \sqrt[3]{\lambda}\mapsto \zeta \sqrt[3]{\lambda}.$ Let $\xi\in K^*$ and let $S_\xi$ be the unique Severi-Brauer surface defined over $K$ that admits an isomorphism $\varphi\colon (S_{\xi})_L\iso \mathbb{P}^2_L$ satisfying $\varphi\circ g\circ \varphi^{-1}([x:y:z])= [\xi g(z) :g(x):g(y)].$ Then one has the following:
\begin{enumerate}
\item\label{exi6}
 For each $\alpha\in K$ that is not a square in $L$, there is a point $p$ of $S$ of degree $6$ whose splitting field is $L[\sqrt{\alpha}]$.
 \item\label{exi3}
 For each $\mu\in K$ such that $\mu$ is not a cube in $L$ but $\xi-27\lambda\mu$ is a cube in $K$, there is an element $\varphi\in \Bir_K(S)$ of order $3$, that is the composition of one $3$-link with splitting field $L$ and one $3$-link with splitting field $K[\sqrt[3]{\mu}]$.
 \end{enumerate}
\end{prop}
\begin{proof}
\ref{exi6}
	Let $\alpha\in K$ be an element that is not a square in $L$. We write $L'=L[\sqrt{\alpha}]$, which is an extension of degree $2$ of $L$. We observe that $L'/K$ is Galois of order $6$, with an abelian Galois group generated by $g,h$, given by
	\[g\colon \sqrt{\alpha}\mapsto \sqrt{\alpha}, \quad \sqrt[3]{\lambda}\mapsto \zeta \sqrt[3]{\lambda}\]
	\[h\colon \sqrt{\alpha}\mapsto -\sqrt{\alpha}, \quad \sqrt[3]{\lambda}\mapsto  \sqrt[3]{\lambda}\]
	Here we have extended $g\in \mathrm{Gal}(L/K)$ to an element of $\mathrm{Gal}(L'/K)$ and will keep the same letter.

	We now consider the two points $[0:1:\sqrt{\alpha}], [0:1:-\sqrt{\alpha}]\in \mathbb{P}^2(L')$. As $\varphi$ is defined over $L$ and $h\in \mathrm{Gal}(L'/L)$, we find $\varphi\circ h\circ \varphi^{-1}=h$.
	The action of  $\varphi\circ g\circ \varphi^{-1}$ gives us $6$ points, that form an orbit of $6$ points of $\mathbb{P}^2(L')$ under the action of the group of order $6$ generated by $\varphi\circ g\circ \varphi^{-1}$ and $\varphi\circ h\circ \varphi^{-1}=h$. The preimage by $\varphi^{-1}$ is then a point on $S$ of degree $6$. As $L'/K$ is of order $6$ and the splitting field lies in $L'$, we find that $L'$ is the splitting field of the degree $6$ point.

\ref{exi3}:
	Let $\mu\in K$  be such that $\mu$ is not a cube in $L$ but $\xi-27\lambda\mu$ is a cube in $K$. We can then find $\nu\in K$ such that $\xi=27\lambda\mu +\nu^3$. We write $L'=L[\sqrt[3]{\mu}]$, which is an extension of degree $3$ of $L$ and observe that $L'/K$ is Galois of order $9$, with an abelian Galois group generated by $g,h$, given by
	\[g\colon \sqrt[3]{\alpha}\mapsto \sqrt[3]{\mu}, \quad \sqrt[3]{\lambda}\mapsto \zeta \sqrt[3]{\lambda},\]
	\[h\colon \sqrt[3]{\alpha}\mapsto \zeta\sqrt[3]{\mu}, \quad \sqrt[3]{\lambda}\mapsto  \sqrt[3]{\lambda}.\]
	As before, we have extended the action of $g$.  Proposition~\ref{ExplicitForm:SB3} then yields the existence of an element $\varphi\in \Bir_K(S)$ of order $3$ that the composition of one $3$-link with splitting field $L$ and one $3$-link with splitting field $K[\sqrt[3]{\mu}]$.
\end{proof}

\section{Higher-dimensional Sarkisov program}

\subsection{Tools from the Sarkisov program}

From now on, all varieties are defined over the complex numbers, unless stated otherwise. To construct the group homomorphism of Theorem \ref{thm: free product}, we shall need a higher-dimensional version of the Sarkisov theory introduced in Section \ref{subsec: Sarkisov}. For the proofs we refer to \cite{BLZ}, here we only introduce necessary definitions. We often write $X/B$ for a morphism $X\to B$, and denote by $\Bir(X/B)$ the subgroup of $\Bir(X)$ of maps defined over $B$. In particular, the group $\Bir(X/B)$ is naturally isomorphic to the group $\Bir(S)$ of the generic fibre $S$.

Let $r\geqslant 1$ be an integer. The notion of a rank $r$ fibration from Section \ref{subsec: Sarkisov} generalizes to higher dimensions. We omit technical details here and refer to \cite[Definition 3.1]{BLZ}. The essential thing is that a rank $r$ fibration is a morphism $\eta\colon X\to B$, $\dim X > \dim B \geqslant 0$, such that $X/B$ is a $\QQ$-factorial terminal Mori dream space with $\rho(X/B) = r$. If $\eta\colon X\to B$ is a surjective morphism between normal varieties, then $X/B$ is a rank $1$ fibration if and only if $X/B$ is a terminal Mori fibre space \cite[Lemma 3.3]{BLZ}.

We say that a rank~$r$ fibration $X/B$ \emph{factorises through} a rank~$r'$ fibration $X'/B'$, or that \emph{$X'/B'$ is dominated by $X/B$}, if the fibrations $X/B$ and $X'/B'$ fit in a commutative diagram
\[
\begin{tikzcd}[link]
	X \ar[rrr] \ar[dr,dashed] &&& B \\
	& X' \ar[r] & B' \ar[ur]
\end{tikzcd}
\]
where $X \rat X'$ is a birational contraction, and $B' \to B$ is a morphism with connected fibres. Note that $r \geqslant r'$.

The notion of a rank $2$ fibration corresponds to the notion of a higher-dimensional Sarkisov link. Any rank~$2$ fibration $Y/B$ factorises through exactly two rank~$1$ fibrations $X_1/B_1, X_2/B_2$ (up to isomorphisms), which both fit into a diagram
\[
\begin{tikzcd}[link]
	 \ar[dd]  & Y \ar[l,dotted] \ar[r,dotted] &  \ar[dd] \\ \\
	\ar[dr] && \ar[dl] \\
	& B &
\end{tikzcd}
\]
where the top dotted arrows are sequences of log-flips, and the other four arrows are morphisms of relative Picard rank~$1$.  In the above diagram, there are two possibilities for the sequence of two morphisms on each side of the diagram: either the first arrow is already a Mori fibre space, or it is a divisorial contraction and in this case the second arrow is a Mori fibre space.
This gives $4$ possibilities, which correspond to the usual definition of \emph{Sarkisov links of type \I, \II, \III\ and \IV}, as illustrated on Figure~\ref{fig:SarkisovTypes}. The induced birational map $\chi\colon X_1\rat X_2$ is called a {\it Sarkisov link} then.

\begin{figure}[ht]
\[
{
\def\arraystretch{2.2}
\begin{array}{cc}
\begin{tikzcd}[ampersand replacement=\&,column sep=1.3cm,row sep=0.16cm]
\ar[dd,"\rm div",swap]  \ar[rr,dotted,-] \&\& X_2 \ar[dd,"\rm fib"] \\ \\
X_1 \ar[uurr,"\chi",dashed,swap] \ar[dr,"\rm fib",swap] \&  \& B_2 \ar[dl] \\
\& B_1 = B \&
\end{tikzcd}
&
\begin{tikzcd}[ampersand replacement=\&,column sep=.8cm,row sep=0.16cm]
\phantom{X}\ar[dd,"\rm div",swap]  \ar[rr,dotted,-] \&\& \ar[dd,"\rm div"] \\ \\
X_1 \ar[rr,"\chi",dashed,swap] \ar[dr,"\rm fib",swap] \&  \& X_2 \ar[dl,"\rm fib"] \\
\& B_1 = B = B_2 \&
\end{tikzcd}
\\
\I & \II
\\
\begin{tikzcd}[ampersand replacement=\&,column sep=1.3cm,row sep=0.16cm]
X_1 \ar[ddrr,"\chi",dashed,swap] \ar[dd,"\rm fib",swap]  \ar[rr,dotted,-] \&\& \ar[dd,"\rm div"] \\ \\
B_1 \ar[dr] \& \& X_2 \ar[dl,"\rm fib"] \\
\& B = B_2 \&
\end{tikzcd}
&
\begin{tikzcd}[ampersand replacement=\&,column sep=1.7cm,row sep=0.16cm]
X_1 \ar[rr,"\chi",dotted,swap] \ar[dd,"\rm fib",swap]  \&\& X_2 \ar[dd,"\rm fib"] \\ \\
B_1 \ar[dr] \& \& B_2 \ar[dl] \\
\& B \&
\end{tikzcd}
\\
\III & \IV
\end{array}
}
\]
\caption{The four types of Sarkisov links.}
\label{fig:SarkisovTypes}
\end{figure}
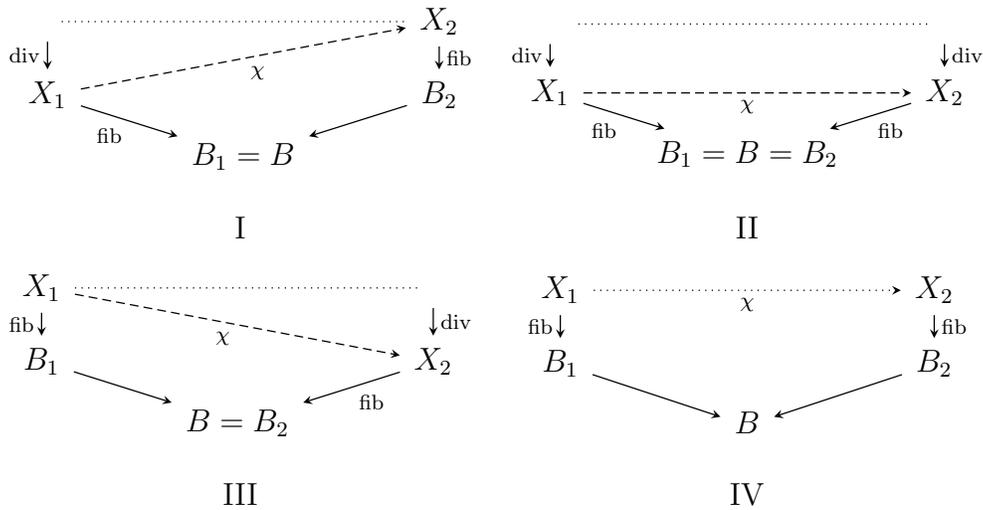

The definition of a rank $r$ fibration implies that every rank~$3$ fibration $T/B$ dominates only finitely many rank~1 and rank~$2$ fibrations. Moreover, there are only finitely many Sarkisov links $\chi_i$ dominated by $T/B$, up to isomorphism of Mori fibre space, and they fit in a relation
	\[
	\chi_t \circ \dots \circ \chi_1 = \id,
	\]
see \cite[Proposition 4.3]{BLZ}. Note that this relation is not uniquely determined by $T/B$, but only up to replacing the links by compositions with isomorphisms (which changes the relation by applying trivial relations, see below for the definition), taking the relation up to cyclic permutation or reversing the order by taking inverse. We say that this is an \emph{elementary relation} between Sarkisov links, coming from the rank $3$ fibration $T/B$.
We call \emph{trivial relations} those of the form $\chi\circ\chi^{-1}=\id$ and $\chi'=\beta\circ\chi\circ\alpha$ for automorphisms $\alpha,\beta$.

Let $X/B$ be a Mori fibre space. We denote by $\BirMori(X)$ the {\it groupoid} of birational maps between Mori fibre spaces birational to $X$. The key result that will be used in the construction of the homomorphism is the following

\begin{thm}[{\cite[Theorem 4.29]{BLZ}, see also \cite[Proposition 3.15]{LamyZimmermann}}]\label{thm: sarkisov}
	Let $X/B$ be a terminal Mori fibre space.
	\begin{enumerate}
		\item\label{sarkisov1} The groupoid $\BirMori(X)$ is generated by Sarkisov links and isomorphisms of Mori fibre spaces.
		\item\label{sarkisov2} Any relation between Sarkisov links in $\BirMori(X)$ is generated by elementary relations and trivial relations.
	\end{enumerate}
\end{thm}

\subsection{The covering genus}\label{sec: covering genus}

Following the ideas of \cite{BLZ} and \cite{BlancYasinsky}, we shall introduce some numerical invariants of Sarkisov links which allows to control the occurrence of links in elementary relations. In \cite{BLZ} such an invariant was the {\it covering gonality} of the centre $\Gamma$ of a link, i.e.
\[
\covgon(\Gamma) = \min \left  \{ c > 0 \
\left |
\parbox{4.2in}{
	\begin{center}
		There is $U\subseteq \Gamma$ open and dense such that each point $x\in U$ is contained in an irreducible curve $C \subseteq \Gamma$ of {\bf gonality} $c$.
	\end{center}
}
\right \}\right..
\]
Then the crucial part of the proof was that the gonality of suitable links can take arbitrary large values. However, in our case, the gonality of the links will always be bounded by~$6$. Therefore, we consider the genus instead of the gonality and define, as in \cite{Lazarsfeld}, the analogous notion of the \emph{covering genus}
\[
\cg(\Gamma) = \min \left  \{ c \geqslant 0 \
\left |
\parbox{4.2in}{
	\begin{center}
		There is  $U\subseteq \Gamma$ open and dense such that each point $x\in U$ is contained in an irreducible curve $C \subseteq \Gamma$ of {\bf genus} $c$.
	\end{center}
}
\right \}\right..
\]
It will be even more convenient for us to use the following
\begin{mydef}[cf. {\cite[Definition 1.4]{Bastianelli}}]\label{def: covering family}
	Let $X$ be an irreducible projective variety. A \emph{covering family of curves of genus $g$} on $X$ is a smooth morphism of algebraic varieties
	\[
	\pi\colon \mathscr{C}\to T
	\] together with a dominant morphism $f\colon\mathscr{C}\to X$ which satisfies the following:
	\begin{enumerate}
		\item For  $t\in T$ general, the fibre $C_t=\pi^{-1}(t)$ is a smooth projective curve with $g(C_t)=g$;
		\item For   $t\in T$ general, the map $f_t\colon C_t\to X$ is birational onto its image.
	\end{enumerate}
\end{mydef}

\begin{rem}\label{rem: properties of covering family}
	Assume we are in the setting of Definition \ref{def: covering family}.

	\begin{enumerate}
		\item Resolving the singularities of $T$, if needed, we may assume that $T$ and $\mathscr{C}$ are smooth. Moreover, restricting to a suitable suvariety of $T$ we may suppose that $\dim\mathscr{C}=\dim X$, so the morphism $f\colon \mathscr{C}\to X$ is generically finite.
		\item  The \emph{covering genus} $\cg(X)$ of $X$ is the least integer $g\geqslant 0$ for which such a covering family exists, as proven in \cite[Lemma 2.23]{BLZ} in the case of gonality, but the proof is exactly the same for the genus. In particular, we recover the notion of \emph{covering genus} given in  \cite{Voisin}, defined with covering families.
	\end{enumerate}
\end{rem}

\begin{ex}\label{ex: covgen}
	Varieties $X$ of covering genus $\cg(X)=0$ (respectively $\covgon(X)=1$) are exactly uniruled varieties. If $S$ is a K3 surface then $\cg(S)=1$, and if $A$ is a very general abelian surface then $\cg(A)=2$ \cite[\S 1]{Lazarsfeld}.  If $X_d$ is a smooth hypersurface of degree $d$ and dimension $n$ in $\PP_\CC^n$, then $\cg(X_d)\sim d^2/2$, see \cite[Remark 3.7]{Lazarsfeld}.
\end{ex}

In the sequel, we will often bound the covering genus of some families of varieties. This also bounds the covering gonality, and is thus stronger than bounding the covering gonality as in \cite{BLZ}.

\begin{lem}\label{KXgeacg}
Let $X$ be a smooth projective irreducible variety, and let $a\geqslant 0$ be an integer such that the union of all irreducible curves $C\subseteq X$ with $C\cdot K_X<a$ is not dense. Then, the covering genus of $X$ satisfies $\cg(X)\geqslant \frac{a}{2}+1$
\end{lem}

\begin{proof}We consider a covering family of genus $g=\cg(X)$ on $X$, given by $\pi\colon \mathscr{C}\to T$ and $f\colon \mathscr{C}\to X$ as in Definition \ref{def: covering family}. By Remark \ref{rem: properties of covering family} we may assume $\mathscr{C}$ and $T$ smooth. By the same remark, restricting to a suitable subvariety $T_0\hookrightarrow T$, we may assume that there is a morphism $f_0\colon\mathscr{C}_0\to X$ which is generically finite, so that we have a commutative diagram
	\[
	\xymatrix@R=6pt@C=20pt{
		\mathscr{C}\ar[dr]_{\pi}\ar[dd]_{f} &&&& \mathscr{C}_0\ar[dl]^{\pi_0}\ar@{_{(}->}[llll]\ar[dd]^{f_0}\\
		& T && T_0\ar@{_{(}->}[ll] &\\
		X\ar@{=}[rrrr] &&&&X.
	}
	\]
We now show that $g\geqslant \frac{a}{2}+1$. Let $t\in T_0$ be a general point, let $C_0=\pi_0^{-1}(t)$ and $C=f(C_0)$, birational to $C_0$ via $f$. As $t$ is a general point and the family is covering, we have $C\cdot K_X \geqslant a$.
Since $f_0\colon\mathscr{C}_0\to X$ is generically finite, we have
\[
K_{\mathscr{C}_0}\equiv f_0^*K_{X}+R,
\]
where $R$ is an effective (ramification) divisor. Note that
\[
g(C_0)=\frac{1}{2}\deg K_{C_0}+1=\frac{1}{2}\left ( C_0\cdot K_{\mathscr{C}_0}\right ) +1
\]
by the adjunction formula. Since $R$ is effective, we have
\[
C_0\cdot K_{\mathscr{C}_0}\geqslant C_0\cdot f_0^* K_{X}=(f_0)_*(C_0)\cdot K_X= C\cdot K_X\geqslant a,
\]
and we conclude that $g=g(C)=g(C_0)\geqslant \frac{a}{2}+1$.
\end{proof}

\begin{lem}\label{lem: cyclic covering covgen}
Let $m,d\geqslant 1$, $n\geqslant 2$ be integers, and let $p\in \CC[x_0,\ldots,x_n]$ be an irreducible homogeneous polynomial of degree $md$, that defines a smooth hypersurface in $\PP^n$, and assume that
$
e=(m-1)d-n-1\geqslant 1.
$
The cyclic $m$-th covering
\[
\Delta=\{[w:x_0:\cdots:x_n]\in \PP(d,1,\ldots,1)\mid w^m=p(x_0,\ldots,x_n)\}
\]
is a smooth variety of dimension $n$, with ample canonical divisor $K_\Delta$, whose covering genus satisfies $\cg(\Delta)\geqslant \frac{e}{2}+1$. Moreover, the $m$-th covering $\Delta\to \PP^n$, $[w:x_0:\cdots:x_n]\mapsto [x_0:\cdots:x_n]$ is given by the linear system $\frac{1}{e}K_\Delta$.
\end{lem}
\begin{proof}
As the hypersurface of $\PP^n$ given by $p=0$ is smooth, so is $\Delta$. The canonical divisor of the toric variety $\PP(d,1,\ldots,1)$ is equal to $-(w=0)-(x_0=0)-\cdots -(x_n=0)$ and is thus equivalent to $-(n+1+d)H$, where $H$ is a hyperplane: the zero set of a linear polynomial in $x_0,\ldots,x_n$. By the adjunction formula, we find that $K_\Delta=(md-n-1-d)H|_{\Delta}=eH|_{\Delta}$. For each irreducible curve $C\subseteq \Delta$, one has $C\cdot H\geqslant 1$, which implies that $C\cdot K_\Delta\geqslant e$. As $e\geqslant 1$, $K_\Delta$ is ample by the Nakai-Moishezon criterion.
It then follows from Lemma~\ref{KXgeacg} that $\cg(\Delta)\geqslant \frac{e}{2}+1.$ The $m$-th covering $\Delta\to \PP^n$ is given by $H|_\Delta$, which is equal to $\frac{1}{e}K_\Delta$, as we already calculated.
\end{proof}

Finally, we will need the following result which is known to experts:

\begin{lem}\label{lem: birational map of canonically polarized}
	Let $X$ and $Y$ be smooth projective varieties with ample canonical divisors $K_X$ and $K_Y$. Then every birational map $\varphi\colon X\rat Y$ is biregular.
\end{lem}
\begin{proof}
	Since $K_X$ and $K_Y$ are ample, the canonical rings $R(X,K_X)=\bigoplus_{m=0}^{\infty}\Cohom^0(X,mK_X)$ and $R(Y,K_Y)=\bigoplus_{m=0}^{\infty}\Cohom^0(Y,mK_Y)$ are finitely generated. Since $\varphi$ is birational, for each $m$ the natural homomorphisms
	$
	\varphi^*\colon\Cohom^0(Y,mK_Y)\to \Cohom^0(X,mK_X)
	$
	are isomorphisms, see e.g. \cite[Lemma 6.3]{UenoClassificationTheory}. Now the isomorphism $\varphi^*\colon R(Y,K_Y)\iso R(X,K_X)$ induces a biregular isomorphism
	$
	\varphi\colon X\simeq\Proj R(X,K_X)\iso\Proj R(Y,K_Y)\simeq Y,
	$
	as desired.
\end{proof}

\subsection{Elementary relations over bases of codimension at least $3$}\label{sec: covgen bound}
In this section, we bound the covering genus of centres of Sarkisov links of type \II~arising in elementary relations over a base of codimension at least $3$ (Proposition~\ref{prop: covgen link bound} below).

\medskip

As explained in \cite{BLZ}, the following is a consequence of the BAB conjecture, which was recently established in arbitrary dimension by C.~Birkar \cite{Birkar}.

\begin{prop} \cite[Proposition 5.1]{BLZ}\label{pro:BAB}
Let $n$ be an integer, and let $\mathcal{Q}$ be the set of weak Fano terminal varieties of dimension~$n$.
There are integers $d,l,m\geqslant 1$, depending only on $n$, such that for each $X \in \mathcal{Q}$ the following hold:
\begin{enumerate}
\item
$\dim(H^0(-mK_X))\leqslant l$;
\item
The linear system $\lvert -mK_X \rvert$ is base-point free;
\item \label{BAB:3}
The morphism $\varphi\colon X\stackrel{\lvert -mK_X\rvert}{\longto} \mathbb{P}^{\dim(H^0(-mK_X))-1}$ is birational onto its image and contracts only curves $C\subseteq X$ with $C\cdot K_X=0$;
\item
$\deg \varphi(X) \leqslant d$.
\end{enumerate}
\end{prop}

As a consequence, the following result was proven in \cite{BLZ}.

\begin{cor} \label{cor:BAB} \cite[Corollary 5.2]{BLZ}
Let $\pi\colon Y \to X$ be the blow-up of a reduced but not necessarily irreducible codimension $2$ subvariety $\Gamma \subseteq X$, $Y \ps \widehat Y$ a pseudo-isomorphism, and assume that both $X$ and $\widehat Y$ are weak Fano terminal varieties of dimension~$n \geqslant 3$,
whose loci covered by curves with trivial intersection against the canonical divisor has codimension at least $2$.
Let $\varphi$ be the birational morphism associated to the linear system $\lvert -mK_X \rvert$, with $m$ given by  Proposition~$\ref{pro:BAB}$, and assume that $\Gamma$ is not contained in the exceptional locus $\mathrm{Ex}(\varphi)$.
Then through any point of $\Gamma \setminus \mathrm{Ex}(\varphi)$ there is an irreducible curve $C\subseteq \Gamma$ with $\gon(C)\leqslant d$ and $C\cdot (-mK_X)\leqslant d$, where $d$ is the integer from Proposition~$\ref{pro:BAB}$.
\end{cor}

As the curves $C$ of Corollary~\ref{cor:BAB} are birationally sent by $\lvert -mK_X \rvert$ to curves of bounded degree in a projective space, Castelnuovo's bound implies that they also have bounded genus, see e.g.~\cite{Harris}. Hence, the covering genus of $\Gamma$ is bounded by some constant only depending on the dimension $n$.

The next result, crucial for our application to Severi-Brauer fibrations, bounds the covering genus of centres associated to a Sarkisov link of type \II. We consider a rank $2$ fibration $X/\widehat{B}$ that is dominated by a rank $3$ fibration $T/B$, and prove that the covering genus of the centres of the divisorial contractions in the rank $2$ fibration are bounded. The proof we give is the one of \cite[Proposition 5.3]{BLZ}, adapted to our situation. The difference in the assumptions is that we suppose $\dim(T)-\dim(B)\geqslant 3$, but  \cite[Proposition 5.3]{BLZ} supposes $\dim(T)-\dim(B)\geqslant 2$ and $\dim(X)-\dim(\widehat{B})=1$, and only proves that the covering gonality is bounded. In  Example~\ref{ExCovgon} below, we explain why the covering genus can actually not be bounded if one only assumes $\dim(T)-\dim(B)=2$, even in the case where $\dim(X)-\dim(\widehat{B})=1$.

We are interested in Sarkisov links of type \II.

\begin{mydef}
Let $\chi\colon X_1\dasharrow X_2$ be a Sarkisov link of type \II\ over $B_{\chi}$. As in Figure~\ref{fig:SarkisovTypes}, this corresponds to a diagram
\begin{equation}
	\begin{tikzcd}[ampersand replacement=\&,column sep=.8cm,row sep=0.16cm]
	Y_1\ar[dd,swap]  \ar[rr,dotted,-] \&\& \ar[dd] Y_2 \\ \\
	X_1 \ar[rr,"\chi",dashed,swap] \ar[dr,"\rm fib",swap] \&  \& X_2, \ar[dl,"\rm fib"] \\
	\& B_\chi \&
	\end{tikzcd}
	\label{eq: Link of type II with Y}
\end{equation}
where $Y_i\to X_i$ is a divisorial contraction and $X_i\to B_\chi$ is a Mori fibre spaces for $i=1,2$.

The \emph{covering genus} $\cg(\chi)$ of $\chi$ is defined to be the covering genus of the centre $\Gamma_1\subseteq X_1$ of the divisorial contraction $Y_1\to X_1$ associated to~$\chi$.
\end{mydef}
\begin{prop}\label{prop: covgen link bound}
For each $n\geqslant 3$, there is an integer $N(n)\geqslant 1$ such that the following holds. Let $\chi$ be a Sarkisov link of type \II\ between Mori fibre spaces of dimension $n$. If $\chi$ arises in an elementary relation induced by a rank $3$ fibration $T/B$ with $\dim(T)-\dim(B)\geqslant 3$, then both $\chi$ and $\chi^{-1}$ have covering genus smaller or equal to $N(n)$.
\end{prop}
\begin{proof}
We write our Sarkisov link as $\chi\colon X_1\dasharrow X_2$ as in diagram~\eqref{eq: Link of type II with Y}, with a base $B_\chi=\widehat{B}$, that dominates $B$ via a morphism $\widehat{B}\to B$. Our aim is to prove that the centre $\Gamma_i\subseteq X_i$ of the divisorial contraction $Y_i\to X_i$ has covering genus bounded by some integer that depends only on the dimension $n$ of $X_i$ and $Y_i$. We may  prove it only for $\Gamma_1$, as we may replace $\chi$ with $\chi^{-1}$. We may morever assume that $\cg(\Gamma_1)\geqslant 1$, and thus that $\Gamma_1$ is not covered by rational curves.

As $\chi$ arises in an elementary relation induced by the rank $3$ fibration $T/B$, the rank $3$ fibration $T/B$ factorises through the rank $2$ fibration $Y_1/\widehat{B}$. This gives a commutative diagram
\[
\begin{tikzcd}[link]
	T \ar[rrr] \ar[dr,dashed] &&& B \\
	& Y_1 \ar[r] & \widehat{B} \ar[ur]
\end{tikzcd}
\]where $T \rat Y_1$ is a birational contraction, and $\widehat{B} \to B$ is a morphism with connected fibres. We obtain a sequence of rational maps
\[
\begin{tikzcd}[link]
	T \ar[r,dashed] & Y_1 \ar[r] & X_1\ar[r]& \widehat{B} \ar[r]& B.
\end{tikzcd}
\]
As $T/B$ is a rank $3$ fibration and $T\dasharrow Y_1$ and $T\dasharrow X_1$ are compositions of log-flips and divisorial contractions, the two morphisms $Y_1/B$ and $X_1/B$ are also rank $r+1$ and rank $r$ fibrations for some $r\in \{1,2\}$ \cite[Lemma 3.4]{BLZ}. Moreover, \cite[Lemma 3.5]{BLZ} implies the rank $r$ fibration $X_1/B$ is pseudo-isomorphic to another rank $r$ fibration $X/B$, such that $-K_X$ is relatively big and nef. Then, \cite[Lemma 2.17]{BLZ} gives a sequence of log-flips over $B$ from $Y_1/B$ to another rank $r+1$ fibration $Y/B$ such that the induced map $Y\to X$ is a divisorial contraction. We finally apply  \cite[Lemma 3.5]{BLZ} again and get that the  rank $r+1$ fibration $Y/B$ is pseudo-isomorphic to another rank $r$ fibration $\widehat{Y}/B$, such that $-K_{\widehat{Y}}$ is relatively big and nef, see the diagram below.
\[\begin{tikzcd}[link]
 Y_1\ar[dd]\ar[r,dotted,-]&Y\ar[dd]\ar[r,dotted,-]& \widehat{Y}\ar[ldddd]  \\ \\
X_1  \ar[rdd]\ar[r,dotted,-]& X\ar[dd]\\ \\
&B
\end{tikzcd}
\]
As $\Gamma_1$ is of codimension $2$ and is not covered by rational curves, it is not contained in the exceptional locus of the pseudo-isomorphism $X_1\ps X$ (by \cite[Lemma 2.15(3)]{BLZ}), and thus it is birational to its image, $\Gamma\subseteq X$. Bounding the covering genus of $\Gamma$ is equivalent to bounding the covering genus of $\Gamma_1$.

For a general point $p\in B$, we denote by $X_p$, $Y_p$ and $\widehat{Y}_p$ the fibres over $p$ in $X$, $Y$ and $\widehat{Y}$ respectively. All three are varieties of dimension $n_0=\dim(T)-\dim(B)\geqslant 3$. As $-K_X$ and $-K_{\widehat{Y}}$ are relatively big and nef over $B$, the varieties $X_p$ and $\widehat{Y}_p$ are weak-Fano varieties. Moreover, in each of these two weak-Fano varieties, the curves having intersection $0$ with the canonical divisor cover a subset of codimension at least $2$ of $X_p$ and $\widehat{Y}_p$ respectively, by \cite[Corollary 3.6]{BLZ}.

We denote by $\eta\colon X\to B$ the morphism studied above, and consider the morphism $\varphi =  \lvert -mK_X \rvert \times \eta \colon X\times X \rat \mathbb{P}^N \times B$, where $m$ is given by Proposition~$\ref{pro:BAB}$ applied in dimension $n_0$. We write $\Gamma_p=\Gamma \cap X_p$ , which is the fibre of $\Gamma\to B$ over $p$. By the generic smoothness result applied to $X_{\mathrm{reg}}$, we find that this intersection is transverse and that $\Gamma_p$ is reduced. The restriction of $\varphi$ is a  morphism $\varphi_p\colon X_p\to \mathbb{P}^N$, birational onto its image, described in Proposition~$\ref{pro:BAB}$. As $\Gamma_p$ is not covered by rational curves, we may again apply \cite[Lemma 2.15(3)]{BLZ} and obtain that $\Gamma_p$ is not contained in the exceptional locus of $\varphi_p$.
We may thus apply Corollary~\ref{cor:BAB} to the blow-up $Y_p\to X_p$ of $\Gamma_p$ and the pseudo-isomorphism $Y_p\ps \widehat{Y}_p$. This corollary implies that for a general $p$, the locus $\Gamma_p \setminus \mathrm{Ex}(\varphi_{p})$ is covered by curves $C$ with  $C\cdot (-mK_{X_p})\leqslant d$, where $d$ is the integer from Proposition $\ref{pro:BAB}$. The curve $C$ is sent by $\varphi$ onto a curve of $\mathbb{P}^N\times \{p\}$ of bounded degree in a projective space, and thus of bounded genus, by Castelnuovo's bound \cite{Harris}. This achieves the proof.
\end{proof}

\begin{ex}\label{ExCovgon}
We now observe that Proposition~\ref{prop: covgen link bound} is false if $\dim(T)-\dim(B)=2$. Define $B=\PP^1$, $B_2=\PP^1\times \PP^1$, with $B_2\to B$ given by $(x,y)\mapsto x$, $X=\PP^1\times \PP^1\times \PP^1$, and obtain a $\PP^1$-bundle $X\to B_2$, $(x,y,z)\mapsto (y,z)$. For $n\geqslant 1$, we define $\chi\in \Bir(X/B_2)$ to be the birational involution given by
\[
\chi\colon \left ( [x_0:x_1],y,z \right )\mapsto \left ( [x_0a(y,z)+x_1b(y,z):x_0c(y,z)-x_1a(y,z)],y,z\right ),
\] where $a,b,c\in \CC[y_0,y_1,z_0,z_1]$ are general bihomogeneous polynomials of bidegree $(1,n)$. The birational map $\chi$ is an isomorphism over $B_2\setminus \widehat{\Gamma}$, where $\widehat\Gamma$ is the zero locus of $a^2+bc=0$, being a smooth curve of bidegree $(2,n)$ in $B_2$, having then genus $n-1$. The birational map $\chi$ is a Sarkisov link of type \II~over the base $\widehat{B}$, as in diagram~\eqref{eq: Link of type II with Y}, with  $X_1=X_2=X$, $Y_1=Y_2=Y$ and $B_\chi=\widehat{B}$, where $Y\to X$ is the blow-up of a curve $\Gamma\subseteq X$, birational to $\widehat{\Gamma}$ via the morphism $X\to B_2$. Note that $Y/B$ is a rank $3$-fibration. Indeed, $X\to B$ is a trivial $\PP^1\times \PP^1$-bundle, and $Y\to X$ blows up two points in a general fibre, giving then some del Pezzo surface fibration $Y\to B$. One can also check all other conditions, and get an elementary relation involving $\chi$.
\end{ex}

\subsection{Elementary relations over bases of the same dimension}\label{sec:classificationElementaryRelations}
In this section we describe elementary relations that involve Sarkisov links whose bases are all equidimensional and where one link is of type \II~but not all. As usual, to construct a non-trivial homomorphism from a group of birational self-maps, we will need a convenient equivalence relation between Sarkisov links.

\begin{mydef}\label{Def:equiLink}
We say that two Sarkisov links $\chi\colon X_1\dasharrow X_2$ over $B$ and $\chi'\colon X_1'\dasharrow X_2'$ over $B'$ are \emph{equivalent} if there is a commutative diagram\[
\begin{tikzcd}[link]
	X_1\ar[dddd,swap]\ar[rrr,dashed,"\psi_1"]\ar[ddr,dashed,"\chi",swap]&&& X'_1\ar[ddr,dashed,"\chi'"]\ar[dddd]  \\ \\
	&X_2\ar[ddl,swap]\ar[rrr,dashed,"\psi_2"] &&& X'_2\ar[ddl] \\ \\
	B \ar[rrr,"\psi",dashed]&&& B'
\end{tikzcd}
\]
where $\psi,\psi_1,\psi_2$ are birational maps, and where $\psi_i$ induces an isomorphism between the generic fibres of $X_i/B$ and $X_i'/B'$, via the base-change $\psi$, for each $i\in \{1,2\}$.
\end{mydef}

We start with the following observation, that generalises \cite[Lemma 3.19(1)]{BLZ}.

\begin{lem}
\label{lem:LinksIV}
Let $Y/B$ be a rank~$2$ fibration and let $\chi\colon X_1/B_1\dasharrow X_2/B_2$ be the associated Sarkisov link, well-defined up to taking inverse and up to isomorphisms of Mori fibre spaces. Suppose that $\dim(B)=\dim(B_1)=\dim(B_2)$. Then, $\chi$ has type \IV\ if and only if $B$ is not $\QQ$-factorial. Moreover, in this case, $B_i\to B$ is a small contraction for each $i\in \{1,2\}$.
\end{lem}
\begin{proof}
Assume that $B$ is not $\QQ$-factorial. Then $\chi$ has type \IV, since the base of a terminal Mori fibre space is always $\QQ$-factorial \cite[Proposition~2.10]{BLZ}: compare the diagrams in Figure~\ref{fig:SarkisovTypes}. Moreover, $B_1,B_2$ are $\QQ$-factorial and $B_i\to B$ is birational for $i\in \{1,2\}$, so $B_i\to B$ is a small contraction.

Assume that $B$ is $\QQ$-factorial and  that $\chi$ is  of type \IV, and let us derive a contradiction. For each $i\in \{1,2\}$, $B_i$ is $\QQ$-factorial and since $\rho(B_i/B)=1$ (Figure~\ref{fig:SarkisovTypes} for type \IV) and both have the same dimension, we conclude that $B_i\to B$ is a divisorial contraction, that contracts a divisor $E_i\subseteq B_i$.
If the birational map $B_1 \rat B_2$ sends $E_1$ onto $E_2$, then $\chi\colon X_1/B_1\dasharrow X_2/B_2$ is an isomorphism of Mori fibre spaces \cite[Lemma 2.18]{BLZ},  impossible as it is a Sarkisov link. Hence, the pull-backs of $E_1$, $E_2$ together with the choice of any ample divisor give three independent classes in $N^1(Y/B)$, which contradicts $\rho(Y/B) = 2$.
\end{proof}

We now consider elementary relations involving one Sarkisov link of type \II, and one other type of Sarkisov link, and where the bases all have the same dimension. Note that an elementary relation is always of length at least $3$, and that up to cyclic permutation, we may always assume the link of type \II~to be the second and the first to be of type~\I, \III~or \IV. This gives three distinct cases, studied in Lemmas  \ref{lem:IandthenII}, \ref{lem:IIIandthenII} and \ref{lem:IVandthenII} respectively.

We will use the following consequences of $2$-ray games repeatedly (see \cite[Lemma~4.2 and Proposition~4.3]{BLZ} and their proofs, including \cite[Lemmas 2.17 and~2.19]{BLZ}):

\begin{enumerate}
\item Given two rank $2$ fibrations $Y_i/B$, for $i=1,2$ with a sequence of log-flips $Y_1\ps Y_2$ over $B$, and a rank $3$ fibration $T_1/B$ that dominates $Y_1/B$ with a divisorial contraction $T_1\to Y_1$ over $B$, the involved maps can uniquely be completed to a commutative diagram
	\[\begin{tikzcd}[ampersand replacement=\&,column sep=1.3cm,row sep=0.16cm]
	T_1\ar[rr,dotted,-] \ar[dd,"E",swap]\&\&T_2\ar[dd,"E",swap] \\ \\
Y_1\ar[ddr,swap]  \ar[rr,dotted,-] \&\& Y_2, \ar[ddl]  \\ \\
\& B \&
\end{tikzcd}\]
where $T_1\ps T_2$ is a sequence of log-flips, and $T_2\to Y_2$ a divisorial contraction of the ``same'' divisor, namely the strict transform of the divisor contracted by $T_1\to Y_1$ under the map $T_1\ps T_2$. As usual, the labels of arrows denote the contracted divisors.
\item Given a rank $2$ fibration $T_1/X$ with two divisorial contractions $T_1\to Y_1\to X$, the involved maps can uniquely be completed to a commutative diagram
\[\begin{tikzcd}[ampersand replacement=\&,column sep=1.3cm,row sep=0.16cm]
	T_1\ar[rr,dotted,-] \ar[dd,"E",swap]\&\&T_2\ar[dd,"F",swap] \\ \\
Y_1\ar[ddr,swap,"F"]  \ar[rr,dotted,-] \&\& Y_2. \ar[ddl, "E"]  \\ \\
\& X \&
\end{tikzcd}\]
\end{enumerate}

\begin{lem}\label{lem:IandthenII}
	Let $T/B$ be a rank $3$ fibration inducing an elementary relation \[\chi_n\circ \cdots\circ \chi_2\circ\chi_1=\id\]
	with $n\geqslant 3$ and Sarkisov links $\chi_i\colon X_i\dashrightarrow X_{i+1}$ $($where $X_{n+1}=X_1)$ between Mori fibre spaces $X_i/B_i$, for $i=1,\ldots,n$.
	Assume that $\dim(B)=\dim(B_2)$, that $\chi_1$ is of type \I\, and that $\chi_2$ is of type \II. Then, $\chi_3$ is of type \III, and one of the following two possiblities occurs:
	\begin{enumerate}
	\item\label{IandthenII4}
	$n=4$, and $\chi_4$ is a link of type \II\ that is equivalent to  $\chi_2^{-1}$.
	\item\label{IandthenII3}
	$n=3$ and the centres of the divisorial contractions $Y_2\to X_2$ and $Y_3\to X_3$ associated to $\chi_2$ are contained in the preimage of a uniruled divisor of $B_2$.
	\end{enumerate}

\end{lem}
\begin{proof}
As  $\chi_1$ is of type \I, it comes with a divisorial contraction  $Y_1\to X_1$, that contracts a divisor $G\subseteq Y_1$. Similarly, $\chi_2$ being of type \II, comes with divisorial contractions $Y_i\to X_i$, contracting a divisor $E_i$, for $i\in \{2,3\}$.
This gives a commutative diagram
\[\begin{tikzcd}[ampersand replacement=\&,column sep=1.3cm,row sep=0.16cm]
T_1\ar[rr,dotted,-] \ar[dd,"E_2",swap]\&\&Y_2\ar[dd,"E_2",swap]  \ar[rr,dotted,-] \&\& \ar[dd,"E_3"] Y_3\\ \\
Y_1\ar[dd,"G",swap]  \ar[rr,dotted,-] \&\& X_2 \ar[dd] \ar[rr,"\chi_2",dashed,swap]\&  \& X_3\ar[dd]   \\ \\
X_1 \ar[uurr,"\chi_1",dashed,swap] \ar[dr,swap] \&  \& B_2 \ar[dl]\ar[rr,equal]\&  \& B_3. \\
\& B_1 \&
\end{tikzcd}\]
As $Y_2/B_1$ is a rank $3$ fibration and the relation comes from $T/B$, we find $B=B_1$. Moreover, $\dim(B)=\dim(B_2)$ by assumption, so all bases $B_1,B_2,B_3$ have the same dimension, and are $\QQ$-factorial, as these are the bases of Mori fibre spaces. Hence, $B_2\to B_1$ is a divisorial contraction, that contracts a divisor $G_2\subseteq B_2$. The preimage in $X_2$ of this divisor $G_2$ is a divisor of $X_2$; each irreducible component that surjects to $G_2$ is contracted by $\chi_1^{-1}$, so there is only one such component, and it is the strict stransform of $G\subseteq Y_1$.

The $2$-ray game on $X_{3}/B_1$ gives the link $\chi_{3}$, which is thus of type \III\, or \IV. As $B_1$ is $\QQ$-factorial and all bases have the same dimension, $\chi_3$ is of type \III\ (Lemma~\ref{lem:LinksIV}). We complete the diagram and obtain
\[\begin{tikzcd}[ampersand replacement=\&,column sep=1.3cm,row sep=0.16cm]
	T_1\ar[rr,dotted,-] \ar[dd,"E_2",swap]\&\&Y_2\ar[dd,"E_2",swap]  \ar[rr,dotted,-] \&\& \ar[dd,"E_3"] Y_3\ar[rr,dotted,-]\&\& T_4 \ar[dd,"E_3",swap]\\ \\
Y_1\ar[dd,"G",swap]  \ar[rr,dotted,-] \&\& X_2 \ar[dd] \ar[rr,"\chi_2",dashed,swap]\&  \& X_3\ar[ddrr,"\chi_3",dashed,swap]\ar[dd]\ar[rr,dotted,-]  \&\& Y_4 \ar[dd,"G'"]  \\ \\
X_1 \ar[uurr,"\chi_1",dashed,swap] \ar[dr,swap] \&  \& B_2 \ar[dl,"G_2"]\ar[rr,equal]\&  \& B_3\ar[dr]\&  \& X_4 \ar[dl] \\
\& B_1 \& \& \& \& B_4=B_1
\end{tikzcd}\]
where $G'$ is the divisor contracted by $Y_4\to X_4$. We now consider two cases, depending on whether $E_3$ is the strict transform of $G$, or not.

Suppose first that $E_3$ is not the strict transform of $G$. In this case, $G'$ is the strict transform of $G$. Indeed, by the same argument as above, the preimage of $G_2\subseteq B_3$ in $X_3$ is a divisor that contains a unique irreducible component that surjects to $G_2$ and this latter is the strict transform of $G$ and of $G'$. It remains to apply a $2$-ray game, on $T_i\to X_i$ for $i\in \{1,4\}$, which exchanges the two divisors $G$ and $E_2$ (respectively $E_3$). This gives the last link $\chi_4$ of type \II, as in Figure~\ref{fig:relationwithtypeI}. Hence, $\chi_4$ is equivalent to $\chi_2^{-1}$.

Suppose now that $E_3$ is the strict transform of $G$. In this case, the $2$-ray game applied to $T_1/X_1$ exchanges the two divisors $E_2$ and $E_3$.  This gives $n=3$, and yields the diagram of Figure~\ref{fig:triangle_relation}. Moreover, the centres of the divisorial contractions $Y_2\to X_2$ and $Y_3\to X_3$ associated to $\chi_2$ are contained in the preimage of the divisor contracted by $B_2\to B_1$, which is uniruled by  \cite[Lemma 2.15]{BLZ}.
\end{proof}

\begin{figure}[ht]
\begin{tabular}{c c}
	\begin{tikzpicture}[scale=1.1,font=\small,outer sep=-1pt] 
	\node (Y3) at (60+90:\Ra) {$Y_3$};
	\node (Y2) at (120+90:\Ra) {$Y_2$};
	\node (T1') at (180+90:\Ra) {$T_1$};
	\node (T1) at (240+90:\Ra) {$T_1'$};
	\node (T4) at (300+90:\Ra) {$T_4'$};
	\node (T4') at (360+90:\Ra) {$T_4$};
	\node (X3) at (60+90:\Rb) {$X_3$};
	\node (X2) at (120+90:\Rb) {$X_2$};
	\node (X1) at (300:\Rc) {$X_1$};
	\node (X4) at (420:\Rc) {$X_4$};
	\node (Y1') at (180+90:\Rb) {$Y_1$};
	\node (Y1) at (240+90:\Rb) {$Y_1'$};
	\node (Y4) at (300+90:\Rb) {$Y_4'$};
	\node (Y4') at (360+90:\Rb) {$Y_4$};
	\node (B) at (0:0cm) {$B$};
	\node (Bhat) at (180:\Rc) {$B_2$};
	\draw[dotted,-] (Y3) to [bend right=21] (Y2)
			(Y2) to [bend right=21] (T1')
			(T1') to [bend right=21] (T1)
			(T1) to [bend right=21] (T4)
			(T4) to [bend right=21] (T4')
			(T4') to [bend right=21] (Y3)
			(X2) to [bend right=21] (Y1')
			(Y1) to [bend right=21] (Y4)
			(Y4') to [bend right=21] (X3);
	\draw[->] (Bhat) to(B);
	\draw[->] (X1) to(B);
	\draw[->] (X4) to(B);
	\draw[->] (T4') to [auto,"\tiny $E_3$"] (Y4');
	\draw[->] (Y3) to [auto,pos=.1,"\tiny $E_3$"] (X3);
	\draw[->] (Y2) to [auto,swap,pos=.1,"\tiny $E_2$"] (X2);
	\draw[->] (T1') to [auto,swap,"\tiny $E_2$"] (Y1');
	\draw[->] (T1) to [auto,swap,"\tiny $G$"] (Y1);
	\draw[->] (T4) to [auto,"\tiny $G$"] (Y4);
	\draw[->] (Y1') to [auto,swap,"\tiny $G$"] (X1);
	\draw[->] (Y4') to [auto,"\tiny $G$"] (X4);
	\draw[->] (Y1) to [auto,swap,pos=.7,"\tiny $E_2$"] (X1);
	\draw[->] (Y4) to [auto,pos=.7,"\tiny $E_3$"] (X4);
	\draw[->] (X3) to  (Bhat);
	\draw[->] (X2) to  (Bhat);
	\draw[dashed,->] (X2) to [bend left=30,auto,"$\chi_2$"] (X3);
	\draw[dashed,->] (X1) to [auto,pos=.7,"$\chi_1$"] (X2);
	\draw[dashed,->] (X4) to [bend left=38,auto,"$\chi_4$"] (X1);
	\draw[dashed,->] (X3) to [auto,pos=.3,"$\chi_3$"] (X4);
	\end{tikzpicture}%
	&
	\begin{tikzpicture}[scale=1.3,font=\small,outer sep=-1pt] 
	\node (X3) at (45+90:\Ra*1.15) {$X_3/B_3$};
	\node (X2) at (45+180:\Ra*1.15) {$X_2/B_2$};
	\node (X1) at (45+270:\Ra*1.15) {$X_1$};
	\node (X4) at (45:\Ra*1.15) {$X_4$};
	\node (Y3) at (180-15:\Rb*1.3) {$Y_3/B_3$};
	\node (Y2) at (180+15:\Rb*1.3) {$Y_2/B_2$};
	\node (X2P) at (270-15:\Rb*1.3) {$X_2$};
	\node (Y1) at (270+12:\Rb*1.3) {$Y_1$};
	\node (Y1') at (-10:\Rb*1.3) {$Y_1'$};
	\node (Y4') at (+10:\Rb*1.3) {$Y_4'$};
	\node (Y4) at (90-12:\Rb*1.3) {$Y_4$};
	\node (X3P) at (90+15:\Rb*1.3) {$X_3$};
	\node (Y3P) at (45+90:\Rc) {$Y_3$};
	\node (Y2P) at (45+180:\Rc) {$Y_2$};
	\node (T1) at (-60:\Rc) {$T_1$};
	\node (T1') at (-20:\Rc*1.1) {$T_1'$};
	\node (T4') at (20:\Rc*1.1) {$T_4'$};
	\node (T4) at (60:\Rc) {$T_4$};
	\draw[thick,dotted,-] (Y3P) to [bend right=21] (Y2P)
			(Y2P) to [bend right=21] (T1)
			(T1) to [bend right=21] (T1')
			(T1') to [bend right=21] (T4')
			(T4') to [bend right=21] (T4)
			(T4) to [bend right=21] (Y3P);
	\draw[thick,dotted,-]	(Y3) to (Y2)
			(X2P) to (Y1)
			(Y1') to (Y4')
			(Y4) to (X3P);
	\draw[very thick,-] (X2) to (X2P)
		(X3P) to (X3)
		(Y3P) to (Y3)
		(Y2) to (Y2P);
	\draw[very thick,-,red] (Y2) to node[right] {$E_2$}(X2)
	(Y2P) to node[right] {$E_2$}(X2P)
	(Y1) to node[right] {$E_2$}(T1)
	(X1) to node[right] {$E_2$}(Y1');
	\draw[very thick,-,blue] (T1') to node[above] {$G$}(Y1')
	(Y1) to node[above] {$G$}(X1)
	(Y4') to node[above] {$G$}(T4')
	(X4) to node[above] {$G$}(Y4);
	\draw[very thick,-,orange]
	(X3) to node[right] {$E_3$} (Y3)
	 (X3P) to node[right] {$E_3$} (Y3P)
	(T4) to node[right] {$E_3$}(Y4)
	(Y4') to node[right] {$E_3$}(X4);
	\end{tikzpicture}
\end{tabular}
\caption{Diagrams illustrating the relation of Lemma~\ref{lem:IandthenII}\ref{IandthenII4}: The one on the left is \cite[Figure~11]{BLZ}; The one on the right depicts the involved rank $1$, $2$ and $3$ fibrations, where the base is $B$ if not specified otherwise}
\label{fig:relationwithtypeI}
\end{figure}

\begin{figure}[ht]
\begin{tabular}{c c}
\begin{tikzpicture}[scale=1.5,font=\small] 
\node (Y3) at (135:\Ra) {$Y_3$};
\node (Y2) at (-135:\Ra) {$Y_2$};
\node (X3) at (135:\Rb) {$X_3$};
\node (X2) at (-135:\Rb){$X_2$}; 
\node (X) at (-42.5:\Rb) {$Y_1$};
\node (Y) at (42.5:\Rb) {$Y_1'$};
\node (T1) at (-42.5:\Ra) {$T_1$};
\node (T1') at (42.5:\Ra) {$T_1'$};
\node (B) at (0:0cm) {$B$};
\node (hatB) at (180:\Rc) {$B_2$};
\node (X1) at (0:\Rc) {$X_1$};
\draw[->] (Y2) to [auto,"\tiny $E_2$"] (X2);
\draw[->] (Y3) to [auto,"\tiny $E_3$"] (X3);
\draw[->] (X3)--  (hatB);
\draw[->] (X2)--  (hatB);
\draw[->] (T1)  to [auto,"\tiny $E_2$"]  (X);
\draw[->] (T1')  to [auto,"\tiny $E_3$"]  (Y);
\draw[->] (X) to [auto,swap,pos=.3,"\tiny $E_3$"] (X1);
\draw[->] (Y) to [auto,swap,pos=.3,"\tiny $E_2$",swap] (X1);
\draw[->] (hatB)--  (B);
\draw[->] (X1)--  (B);
\draw[dashed,<-,swap] (X3) to [bend right=38] node[auto]{$\chi_2$} (X2);
\draw[dashed,<-,swap] (X2) to [bend right=42] node[auto]{$\chi_1$} (X1);
\draw[dashed,<-,swap] (X1) to [bend right=42] node[auto]{$\chi_3$} (X3);
\draw[dotted,-] (X2) to [bend right=30] (X);
\draw[dotted,-] (X3) to [bend left=30] (Y);
\draw[dotted,-] (Y2) to [bend left=30] (Y3);
\draw[dotted,-] (Y2) to [bend right=30] (T1);
\draw[dotted,-] (Y3) to [bend left=30] (T1');
\draw[dotted,-] (T1') to [bend left=30] (T1);
\end{tikzpicture}
&
\begin{tikzpicture}[scale=1.1,font=\small,outer sep=-1pt] 
\node (X3) at (120:\Ra*1) {$X_3/B_3$};
\node (X2) at (240:\Ra*1) {$X_2/B_2$};
\node (X1) at (0:\Ra*1) {$X_1$};
\node (Y3) at (180-30:\Rb*0.9) {$Y_3/B_3$};
\node (Y2) at (180+30:\Rb*0.9) {$Y_2/B_2$};
\node (X2P) at (300-30:\Rb*1) {$X_2$};
\node (Y1') at (300+30:\Rb*1) {$Y_1$};
\node (Y1) at (60-30:\Rb*1) {$Y_1'$};
\node (X3P) at (60+30:\Rb*1) {$X_3$};
\node (Y3P) at (180-65:\Rc*0.6) {$Y_3$};
\node (Y2P) at (180+65:\Rc*0.6) {$Y_2$};
\node (T1') at (-30:\Rc) {$T_1$};
\node (T1) at (30:\Rc) {$T_1'$};
\draw[thick,dotted,-] (Y3P) to [bend right=21] (Y2P)
		(Y2P) to [bend right=21] (T1')
		(T1') to [bend right=21] (T1)
		(T1) to [bend right=21] (Y3P);
\draw[thick,dotted,-]	(Y3) to (Y2)
		(X2P) to (Y1')
		(Y1) to (X3P);
\draw[very thick,-] (X2) to (X2P)
	(X3P) to (X3)
	(Y3P) to (Y3)
	(Y2) to (Y2P);
\draw[very thick,-,red] (Y2) to node[right] {$E_2$}(X2)
(Y2P) to node[right] {$E_2$}(X2P)
(Y1') to node[above right] {$G'$}(T1')
(X1) to node[above right] {$G'$}(Y1);
\draw[very thick,-,orange] (T1) to node[below right] {$G$}(Y1)
(Y1') to node[below right] {$G$}(X1);
\draw[very thick,-,orange]
(X3) to node[right] {$E_3$} (Y3)
 (X3P) to node[right] {$E_3$} (Y3P);
\end{tikzpicture}
\end{tabular}
\caption{Diagrams illustrating the relation of Lemma~\ref{lem:IandthenII}\ref{IandthenII3}: The one on the left is a generalization of \cite[Example 4.32]{BLZ}; The one on the right depicts the involved rank $1$, $2$ and $3$ fibrations, where the base is $B$ if not specified otherwise}
\label{fig:triangle_relation}
\end{figure}
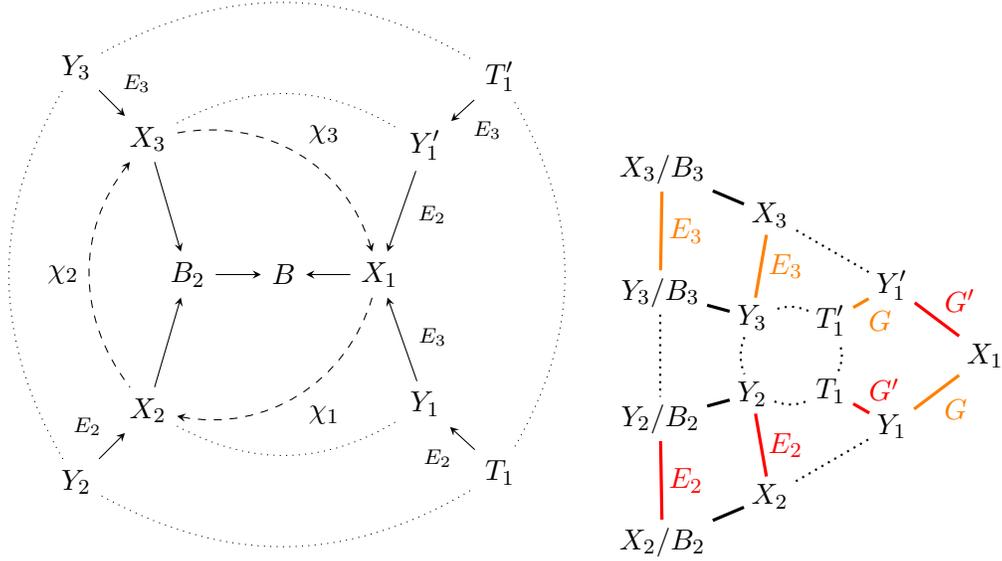

\begin{lem}\label{lem:IVandthenII}
Let $T/B$ be a rank $3$ fibration inducing an elementary relation \[
\chi_n\circ \cdots\circ \chi_2\circ\chi_1=\id
\]
with $n\geqslant 3$ and Sarkisov links $\chi_i\colon X_i\dashrightarrow X_{i+1}$ $($where $X_{n+1}=X_1)$ between Mori fibre spaces $X_i/B_i$, for $i=1,\ldots,n$.
Assume that $\dim(B)=\dim(B_1)=\dim(B_2)$, that $\chi_1$ is of type \IV\, and that $\chi_2$ is of type \II. Then, $n=4$, $\chi_3$ is of type \IV, $\chi_4$ is of type \II\ and is equivalent to~$\chi_2^{-1}$.

\end{lem}
\begin{proof}
As $\chi_2$ is of type \II, it comes with divisorial contractions $Y_i\to X_i$, contracting a divisor $E_i$, for $i\in \{2,3\}$.
As  $\chi_1$ is of type \IV, it is over a base $B'$ that comes with two morphisms $B_1\to B'$ and $B_2\to B'$. As $Y_2\to X_2\to B_2\to B'$ is a rank $3$ fibration, we find $B'=B$. By assumption, $\dim(B)=\dim(B_1)=\dim(B_2)$, so $B$ is not $\QQ$-factorial and we have  small extremal contractions $B_1\to B$ and $B_2\to B$ (Lemma~\ref{lem:LinksIV}). We obtain the following diagram
\[
\begin{tikzcd}[ampersand replacement=\&,column sep=1.3cm,row sep=0.16cm]
Y_1\ar[rr,dotted] \ar[dd,"E_2",swap]\&\&Y_2\ar[dd,"E_2",swap]  \ar[rr,dotted,-] \&\& \ar[dd,"E_3"] Y_3\\ \\
X_1\ar[dd,swap]  \ar[rr,dotted,"\chi_1",->]\&\& X_2 \ar[dd] \ar[rr,"\chi_2",dashed]\&  \& X_3\ar[dd]   \\ \\
B_1  \ar[dr,swap] \&  \& B_2 \ar[dl]\ar[rr,equal]\&  \& B_3. \\
\& B \&
\end{tikzcd}
\]
The $2$-ray game on $X_{3}/B$ gives the link $\chi_{3}$, which is thus of type \III\, or \IV~(Figure~\ref{fig:SarkisovTypes}). Moreover, $\chi_3$ is of type \IV\, as $B$ is not $\QQ$-factorial (Lemma~\ref{lem:LinksIV}), and thus we  find that $\chi_1$ and $\chi_3$ are both of type \IV. We complete the diagram and obtain
\[\begin{tikzcd}[ampersand replacement=\&,column sep=1.3cm,row sep=0.16cm]
Y_1\ar[rr,dotted] \ar[dd,"E_2",swap]\&\&Y_2\ar[dd,"E_2",swap]  \ar[rr,dotted,-] \&\& \ar[dd,"E_3"] Y_3\ar[rr,dotted,-]  \&\& \ar[dd,"E_3"] Y_4\\ \\
X_1\ar[dd,swap]  \ar[rr,dotted,"\chi_1",->]\&\& X_2 \ar[dd] \ar[rr,"\chi_2",dashed]\&  \& X_3\ar[dd] \ar[rr,dotted,"\chi_3",->]\&  \& X_4\ar[dd]  \\ \\
B_1  \ar[dr,swap] \&  \& B_2 \ar[dl]\ar[rr,equal]\&  \& B_3\ar[dr]  \& \& B_4.\ar[dl] \\
\& B \&\&\& \&B
\end{tikzcd}\]
By the uniqueness of log-flips \cite[p.~340]{Matsuki}, there are exactly two small contractions from a $\QQ$-factorial variety to $B$, up to isomorphism, one being $B_2=B_3\to B$, and the second one then being $B_1\to B$ or $B_4\to B$. Hence, the birational map $B_1\dasharrow B_4$ induced by the above diagram is an isomorphism and we can choose $B_4=B_1$ and obtain the diagram of Figure~\ref{fig:relationwithtypeIVbeforeII}. It implies that $\chi_4$ is of type \II, and is equivalent to $\chi_2^{-1}$.
\end{proof}

\begin{figure}[ht]
\begin{tabular}{c c}
	\begin{tikzpicture}[scale=1.1,font=\small,outer sep=-1pt] 
	\node (Y1) at (135:\Ra) {$Y_3$};
	\node (Y2) at (-135:\Ra) {$Y_2$};
	\node (Y3) at (-45:\Ra) {$Y_1$};
	\node (Y4) at (45:\Ra) {$Y_4$};
	\node (X1) at (135:\Rb) {$X_3$};
	\node (X2) at (-135:\Rb) {$X_2$};
	\node (X3) at (-45:\Rb) {$X_1$};
	\node (X4) at (45:\Rb) {$X_4$};
	\node (B) at (0:0cm) {$B$};
	\node (Bl) at (180:\Rc) {$B_2$};
	\node (Br) at (0:\Rc) {$B_1$};
	\draw[->] (Y1) to [auto,swap,"\tiny $E_3$"] (X1);
	\draw[->] (Y2) to [auto,"\tiny $E_2$"] (X2);
	\draw[->] (Y3) to [auto,swap,"\tiny $E_2$"] (X3);
	\draw[->] (Y4) to [auto,"\tiny $E_3$"] (X4);
	\draw[->] (X1) to  (Bl);
	\draw[->] (X2) to  (Bl);
	\draw[->] (X3) to  (Br);
	\draw[->] (X4) to  (Br);
	\draw[->] (Bl) to  (B);
	\draw[->] (Br) to  (B);
	\draw[dashed,->] (X2) to [bend left=38,auto,"$\chi_2$"] (X1);
	\draw[dotted,->] (X3) to [bend left=38,auto,"$\chi_1$"] (X2);
	\draw[dashed,->] (X4) to [bend left=38,auto,"$\chi_4$"] (X3);
	\draw[dotted,->] (X1) to [bend left=38,auto,"$\chi_3$"] (X4);
	\draw[dotted,-] (Y1) to [bend right=38] (Y2)
			(Y2) to [bend right=38] (Y3)
			(Y3) to [bend right=38] (Y4)
			(Y4) to [bend right=38] (Y1)
			(Bl) to [bend right=20] (Br);
	\end{tikzpicture}
	&
	\begin{tikzpicture}[scale=1.3,font=\small,outer sep=-1pt] 
	\node (X1) at (45+90:\Ra*1.15) {$X_3/B_3$};
	\node (X2) at (45+180:\Ra*1.15) {$X_2/B_2$};
	\node (X3) at (45+270:\Ra*1.15) {$X_1/B_1$};
	\node (X4) at (45:\Ra*1.15) {$X_4/B_4$};
	\node (Y1) at (180-15:\Rb*1.3) {$Y_3/B_3$};
	\node (Y2) at (180+15:\Rb*1.3) {$Y_2/B_2$};
	\node (X2B) at (270-15:\Rb*1.3) {$X_2/B$};
	\node (X3B) at (270+15:\Rb*1.3) {$X_1/B$};
	\node (Y3) at (-15:\Rb*1.3) {$Y_1/B_1$};
	\node (Y4) at (+15:\Rb*1.3) {$Y_4/B_4$};
	\node (X4B) at (90-15:\Rb*1.3) {$X_4/B$};
	\node (X1B) at (90+15:\Rb*1.3) {$X_3/B$};
	\node (Y1B) at (45+90:\Rc) {$Y_3/B$};
	\node (Y2B) at (45+180:\Rc) {$Y_2/B$};
	\node (Y3B) at (45+270:\Rc) {$Y_1/B$};
	\node (Y4B) at (45:\Rc*1.1) {$Y_4/B$};
	\draw[thick,dotted,-] (Y1) to  (Y2)
	(X2B) to  (X3B)
	(Y3) to (Y4)
	(X4B) to (X1B);
	\draw[thick,dotted,-]	(Y1B) to [bend right=21] (Y2B)
		(Y2B) to [bend right=21] (Y3B)
		(Y3B) to [bend right=21] (Y4B)
		(Y4B) to [bend right=21] (Y1B);
	\draw[very thick,-] (X1) to (X1B)
		(Y1) to (Y1B)
		(Y2) to (Y2B)
		(X2) to (X2B);
		\draw[very thick,-] (X4) to (X4B)
			(Y4) to (Y4B)
			(Y3) to (Y3B)
			(X3) to (X3B);
	\draw[very thick,-,red] (Y2) to node[right] {$E_2$}(X2)
	(Y2B) to node[right] {$E_2$}(X2B)
	(Y3) to node[right] {$E_2$}(X3)
	(Y3B) to node[right] {$E_2$}(X3B);
	\draw[very thick,-,orange]
	(X1) to node[right] {$E_3$} (Y1)
	 (X1B) to node[right] {$E_3$} (Y1B)
	 (X4) to node[right] {$E_3$} (Y4)
 	 (X4B) to node[right] {$E_3$} (Y4B);
	\end{tikzpicture}
\end{tabular}
\caption{Diagrams illustrating the relations of Lemma~\ref{lem:IVandthenII}: The one on the left is \cite[Figure~11]{BLZ}; The one on the right depicts the involved rank $1$, $2$ and $3$ fibrations.}
\label{fig:relationwithtypeIVbeforeII}
\end{figure}
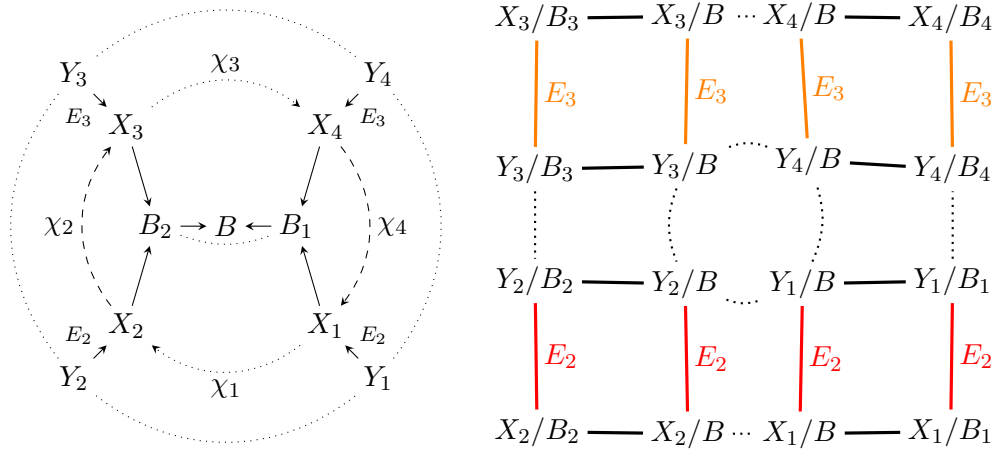

The next lemma shows that in an elementary relation with a link of type \II\, there are never two successive links of type \I~(or~\III) with equidimensional bases. It will be used in Lemma~\ref{lem:IIIandthenII}.

\begin{lem}\label{lem:IandthenI}
	Let $T/B$ be a rank $3$ fibration inducing an elementary relation \[\chi_n\circ \cdots\circ\chi_2\circ\chi_1=\id\]
	with $n\geqslant 3$ and Sarkisov links $\chi_i\colon X_i\dashrightarrow X_{i+1}$ $($where $X_{n+1}=X_1)$ between Mori fibre spaces $X_i/B_i$, for $i=1,\ldots,n$.
	Assume that $\dim(B_1)=\dim(B_2)=\dim(B_3)$, and that $\chi_1$ and $\chi_2$ are of type \I. Then, $n\geqslant 4$, $\chi_{3},\ldots,\chi_{n-2}$ are $n-4$ links of type \IV\, and both $\chi_{n-1}$ and $\chi_n$ are of type \III.
\end{lem}
\begin{proof}
	As $\chi_i$ is of type \I\, for $i\in\{1,2\}$, it comes with a divisorial contraction  $Y_i\to X_i$, that contracts a divisor $F\subseteq Y_1$ respectively $G\subseteq Y_2$.
	Completing the diagram gives
	\[\begin{tikzcd}[ampersand replacement=\&,column sep=1.3cm,row sep=0.16cm]
	T_1\ar[rr,dotted,-] \ar[dd,"G",swap]\&\&Y_2\ar[dd,"G",swap]  \ar[rr,dotted,-] \&\&  X_3\ar[dd]\\ \\
Y_1\ar[dd,"F",swap]  \ar[rr,dotted,-] \&\& X_2 \ar[dd] \ar[rruu,"\chi_2",dashed,swap]\&  \& B_3.\ar[ddll]   \\ \\
X_1 \ar[uurr,"\chi_1",dashed,swap] \ar[dd,swap] \&  \& B_2 \ar[ddll]\&  \& \\ \\
B_1\&  \&\&\&
\end{tikzcd}\]
All bases $B_1,B_2,B_3$ have the same dimension, and are $\QQ$-factorial, as these are the bases of Mori fibre spaces. Hence, $B_2\to B_1$ is a divisorial contraction, that contracts a divisor $\widehat F\subseteq B_2$. The preimage in $X_2$ of this divisor $\widehat F$ is a divisor of $X_2$; each irreducible component that surjects to $\widehat F$ is contracted by $\chi_1^{-1}$, so there is only one such component, and it is the strict transform of $F\subseteq Y_1$.
Similarly, $B_3\to B_2$ is a divisorial contraction contracting a divisor $\widehat G\subseteq B_3$, whose preimage in $X_3$ is the strict transform of $G\subseteq Y_2$.

The $2$-ray game on $B_3/B_1$ gives therefore a pseudo-isomorphism $B_3\ps B'$, and two divisorial contractions $B'\stackrel{\widehat F}\to B''\stackrel{\widehat G}\to B_1$. We decompose the pseudo-isomorphism $B_3\ps B'$ into a sequence of $r-3\geqslant 0$ log-flips, say $\psi_i\colon B_{i}\ps B_{i+1}$ for $i=3,\ldots,r-1$ with $B'=B_r$. In particular, if $B_3\ps B'$ is an isomorphism then $r=3$ and $B'=B_3$.
For each $i\in \{3,\ldots,r-1\}$, we obtain the commutative diagram
\[\begin{tikzcd}[ampersand replacement=\&,column sep=1.3cm,row sep=0.16cm]
B_{i}\ar[rr,"\psi_i",dotted] \ar[dr,"\nu_i",swap] \& \& B_{i+1} \ar[dl,"\tau_i"] \\
\& C_i \&
\end{tikzcd}\]
where $\nu_i$, $\tau_i$ are small extremal contractions and $C_i$ is a non-$\QQ$-factorial variety. Starting with $i=3$ and proceeding inductively, the $2$-ray game on $X_i/C_i$ gives the link $\chi_{i}\colon X_{i}\ps X_{i+1}$, that is of type \IV\, (see Lemma~\ref{lem:LinksIV}), with diagram
\[\begin{tikzcd}[ampersand replacement=\&,column sep=1.7cm,row sep=0.16cm]
X_i \ar[rr,"\chi_i",dotted,swap] \ar[dd,swap]  \&\& X_{i+1} \ar[dd] \\ \\
B_i\ar[dr,"\nu_i",swap] \& \& B_{i+1}.\ar[dl,"\tau_i"]  \\
\& C_i \&
\end{tikzcd}\]

This gives a sequence of $r-3$ links of type \IV\, (and thus $0$ such links if $B_3\ps B'$ is an isomorphism, where $r=3$).

In any case, we now can use the divisorial contractions $B'\stackrel{\widehat F}\to B''\stackrel{\widehat G}\to B_1$ to find the last links. The link $\chi_r$ is given by the $2$-ray game on $X_r/B''$ (via $X_r\to B_r=B'\to B''$). By looking at Figure~\ref{fig:SarkisovTypes}, only types \III, \IV~are possible, and \IV\, is excluded by Lemma~\ref{lem:LinksIV}. Hence, $\chi_r$ is a link of type \III, so $B_{r+1}=B''$, and the divisorial contraction $Y_{r+1}\to X_{r+1}$ is given by the contraction of $F$, by the same reasoning as before. This gives the commutative diagram
\[\begin{tikzcd}[ampersand replacement=\&,column sep=1.3cm,row sep=0.16cm]
T_1\ar[rr,dotted,-] \ar[dd,"G",swap]\&\&Y_2\ar[dd,"G",swap]  \ar[rr,dotted,-] \&\&  X_3\ar[dd]\ar[rr,"\chi_{r-1}\circ\cdots \circ \chi_3",dotted,-]\&\&X_r\ar[dd]\ar[rr,dotted,-]\ar[ddrr,dashed,swap,"\chi_{r}"]\&\& Y_{r+1}\ar[dd,"F",swap]\\ \\
Y_1\ar[dd,"F",swap]  \ar[rr,dotted,-] \&\& X_2 \ar[dd] \ar[rruu,"\chi_2",dashed,swap]\&  \& B_3\ar[rr,dotted,-]\ar[ddll,"\widehat G",swap]\&\&B_r\ar[ddrr,"\widehat F"]\&\& X_{r+1}\ar[dd]   \\ \\
X_1 \ar[uurr,"\chi_1",dashed,swap] \ar[dd,swap] \&  \& B_2 \ar[ddll,"\widehat F",swap]\&  \&\&\&\&\&B_{r+1}.\ar[ddllllllll,"\widehat G"] \\ \\
B_1\&  \&\&\&\&\&
\end{tikzcd}\]
Now, $\chi_{r+1}$  is given by the $2$-ray game on $X_{r+1}/B_1$, via $X_{r+1}\to B_{r+1}=B''\stackrel{\widehat G}\to B_1$, and is therefore again a link of type \III. This gives
\[\begin{tikzcd}[ampersand replacement=\&,column sep=1.0cm,row sep=0.16cm]
T_1\ar[rr,dotted,-] \ar[dd,"G",swap]\&\&Y_2\ar[dd,"G",swap]  \ar[rr,dotted,-] \&\&  X_3\ar[dd]\ar[rr,"\chi_{r-1}\circ\cdots \circ \chi_3",dotted,-]\&\& X_r\ar[dd]\ar[rr,dotted,-]\ar[ddrr,dashed,swap,"\chi_{r}"]\&\& Y_{r+1}\ar[dd,"F",swap]\ar[rr,dotted,-]\&\&T_{r+2}\ar[dd,swap,"F"]\\ \\
Y_1\ar[dd,"F",swap]  \ar[rr,dotted,-] \&\& X_2 \ar[dd] \ar[rruu,"\chi_2",dashed,swap]\&  \& B_3\ar[rr,dotted,-]\ar[ddll,"\widehat G",swap]\&\& B_r\ar[ddrr,"\widehat F"]\&\& X_{r+1}\ar[dd]\ar[rr,dotted,-]\ar[ddrr,dashed,swap,"\chi_{r+1}"]\&\&Y_{r+2}\ar[dd,swap,"G"]   \\ \\
X_1 \ar[uurr,"\chi_1",dashed,swap] \ar[dd,swap] \&  \& B_2 \ar[ddll,"\widehat F",swap]\&  \&\&\&\&\& B_{r+1}\ar[ddrr,"\widehat G"]\&\&X_{r+2}\ar[dd] \\ \\
B_1\ar[rrrrrrrrrr,equal]\&\&\&\&\&\&\&\&\&\& B_1.
\end{tikzcd}\]
Applying the $2$-ray game on $T_{r+2}\to Y_{r+2}\to X_{r+2}$, we exchange the two divisors and obtain $T_1\to Y_1\to X_1=X_{r+2}$. This then finishes the relation with $n=r+1$.
\end{proof}

\begin{lem}\label{lem:IIIandthenII}
	Let $T/B$ be a rank $3$ fibration inducing an elementary relation \[\chi_n\circ \cdots\circ \chi_2\circ\chi_1=\id\]
	with $n\geqslant 3$ and Sarkisov links $\chi_i\colon X_i\dashrightarrow X_{i+1}$ $($where $X_{n+1}=X_1)$ between Mori fibre spaces $X_i/B_i$, for $i=1,\ldots,n$.
	Assume that $\dim(B)=\dim(B_1)=\dim(B_2)$, that $\chi_1$ is of type \III\, and that $\chi_2$ is of type \II. Then, $n=4$, $\chi_3$ is of type \I, $\chi_4$ is of type \II\ and is equivalent to  $\chi_2^{-1}$.
\end{lem}
\begin{proof}
As $\chi_1$ is of type \III~and $\chi_2$ is of type \II, we have the commutative diagram
\[\begin{tikzcd}[ampersand replacement=\&,column sep=1.3cm,row sep=0.16cm]
	X_1 \ar[ddrr,"\chi_1",dashed,swap]\ar[r,dotted,-] \ar[dd]\& Y_2'\ar[rdd,"G"]\&Y_2\ar[dd,"E_2"]  \ar[rr,dotted,-] \&\& \ar[dd,"E_3"] Y_3\\ \\
B_1\ar[rrdd]\&\& X_2 \ar[dd] \ar[rr,"\chi_2",dashed,swap]\&  \& X_3\ar[dd] \\ \\
\&  \& B_2\ar[rr,equal]\&  \& B_3
\end{tikzcd}\]
where $Y_i\to X_i$ is a divisorial contraction with exceptional divisor $E_i$, for $i\in \{2,3\}$, and where $Y_2'\to X_2$ is a divisorial contraction with exceptional divisor $G$. We now complete the diagram, using the two observations at the beginning of the section:
\[\begin{tikzcd}[ampersand replacement=\&,column sep=1.3cm,row sep=0.16cm]
 T_1   \ar[dd,"E_2"]\ar[r,dotted,-] \& T_2'  \ar[dd,"E_2"]\ar[r,dotted,-] \& T_2\ar[dd,"G"]\ar[rr,dotted,-]  \&\& T_3\ar[dd,"G"]\\ \\
	X_1 \ar[ddrr,"\chi_1",dashed,swap]\ar[r,dotted,-] \ar[dd]\& Y_2'\ar[rdd,"G"]\&Y_2\ar[dd,"E_2"]  \ar[rr,dotted,-] \&\& \ar[dd,"E_3"] Y_3\\ \\
B_1\ar[rrdd]\&\& X_2 \ar[dd] \ar[rr,"\chi_2",dashed,swap]\&  \& X_3\ar[dd] \\ \\
\&  \& B_2\ar[rr,equal]\&  \& B_3.
\end{tikzcd}\]
If $\chi_n$ is of type \II, we may apply Lemma~\ref{lem:IandthenII} to the elementary relation \[\chi_2^{-1}\circ\cdots \circ\chi_{n-1}^{-1}\circ\chi_n^{-1}\circ \chi_{1}^{-1}=\id.\]
Indeed, $\chi_1^{-1}$ is of type \II, $\chi_n^{-1}$ also, and $\dim(B)=\dim(B_1)=\dim(B_n)$. As $\chi_2$ is of type \II,   Lemma~\ref{lem:IandthenII} implies that $n=2$ and that $\chi_2^{-1}$ is equivalent to $\chi_4$, so $\chi_4^{-1}$ is equivalent to $\chi_2$.

It remains to assume that $\chi_n$ is not of type \II, and to derive a contradiction. As $\chi_n^{-1}$ is obtained by applying a $2$-ray game to $T_1/B_1$, it is either of type \II~or of type \I~(Figure~\ref{fig:SarkisovTypes}). If $\chi_n^{-1}$ is of type \I, then $\chi_1^{-1}$ and $\chi_n^{-1}$ are two consecutive links of type \I~in an elementary relation that contains a link of type \II, impossible by Lemma~\ref{lem:IandthenI}.
\end{proof}

\section{Mori fibre spaces whose generic fibre is a surface}

A Sarkisov link $\chi\colon X_1/B\dasharrow X_2/B$ of type \II\ induces a birational map between the generic fibres of $X_1/B$ and $X_2/B$, which are Fano varieties of Picard rank $1$ over the field $\CC(B)$.
We will now focus on the case where these generic fibres are surfaces, and then apply it to the case of non-trivial Severi-Brauer surfaces, by using the results of Section~\ref{BirSB}.

\subsection{From surfaces over $\CC(B)$ to higher dimensional Mori fibre spaces over $\CC$}

The following result shows how Sarkisov links between higher dimensional Mori fibre spaces over $\CC$ induce Sarkisov links between the generic fibres.

\begin{lem}\label{lem:SBMfsSarki}
Let $\chi\colon X_1\dasharrow X_2$ be a Sarkisov link of type \II\ over a base $B$. For $i\in \{1,2\}$,  let $\Gamma_i\subseteq X_i$ be the centre of the divisorial contraction $Y_i\to X_i$ associated to the Sarkisov link

\[\begin{tikzcd}[ampersand replacement=\&,column sep=.8cm,row sep=0.16cm]
Y_1\ar[dd,"\rm div",swap]  \ar[rr,dotted,-] \&\& \ar[dd,"\rm div"] Y_2 \\ \\
X_1 \ar[rr,"\chi",dashed,swap] \ar[dr,"\rm fib",swap] \&  \& X_2. \ar[dl,"\rm fib"] \\
\& B \&
\end{tikzcd}\]
Suppose that the generic fibres $S_1$ and $S_2$ of $X_1/B$ and $X_2/B$ are surfaces. Let  $\varphi\colon S_1\dasharrow S_2$ be the birational map over $\CC(B)$ induced by $\chi$. Then exactly one of the following holds:
\begin{enumerate}
\item\label{SBSar1}
For each $i\in \{1,2\}$, the morphism $\Gamma_i\to B$ is not dominant, and $\varphi$ is an isomorphism $S_1\iso S_2$;
\item\label{SBSar2}
There is $d_1,d_2\in\{1,\ldots,8\}$ such that for each $i\in \{1,2\}$, the morphism $\Gamma_i\to B$ is a surjective, generically $(d_i:1)$-map. Moreover, $\varphi$ is a $(d_1,d_2)$-link, as in Definition~$\ref{Def36links}$.
\end{enumerate}
\end{lem}
\begin{proof}For $i\in \{1,2\}$,  $X_i/B$ is a Mori fibre space, so $S_i$ is a del Pezzo surface over $\CC(B)$, with Picard rank $1$. Moreover, as $\chi$ is a birational map over $B$, it corresponds to a unique birational map $\varphi\colon S_1\dasharrow S_2$, defined over $\CC(B)$.

Recall that $\chi$ is obtained by first blowing-up $\Gamma_1$, then applying a pseudo-isomorphism (which induces an isomorphism between the generic fibres) and then a divisorial contraction $Y_2\to X_2$. Hence, $\varphi$ is obtained by the composition of the three induced maps on the generic fibres. If $\Gamma_1\to B$ is not dominant, the first two are isomorphisms, and the last one is a birational morphism, which is then necessarily an isomorphism too, as $\rk\Pic(X_2/B)=1$. This gives~\ref{SBSar1}. If $\Gamma_1\to B$ is dominant, it is a surjective, generically $(d_1:1)$-map, for some integer $d_1\geqslant 1$. Hence, the first map for the generic fibres is the blow-up of a $d_1$-point. As $Y_1/B$ is a rank $2$ fibration, the generic fibre of $Y_1/B$ is again a del Pezzo surface, so we have $d_1\leqslant 8$. Applying the argument to $\chi^{-1}$, the last one needs then to be a divisorial contraction, and we obtain that $\varphi $ is a $(d_1,d_2)$-link, as in Definition~\ref{Def36links}. This gives~\ref{SBSar2}.
\end{proof}

We now consider equivalence between Sarkisov links $\chi,\chi'$ of type \II, as in Lemma~\ref{lem:SBMfsSarki}. As we will see in Lemma~\ref{Lem:EqInduceWeakEq} below, this is related to the following new equivalence on Sarkisov links between del Pezzo surfaces, which relates links defined over different but isomorphic field extensions of $\CC$.

\begin{mydef}\label{Def:weakEq}
Let $K,K'$ be two finitely generated field extensions of $\CC$. Let $S_1,S_2$ (respectively $S_1',S_2'$) be two del Pezzo surfaces of Picard rank $1$ over $K$ (respectively $K'$). Let $\varphi\colon S_1\dasharrow S_2$ (respectively $\varphi'\colon S_1'\dasharrow S_2'$) be a birational map, defined over $K$ (respectively $K'$).
We say that $\varphi$ and $\varphi'$ are \emph{weakly equivalent} if there are scheme isomorphisms $\psi_i\colon S_i\iso S_i'$ for $i=1,2$, and a field isomorphism $f^*\colon K'\iso K$ fixing $\CC$, inducing an isomorphism $f\colon \Spec(K)\iso \Spec(K')$, such that the following diagram commutes:

\[\begin{tikzcd}[link]
	S_1\ar[dddd,swap]\ar[rrr,"\psi_1"]\ar[ddr,dashed,"\varphi",swap]&&& S'_1\ar[ddr,dashed,"\varphi'"]\ar[dddd]  \\ \\
	&S_2\ar[ddl,swap]\ar[rrr,"\psi_2"] &&& S'_2\ar[ddl] \\ \\
	\Spec(K) \ar[rrr,"f"]&&& \Spec(K').
\end{tikzcd}\]
\end{mydef}

\begin{rem}\label{Rem:WeaklyBaseChange}
In Definition~\ref{Def:weakEq}, applying the base change induced by $f^*\colon K'\to K$ onto the Sarkisov link $\varphi\colon S_1\dashrightarrow S_2$ over $\Spec(K)$ gives a Sarkisov link $\varphi^f\colon S_1^f\dashrightarrow S_2^f$ over $\Spec(K')$.
Then, $\varphi$ is weakly equivalent to $\varphi'$ if and only if one can find such an $f$ with $\varphi^f$ and $\varphi'$ equivalent (in the sense of Definition~\ref{def:equiLinkDim2}). This also shows that the weak equivalence of Definition~\ref{Def:weakEq} is indeed weaker than the one of Definition~\ref{def:equiLinkDim2}. If $K=K'$, the equivalence of $\varphi$ and $\varphi'$ implies the weak equivalence but the converse does not hold, as the following example shows.
\end{rem}

\begin{ex}\label{Example:LinkEquiInverse}
Let $K=\CC(t_1,t_2)$ and let $L=K[\sqrt[3]{t_1}]$. The Galois group  $\mathrm{Gal}(L/K)$ is generated by $g\colon \sqrt[3]{t_1}\mapsto \zeta\sqrt[3]{t_1}$.  We denote by $S,S'$ the Severi-Brauer surfaces over $K$ that admit isomorphisms $\varphi\colon S_L\iso \PP^2_L$ and  $\varphi'\colon (S')_L\iso \PP^2_L$ such that
\[\varphi\circ g\circ \varphi^{-1}\colon[x:y:z]\mapsto [g(z)t_2 :g(x):g(y)],\ \ \varphi'\circ g\circ \varphi'^{-1}\colon[x:y:z]\mapsto [g(z)/t_2 :g(x):g(y)].\]
The existence of $S,S'$ is given by Lemma~\ref{lem:SBdegree3}\ref{SB31}. Moreover, $S$ is a non-trivial Severi-Brauer surface by Corollary~\ref{cor: non-trivial SB over function field}, and $S'\simeq S^{\mathrm{op}}$ by Lemma~\ref{Lem:SxiOp}. As in the proof of Lemma~\ref{Lem:SxiOp}, we consider the birational involution $\sigma\in \Bir_L(\mathbb{P}^2)$ given by $[x:y:z]\mapsto [yz:xz:xy]$, and check that $\tau=(\varphi')^{-1}\circ \sigma \circ \varphi\colon S\dasharrow S'$ is defined over $K$ and is a $3$-link.
Let $f\in\Aut_K(L)$ be the involutive field isomorphism fixing $\sqrt[3]{t_1}$ and sending $t_2$ onto $t_2^{-1}$. As the base change induced by $f$ gives $(\varphi\circ g\circ \varphi^{-1})^f= \varphi'\circ g\circ \varphi'^{-1}$, the isomorphism $\varphi'^{-1}\circ \varphi^f\colon S^f\iso S$ is defined over $K$. Similarly, $\varphi^{-1}\circ \varphi'^f\colon S'^f\iso S$ is an isomorphism defined over $K$. We obtain the commutative diagram
\[
		\xymatrix@R=6pt@C=20pt{
			S^f\ar[d]_{\varphi'^{-1}\circ \varphi^f}\ar@{-->}[rr]^{\tau^f} && S'^f\ar[d]^{\varphi^{-1}\circ \varphi'^f}\\
			S'\ar@{-->}[rr]^{\tau^{-1}} && S
		}\]
that shows that $\tau$ is weakly equivalent to its inverse. On the other hand, $\tau$ and $\tau^{-1}$ are not equivalent, as $S$ and $S'$ are not isomorphic over $K$.
\end{ex}

\begin{mydef}\label{Def:inducedBy}
Let $X/B$ be a Mori fibre space, whose generic fibre is a del Pezzo surface~$S$. Consider a Sarkisov link $\chi\colon X_1\dasharrow X_2$ of type \II~between Mori fibre spaces $X_1/\widetilde{B}$ and $X_2/\widetilde{B}$ and assume that there exists a commutative diagram \[
	\xymatrix@R=6pt@C=20pt{
		X\ar[d]\ar@{-->}[rr]^{\widehat\psi} &&X_1\ar[dr]\ar@{-->}[rr]^{\chi} && X_2,\ar[dl]\\
		B \ar@{-->}[rrr]^{\psi} &&   &\widetilde{B}&
	} \]
where $\psi,\widehat\psi$ are birational maps.
Let $\varphi\colon S_1\dashrightarrow S_2$ be the birational map over $\CC(\widetilde B)$ induced by $\chi$ between the generic fibres $S_i$ of $X_i/\widetilde B$ for $i=1,2$.
Note that the base-change $f\colon\Spec(\CC(B))\to\Spec(\CC(\widetilde{B}))$ induced by~ $\psi$ gives a birational map $\alpha=\varphi^f\colon S_1^f\dasharrow S_2^f$ over $\CC(B)$, which is contained in $\BirMori(S)$ as $S_1^f$ is birational to $S$ over $\CC(B)$, via the above birational map $\widehat\psi$ and the base-change $f$.
We will say that $\alpha\in\BirMori(S)$ is \emph{induced by $\chi$}.
\end{mydef}

In the above definition, the birational map $\varphi^f$ depends on the choice of $\psi$, but its class under weak equivalence does not. The next lemma immediately follows from the definitions and the fact that birational maps on the base correspond to field isomorphisms.

\begin{lem}\label{Lem:EqInduceWeakEq}
Let $X\to B$ be a Mori fibre space whose generic fibre is a del Pezzo surface $S$. We consider two birational maps $\alpha,\alpha'\in\BirMori(S)$ induced by two Sarkisov links $\chi\colon X_1/\widetilde{B}\dasharrow X_2/\widetilde{B}$ and $\chi'\colon X_1'/\widetilde{B}'\dasharrow X_2'/\widetilde{B}'$, as in Definition~$\ref{Def:inducedBy}$. Then, $\chi$ and $\chi'$ are equivalent $($Definition~$\ref{Def:equiLink})$, if and only if $\alpha$ and $\alpha'$ are weakly equivalent $($Definition~$\ref{Def:weakEq})$.
\end{lem}
\begin{proof}
	For $i\in \{1,2\}$ we denote by $S_i$ the generic fibre of $X_i/\widetilde{B}$, and by $S_i'$ the generic fibre of $X_i'/\widetilde{B}'$. We then denote by $\varphi\colon S_1\dasharrow S_2$ and $\varphi'\colon S_1'\dasharrow S_2'$ the birational maps induced by $\chi$ and $\chi^{-1}$.
	As $\alpha$ is obtained from $\varphi$ by some base-change, it is weakly equivalent to $\varphi$. Similarly,  $\alpha'$ is weakly equivalent to $\varphi'$. We then find that $\alpha$ and $\alpha'$ are weakly equivalent if and only if $\varphi$ and $\varphi'$ are weakly equivalent. This corresponds to asking that $\chi$ and $\chi'$ are equivalent. Indeed, the
	 diagram of Definition~\ref{Def:equiLink} restricted to the generic fibres corresponds to  the diagram of Definition~\ref{Def:weakEq}.
	\end{proof}

We now explain how to extend homomorphisms from $\BirMori(S)$ to homomorphisms from $\BirMori(X)$, where $S$ is a generic fibre of a Mori fibre space $X/B$, in the case when $S$ is a del Pezzo surface not birational to a conic bundle.

\begin{thm}\label{Theorem:DPMfs}
Let $X\to B$ be a Mori fibre space, whose generic fibre is a del Pezzo surface $S$ over the field $\CC(B)$, not birational to a conic bundle.  There is an integer $g\geqslant 1$, depending only on $\dim(X)$, such that the following holds:

Assume there is a group $G$ and groupoid homomorphism $\eta\colon \BirMori(S)\to G$, that sends any two weakly equivalent links of type \II\ onto the same image, and such that each birational map in $\BirMori(S)$ induced by a Sarkisov link of type \II~in $\BirMori(X)$ having covering genus $\leqslant g$ is sent to the identity.

Then, there exists a unique groupoid homomorphism $\rho\colon \BirMori(X)\to G$ that satisfies the following: each Sarkisov link of type \II~that induces  an element $\alpha$ of $\BirMori(S)$ is sent onto $\eta(\alpha)$, any isomorphism of Mori fibre spaces and any other Sarkisov link is sent to the trivial  element. Moreover, $\rho|_{\Bir(X/B)}=\eta|_{\Bir(S)}$.
\end{thm}

\begin{proof}
Observe that since two weakly equivalent Sarkisov links in $\BirMori(S)$ have the same image under $\eta$, the homomorphism $\eta$ sends any isomorphism of Mori fibre spaces onto the trivial word. Indeed, the composition of any Sarkisov link with an isomorphism is a link equivalent to the first one.

By Theorem~\ref{thm: sarkisov}, the groupoid $\BirMori(X)$ is generatedy by Sarkisov links and isomorphisms, and any  relation between Sarkisov links in $\BirMori(X)$ is generated by elementary relations and trivial relations. We send every isomorphism of Mori fibre spaces onto the trivial element, send any Sarkisov link to either the trivial element or to the image defined by $\eta$, as explained above, and check that every trivial relation and every elementary relation is sent to the identity.

For a trivial relation $\chi\circ \chi^{-1}=\mathrm{id}$, either both $\chi$ and $\chi^{-1}$ are sent to the trivial element, or $\chi$ induces an element $\alpha\in \BirMori(S)$ and so $\chi$ is a link of type \II\ over a base $\widetilde B$ birational to $B$. In this latter case, $\chi^{-1}$ induces $\alpha^{-1}$. This follows from Definition~\ref{Def:inducedBy}, by taking the same $\psi\colon B\dasharrow \widetilde{B}$ for both $\chi$ and $\chi^{-1}$. For a trivial relation of the form $\alpha'\circ \chi\circ \alpha=\chi'$, either all elements are sent to the trivial word, or one of the two links $\chi$ and $\chi'$ is a Sarkisov link of type \II, that induces an element $\alpha\in \BirMori(S)$. But then the other link also induces the same element, as $\chi$ and $\chi'$ are equivalent (Lemma~\ref{Lem:EqInduceWeakEq}). Hence, both have the same image, and the relation is sent to the trivial word, as $\rho(\alpha)=\rho(\alpha')$ is trivial.

We now consider elementary relations. If they only involve links sent to the trivial word, we are done. So we should only consider rank $3$ fibrations $T/B_T$ that factorise through a rank $2$ fibration $X'/B'$ that gives rise to a Sarkisov link $\chi$ of type \II, inducing an element $\alpha\in \BirMori(S)$ with $\rho(\alpha)$ not trivial. By Lemma~\ref{lem:SBMfsSarki}, $\alpha$ is a Sarkisov link, as  isomorphisms in $\BirMori(S)$ are in the kernel of $\eta$.

If $\dim(T)-\dim(B_{T})\geqslant 3$, then Proposition~\ref{prop: covgen link bound} bounds the covering genus of $\chi$, by some number only depending on the dimension of $T$, and thus of $X$. Taking $g$ large enough, we may thus avoid such relations, and only restrict ourselves to the case where $\dim(T)-\dim(B_{T})\leqslant 2$, and thus to $\dim(T)-\dim(B_{T})=2$ since $\dim(X')-\dim(B')=2$. Coming from $T/B_{T}$, we have an elementary relation $\chi_n\circ \cdots\circ\chi_2\circ\chi_1=\id$, $n\geqslant 3$, and Sarkisov links $\chi_i\colon X_i\dashrightarrow X_{i+1}$ $($where $X_{n+1}=X_1)$ between Mori fibre spaces $X_i/B_i$, for $i=1,\ldots,n$. We then distinguish some cases:

{\bf (A)} Suppose first that all links $\chi_1,\ldots,\chi_n$ in the elementary relation are of type \II. Thus, all links have the same base: $B_1=B_2=\cdots=B_n$. Let $K=\CC(B_1)=\cdots=\CC(B_n)$. For each $i\in \{1,\ldots,n+1\}$, we denote by $S_i$ the generic fibre of $X_i\to B_i$, which is a surface defined over $K$ (and have $S_{n+1}=S_1$ since $X_{n+1}=X_1$) and is a del Pezzo surface of Picard rank~$1$.
The Sarkisov link $\chi_i\colon X_i\dasharrow X_{i+1}$ gives rise to a birational map $\varphi_i\colon S_i\dasharrow S_{i+1}$ defined over $K$, that is either a Sarkisov link or an isomorphism (Lemma~\ref{lem:SBMfsSarki}). As one of the Sarkisov links $\chi_i$ induces an element of $\BirMori(S)$, there are birational maps $\widehat\psi\colon X\dasharrow X_i$, $\psi\colon B\dasharrow B_i$, compatible with $X/B$ and $X_i/B_i$, as in Definition~\ref{Def:inducedBy}. The element of $\BirMori(S)$ induced by $\chi_i$  is then $(\varphi_i)^f$, where $f\colon\Spec(\CC(B))\to \Spec(\CC(B_i))$ comes from $\psi$. For each of the links $\chi_1,\ldots,\chi_n$, we may take the same $\psi$, and obtain elements $(\varphi_1)^f,\ldots,(\varphi_n)^f$ in $\BirMori(S)$ whose composition is the identity. This implies that $\rho(\chi_n)\cdot \cdots \rho(\chi_1)$ is the trivial element of $G$, as we wanted to prove.

{\bf (B)} Suppose now that at least one of the links $\chi_1,\ldots,\chi_n$ is not of type \II. Recall that one link $\chi_i$ is of type \II~and induces an element $\alpha\in \BirMori(S)$. Making a cyclic permutation, we may assume that $i\in \{2,\ldots,n\}$, that $\chi_2,\chi_3,\ldots,\chi_i$ are of type \II\ and that $\chi_1$ is of type \I, \III~or \IV. With this choice, $B_2=B_3=\cdots=B_i$ have dimension $\dim B$.
We now show that $n=4$, that $\chi_4$ is equivalent to $\chi_2^{-1}$, and that $\chi_1$ and $\chi_3$ are not of type \II. This implies then that the relation is sent onto the trivial word (by Lemma~\ref{Lem:EqInduceWeakEq}). Moreover, it implies that $i=2$.

If $\chi_1$ is of type \I, then Lemma~\ref{lem:IandthenII} proves that $\chi_3$ is of type \III, so $i=2$,~and  gives two possibilities. Case~\ref{IandthenII4} gives $n=4$ and $\chi_4$ is a link of type \II\ that is equivalent to  $\chi_2^{-1}$. Case~\ref{IandthenII3} is only possible if the centres of the divisorial contractions $Y_2\to X_2$ and $Y_3\to X_3$ associated to $\chi_2$ are contained in the preimage of a uniruled divisor of $B_2$, impossible by Lemma~\ref{lem:SBMfsSarki}, as the centres of $\chi_2=\chi_i$ are surjective to the bases $B_2=B_3$. Hence, Case~\ref{IandthenII3} never occurs.

If $\chi_1$ is of type \III, we find that $\dim(B_1)\geqslant \dim(B_2)=\dim(B)$.
To rule out the case $\dim(B_1)>\dim(B_2)$ we now use our assumption that $S$ is not birational to a conic bundle:
The generic fibre of $X_1/B_2$ and $X_2/B_2$ would be birational over $\CC(B_2)$, and the generic fibre of $X_1/B_2$ would admit a morphism to the generic fibre of $B_1/B_2$, which would then be a conic bundle, impossible by assumption. Hence, $\dim(B_1)=\dim(B_2)$ and we can apply Lemma~\ref{lem:IIIandthenII}, which gives $n=4$, implies that $\chi_3$ is of type \I, $\chi_4$ is of type \II\ and is equivalent to  $\chi_2^{-1}$. The remaining case is when $\chi_1$ is of type \IV. As before, we get $\dim(B_1)\geqslant \dim(B_2)$, and the equality is necessary by the same argument. We then apply Lemma~\ref{lem:IVandthenII}, obtain that $n=4$, $\chi_3$ is of type \IV, and that $\chi_4$ is of type \II\ and is equivalent to  $\chi_2^{-1}$.
\end{proof}

\subsection{Severi-Brauer Mori fibre spaces and higher $d$-links}
We now focus on the case where the generic fibres of our Mori fibre spaces are non-trivial Severi-Brauer surfaces.

\begin{mydef}
We say that a Mori fibre space $X/B$ is a \emph{Severi-Brauer Mori fibre space} (SBMfs for short) if the generic fibre $S$ of $X/B$ is a non-trivial Severi-Brauer variety over $\CC(B)$.
\end{mydef}

In this paper, we will only consider Severi-Brauer Mori fibre spaces whose generic fibre is a non-trivial Severi-Brauer \emph{surface} $S$.

\begin{mydef}\label{def: 6-link}
	Let $\chi\colon X_1\dasharrow X_2$ be a Sarkisov link of type \II, over a base~$B$ and let $d_1,d_2\geqslant 1$ be integers. We say that $\chi$ is a \emph{$(d_1,d_2)$-link} (or \emph{$d$-link} if $d=d_1=d_2$) if $X_1/B$ and $X_2/B$ are Mori fibre spaces whose generic fibres are del Pezzo surfaces, and if the birational map $\varphi\colon S_1\dasharrow S_2$ induced by $\chi$ on the generic fibres $S_i$ of $X_i/B$ is a $(d_1,d_2)$-link, as in Definition~\ref{Def36links}.
\end{mydef}

\begin{rem}
As Example~\ref{Example:LinkEquiInverse} and Lemma~\ref{Lem:EqInduceWeakEq} show, a $d$-link $\chi$ between SBMfs can be equivalent to its inverse, contrary to the case of links between Severi-Brauer surfaces (Remark \ref{rem: link is not equivalent to inverse}).
\end{rem}

The following is now a consequence of Theorem~\ref{thm:SBgroupoidhomo} and Theorem~\ref{Theorem:DPMfs}.

\begin{thm}\label{Theorem:SBMfs}Let $X\to B$ be a SBMfs, whose generic fibre is a non-trivial Severi-Brauer surface. There is an integer $g\geqslant 1$, depending only on $\dim(X)$, such that the following holds:

For each $d\in\{3,6\}$, let $\mathcal{M}_d$ be a set of $d$-links in $\BirMori(X)$ with covering genus at least $g$ such that every $\chi\in \mathcal{M}_d$ induces an element of $\BirMori(S)$ as in Definition~$\ref{Def:inducedBy}$, and is not equivalent to its inverse $\chi^{-1}$, nor to $\widetilde{\chi}$ and $\widetilde{\chi}^{-1}$ for each $\widetilde{\chi}\in \mathcal{M}_d\setminus \{\chi\}$.

Then, there is a surjective groupoid homomorphism
\[
\rho\colon \BirMori(X)\to \bigoplus\limits_{\mathcal{M}_3} \ZZ/3\ZZ \ast \left ( \bigast_{\mathcal{M}_6} \mathbb{Z} \right),
\]
that is defined by the following rule. One sends each isomorphism between Mori fibre spaces and each Sarkisov link $\chi$ that is not a $3$-link or a $6$-link onto the trivial word. For $d\in \{3,6\}$, we send a $d$-link $\chi$ onto
\begin{align*}
	\rho(\chi)=\begin{cases}
		1_{\chi'} &  \text{ if }\chi\text{ is equivalent to }\chi'\in\mathcal{M}_d\\
		-1_{\chi'} & \text{ if }\chi^{-1}\text{ is equivalent to }\chi'\in\mathcal{M}_d\\
		0 & \text{ if neither }\chi\text{ nor }\chi^{-1}\text{ is equivalent to any element of }\mathcal{M}_d.
	\end{cases}
\end{align*}
Furthermore, the above induces surjective group homomorphisms  $\Bir(X)\to \bigoplus\limits_{\mathcal{M}_3\setminus \{\chi\}} \ZZ/3\ZZ \ast \left ( \bigast\limits_{\mathcal{M}_6} \mathbb{Z} \right)$ for each $\chi\in \mathcal{M}_3$ and $\Bir(X)\to \bigoplus\limits_{\mathcal{M}_3} \ZZ/3\ZZ \ast \left ( \bigast\limits_{\mathcal{M}_6\setminus \{\chi\}} \mathbb{Z} \right)$ for each $\chi\in \mathcal{M}_6$.
\end{thm}

\begin{proof}Choose the integer $g$ as in Theorem~\ref{Theorem:DPMfs} and apply Theorem~\ref{thm:SBgroupoidhomo} to the generic fibre $S$ of $X\to B$,  that is a non-trivial Severi-Brauer surface over $K=\CC(B)$. We obtain a groupoid homomorphism
\[
	\Psi\colon \BirMori_K(S)\to \bigoplus\limits_{\mathcal{E}_3} \ZZ/3\ZZ \ast \left ( \bigast_{\mathcal{E}_6} \mathbb{Z} \right)
\]
where $\mathcal{E}_d$ is a set of representatives of the equivalence classes of $d$-links $S\dashrightarrow S^{\rm{op}}$, for $d\in \{3,6\}$.
By assumption, every element of $\mathcal{M}_3\cup\mathcal{M}_6$ induces an element $\alpha$ or $\alpha^{-1}$ in $\BirMori(S)$, as in Definition~\ref{Def:inducedBy}, with $\alpha\in\mathcal{E}_3\cup \mathcal{E}_6$ (Lemma~\ref{Lem:EqInduceWeakEq}).
As no two distinct elements of $\mathcal{M}_d$ are equivalent, the two sets $\mathcal{M}_3$ and $\mathcal{M}_6$ are naturally bijective to subsets of $\mathcal{E}_3$ and $\mathcal{E}_6$ (again by Lemma~\ref{Lem:EqInduceWeakEq}), and we again denote these subsets by $\mathcal{M}_3$ and $\mathcal{M}_6$. Projecting on these sets, we obtain a groupoid homomorphism
\[
\Psi'\colon \BirMori_K(S)\to \bigoplus\limits_{\mathcal{M}_3} \ZZ/3\ZZ \ast \left ( \bigast_{\mathcal{M}_6} \mathbb{Z} \right).
\]
It remains to apply Theorem~\ref{Theorem:DPMfs} to obtain the groupoid homomorphism \[\rho\colon \BirMori(X)\to \bigoplus\limits_{\mathcal{M}_3} \ZZ/3\ZZ \ast \left ( \bigast_{\mathcal{M}_6} \mathbb{Z} \right),\] which has the description above, by definition.

Finally, remark that the surjectivity of $\rho$  directly follows from the definition of the groupoid homomorphism, and that the group homomorphisms from $\Bir(X)$ are just obtained by projecting (so we remove the factor $\chi$), so their surjectivity are achieved by taking the link $\chi$ between two Mori fibre spaces with generic fibres $S$ and $S^{\rm op}$ and composing it with links of $\mathcal{M}_3$ and  $\mathcal{M}_6$ and isomorphisms between generic fibres over some birational map between the bases.
\end{proof}

\subsection{Constructing Severi-Brauer surface bundles}

A general approach to Severi-Brauer surface bundles was developed by Takashi Maeda in \cite{MaedaModels}. Maeda's standard models of such bundles may be viewed as a generalization of Sarkisov's well-known results on standard conic bundles associated to twisted forms of $\PP^1$. This construction is not very easy to handle in practice. Moreover, since we are interested in applications to Cremona groups, we would like to have control over rationality of our Severi-Brauer surface bundles. This is known to be a subtle issue: for example, in \cite{KreschTschinkel} A.~Kresch and Yu.~Tschinkel construct Severi-Brauer surface bundles over a rational surface which are not even stably rational. Part \ref{StoSBMfs} of the next result also follows from \cite[Page 14, Theorem]{MaedaModels}. We give a proof as this one is quite short.
\begin{prop}\label{prop:SBtoSBMfs}
Let $B$ be a smooth projective complex variety and let $S$ be a non-trivial Severi-Brauer surface over $K=\CC(B)$.
\begin{enumerate}
\item\label{StoSBMfs}
There is a birational morphism $\widehat{B}\to B$, and a SBMfs $X\to \widehat{B}$ such that the generic fibre of $X\to B$ is $S$.
\item\label{LinkStoSBMfs}
Let $\tau\colon S\dasharrow S^{\rm op}$ be a $d$-link, for $d\in \{3,6\}$.
There is a birational morphism $\widehat{B}\to B$, two SMBfs $Y\to \widehat{B}$ and $Y'\to \widehat{B}$, so that the  generic fibres of $Y\to B$ and $Y'\to B$ are isomorphic to $S$ and $S^{\rm op}$ respectively, and a $d$-link $\chi\colon Y/\widehat{B}\dasharrow Y'/\widehat{B}$, which is a Sarkisov link of type \II, such that the action of the $B$-birational map $\chi \colon Y\dasharrow Y'$ on the generic fibres is $\tau$.
\end{enumerate}
\end{prop}
\begin{proof}\ref{StoSBMfs}: We view $S$ as a closed subvariety of $\PP^m_K$ for some $m\geqslant 3$. The equations of $S$ are then homogeneous polynomials with coefficients in $K=\CC(B)$. Taking the same equations (by multiplying the denominators) in $\PP^m\times B$, we obtain a complex projective variety $Y\subseteq \PP^m\times B$, such that the generic fibre of the projection $Y\to B$ is isomorphic to $S$. We then consider a resolution $\widetilde{Y}\to Y$ of the singularities of $Y$. As $S$ is smooth, we may choose this resolution such that the generic fibre of $\widetilde{Y}\to Y\to B$ is again isomorphic to $S$. We then apply a relative MMP over $B$ and find a Mori fibre space $X\to \widehat{B}$, where $\widehat{B}\to B$ is a birational morphism, such that the generic fibre $S'$ of $X\to B$ is birational to $S$. As $S\to S'$ is a sequence of contractions and the Picard rank of $S$ is $1$, we find that $S'\simeq S$.

\ref{LinkStoSBMfs}: Applying \ref{StoSBMfs} we find two birational morphisms $A\to B$ and $C\to B$, two SBMfs $W\to A$ and $W'\to C$ such that the generic fibres of $W\to B$ and $W'\to B$ are $S$ and $S^{\rm op}$ respectively. Since $\tau\colon S\dasharrow S^{\rm op}$ is a birational map over $K$, there is a unique birational map $\psi\colon W\dasharrow W'$ over $B$ which induces the birational map $\tau\colon  S\dasharrow S^{\rm op}$ on the generic fibres. We now apply the Sarkisov program over $B$ as in \cite{HMcK} and decompose $\psi$ as $\psi=\chi_n\circ \cdots\circ \chi_1$, where $\chi_i\colon W_i/B_i\dasharrow W_{i+1}/B_{i+1}$ is a Sarkisov link between Mori fibre spaces $W_i/B_i$ and $W_{i+1}/B_{i+1}$ for each $i\in \{1,\ldots,n\}$, compatible with morphisms $B_i\to B$, and where $B_1=A$, $B_{n+1}=C$ and $W=W_1$, $W'=W_{n+1}$.

For each $i\in \{1,\ldots,n+1\}$,  the generic fibre $S_i$ of $W_i\to B$ is a surface that is birational to $S$, via the $B$-birational map $W\dasharrow W_i$. First observe that $\dim(B_i)=\dim(B)$ for each $i\in \{1,\ldots,n+1\}$. Indeed, as we have a dominant morphism $B_i\to B$, we have $\dim(B_i)\geqslant \dim(B)$, and $\dim(B_i)<\dim(W_i)=\dim(B)+2$, so $\dim(B_i)\in \{\dim(B),\dim(B)+1\}$. Moreover, if $\dim(B_i)=\dim(B)+1$, then $S_i$ would admit a morphism of relative Picard rank $1$ to the generic fibre of $B_i\to B$ (a curve), and thus would be a conic bundle birational to $S$, impossible by Lemma~\ref{LinkSop}. Since $\dim(B_i)=\dim(B)$, the morphism $B_i\to B$ is birational. Hence, $S_i$ is isomorphic to the generic fibre $S_i$ of $W_i/B_i$, after the base-change $B_i\to B$, and is thus of Picard rank $1$. By  Lemma~\ref{LinkSop}, $S_i$ is a Severi-Brauer surface birational to $S$, and is thus either isomorphic to $S$ or to $S^{\rm op}$.

Since each birational map $\chi_i\colon W_i/B_i\dasharrow W_{i+1}/B_{i+1}$ is defined over $B$, it induces a birational map $\tau_i\colon S_i\dasharrow S_{i+1}$ Moreover, $\chi_i$ being a Sarkisov link, we find that $\tau_i$ is either an isomorphism or a Sarkisov link. Finally, as $\psi=\chi_n\circ \cdots\circ \chi_1$ and $\tau$ is induced by $\psi$, we have $\tau=\tau_n\circ \cdots\circ \tau_1$. Since the image of $\tau$ under the groupoid homomorphism
\[
\Psi\colon \BirMori_K(S)\to \bigoplus\limits_{\mathcal{E}_3} \ZZ/3\ZZ \ast \left ( \bigast_{\mathcal{E}_6} \mathbb{Z} \right)
\]
of Theorem~\ref{thm:SBgroupoidhomo} is equal to one generator of $\ZZ$ or $\ZZ/3\ZZ$, and each $\tau_i$ is sent to one generator, we find that $\Psi(\tau)=\pm \Psi(\tau_i)$ for some $i\in \{1,\ldots,n\}$. If $\Psi(\tau)=\Psi(\tau_i)$, we choose $Y=W_{i}$, $\widehat{B}=B_i=B_{i+1}$,  $Y'=W_{i+1}$ and $\chi=\chi_i$. We may choose the isomorphism of the generic fibres with $S$ and $S^{\rm op}$ so that the action of the $B$-birational map $ \chi \colon Y\dasharrow Y'$ on the generic fibres is $\tau$. If $\Psi(\tau)=-\Psi(\tau_i)$, we simply choose $Y=W_{i+1}$, $Y'=W_{i}$, $\widehat{B}=B_i=B_{i+1}$, and $\chi=\chi_i^{-1}$.
\end{proof}

To prove Theorem~\ref{thm: free product}, Corollaries~\ref{CorUniversalBir}-\ref{CorHopfian} and Theorem~\ref{3Torsion}, we will apply Theorem \ref{Theorem:SBMfs} to a \emph{rational} SBMfs $X\to B$ with a birational morphism $B\to\PP^n$ such that the generic fibre of $X\to\PP^n$  is the non-trivial Severi-Brauer surface $S$ of Corollary~\ref{cor: non-trivial SB over function field}. This latter corresponds to the surface $S=S_\xi$ of Lemma~\ref{lem:SBdegree3} associated to the field extension $L/K$, with $K=\CC(t_1,\ldots,t_n)=\CC(\PP^n)$, $L=K[\sqrt[3]{\lambda}]$, $\lambda=t_1$ and $\xi=t_2$, for each $n\geqslant 2$. The existence of the SBMfs $X\to B$, where $B\to\mathbb{P}^n$ is a birational morphism is provided by Proposition~\ref{prop:SBtoSBMfs}. Note that fixing the generic fibre $S$ does not give a unique SBMfs $X/B$. However, if $X'/B'$ is another such SBMfs, there is a birational map $X\dasharrow X'$ over $\mathbb{P}^n$, inducing an isomorphism between the generic fibres. In particular, $X$ is rational if and only if $X'$ is rational. Note that the generic fibre of $\pi\colon X\to B$ is $S^f$, and it is obtained by base-change $f\colon\Spec\CC(B)\to\Spec\CC(\PP^n)$ induced by the birational morphism $B\to\PP^n$.

\begin{prop}\label{prop: explicit SB fibration}
Let $n\geqslant 2$ and let us consider the quasi-projective variety
\[
\mathscr{X}=\big \{([w:x:y:z],(t_1,\ldots,t_n))\in \PP^3_\CC \times (\mathbb{A}^1_\CC\setminus \{0\})^n\mid t_2 w^3=t_1 x^{3}+y^{3}+t_1^{-1}z^3-3 xyz \big \}.
\]
Then, the following hold:
\begin{enumerate}
\item\label{fibration1}
The variety $\mathscr{X}$ is birational to $\mathbb{P}^{n+2}_{\CC}$;
\item\label{fibration2}
The generic fibre of the projection $p\colon \mathscr{X}\to (\mathbb{A}^1_\CC\setminus \{0\})^n$ is birational to the non-trivial Severi-Brauer surface $S$ of Corollary~$\ref{cor: non-trivial SB over function field}$, defined over $K=\CC(\mathbb{A}^n)=\CC(t_1,\ldots,t_n)$. Furthermore, there exists a commutative diagram
\begin{equation*}
	\xymatrix@R=6pt@C=20pt{
		\mathscr{X}\ar[d]_{p}\ar@{-->}[rr]^{\psi} && X\ar[d]^{\pi}\\
		{(\mathbb{A}^1_\CC\setminus \{0\})^n} \ar[rd]^-{\tau} && B\ar[ld]_{\kappa}\\
		&\PP^n&
	}
\end{equation*}
where $\tau$ is the open embedding $(t_1,\ldots,t_n)\mapsto [1:t_1:\cdots:t_n]$, $\psi$ and  $\kappa$ are birational maps, $\kappa$ being a morphism, $\pi\colon X\to B$ is a SBMfs and where the generic fibre of $X\to \PP^n$ is $S$, via the base change induced by $\tau$.
\end{enumerate}
\end{prop}

\begin{proof}

\ref{fibration1} The variety $\mathscr{X}$ is birational to $\AA^{n+2}_\CC$, as it is given, in the affine chart $w=1$, by the equation $t_2=t_1x^3+y^3+t_1^{-1}z^3-3xyz$ of degree $1$ in $t_2$. The projection to $t_1,t_3,\ldots,t_n,x,y,z$ yields a birational map to $\AA^{n+2}_\CC$.

\ref{fibration2}: As explained before, we consider the surface $S=S_\xi$ of Lemma~\ref{lem:SBdegree3} associated to the field extension $L/K$, with $K=\CC(t_1,\ldots,t_n)$, $L=K[\sqrt[3]{\lambda}]$, $\lambda=t_1$ and $\xi=t_2$. By Corollary~\ref{cor: non-trivial SB over function field}, it is a non-trivial Severi-Brauer surface defined over $K$. Fixing the open embedding $\tau\colon \mathbb{A}^n\to \PP^n$,  $(t_1,\ldots,t_n)\mapsto [1:t_1:\cdots:t_n]$, we  identify $K=\CC(t_1,\ldots,t_n)$ with $\CC(\PP^n)$.  According to Proposition~\ref{prop:SBtoSBMfs}, there exists a birational morphism $B\to \mathbb{P}^n$, and a SBMfs $\pi\colon X\to B$ such that the generic fibre of $X\to \PP^n$ is $S$. Consider the cubic surface
\[
Q=\big\{ t_2w^3=t_1x^3+y^3+t_1^{-1}z^3-3xyz \big \}\subseteq\PP^3_K,
\]
that is the generic fibre of the projection $p\colon \mathscr{X}\to (\mathbb{A}^1_\CC\setminus \{0\})^n$. By Proposition~\ref{prop: birational cubic model}, there is a birational map $Q\dashrightarrow S$, defined over $K$. The generic fibres of $p$ and $\pi$ being birational, we obtain a birational map  $\mathscr{X}\dasharrow X$ compatible with the above diagram.
\end{proof}

\subsection{$6$-links of large covering genus and the proof of Theorem~\ref{thm: free product}}

To prove Theorem~\ref{thm: free product}, we first need to provide sufficiently many $6$-links with arbitrary large covering genus.

\begin{prop}\label{prop:existence6linksC}
Let $n\geqslant 2$. We consider the non-trivial Severi-Brauer surface $S$ of Corollary~$\ref{cor: non-trivial SB over function field}$, defined over $K=\CC(t_1,\ldots,t_n)=\CC(\PP^n)$.
Let $r\geqslant 1$ be an integer, and $q\in \CC[x_0,\ldots,x_{n-1}]_{6r}$ be a homogeneous polynomial of degree $6r$ that defines a smooth hypersurface in $\PP^{n-1}$.
Consider the variety
\[
\Delta=\{[w:x_0:\cdots:x_n]\in \PP(3r,1,\ldots,1)\mid w^2=q(x_0,\ldots,x_{n-1})+x_n^{6r}\}
\]
with its double covering $\eta\colon \Delta \rightarrow \PP^n$ given by $\eta\colon[w:x_0:\cdots:x_n]\mapsto [x_0:\cdots:x_n]$,
and the quotient $\mu\colon \PP^n\to \PP^n/\langle \gamma\rangle$ by the automorphism $\gamma\in \Aut(\PP^n)$ given by
\[\gamma\colon [x_0:\cdots:x_n]\mapsto [x_0:\cdots:x_{n-1}:e^{2\mathbf{i}\pi/3} x_n].
\]
Then, there is a birational morphism $\sigma\colon \widehat{B}\to \PP^n$, two SBMfs $\pi\colon Y\to \widehat{B}$ and $\pi'\colon Y'\to \widehat{B}$, such that the generic fibres of $Y\to \PP^n$ and $Y'\to \PP^n$ are isomorphic to $S$ and $S^{\rm op}$ respectively, and a $6$-link
\[
	\xymatrix@R=6pt@C=20pt{
		Y\ar@{-->}[rr]^{\chi}\ar[rd]_{\pi} &&Y'\ar[ld]^{\pi'} \\
		& \widehat{B},
	}\]
such that the base-loci of $\chi$ and $\chi^{-1}$ are two varieties $\Gamma\subseteq Y$, $\Gamma'\subseteq Y'$, birational over $\PP^n$, together with a commutative diagram
\begin{equation}\label{eq: base loci of 6-link}
	\xymatrix@R=6pt@C=20pt{
		\Gamma\ar[dr]_{\pi}\ar@{-->}[rr]^{\psi} &&  \Gamma'\ar[dl]^{\pi'}\ar@{-->}[r]^{\psi'} & \Delta\ar[d]^{\eta} \\
		& \widehat{B} \ar[d]^{\sigma}&& \PP^n\ar[d]^{\mu}\\
		& \PP^n \ar@{-->}[rr]^{\varphi} &&  \PP^n/\langle \gamma \rangle,
	}
\end{equation}
for some birational maps $\psi,\psi',\varphi$. Moreover, $\varphi$ is given by the identification of the field $K=\CC(t_1,\ldots,t_n)$ with $\CC(\PP^n)$ and $\CC(\PP^n/\langle \gamma\rangle)$, where $t_i=\frac{x_i}{x_0}$ in the first case, and $t_1=(\frac{x_{n}}{x_0})^3, t_2=\frac{x_1}{x_0},\ldots,t_n=\frac{x_{n-1}}{x_0}$ in the second.
\end{prop}

\begin{proof}
By Proposition~\ref{prop: existence of links over function fields}\ref{exi6}, for each $\alpha\in K=\CC(t_1,\ldots,t_n)$ which is not a square in $L=K[\sqrt[3]{t_1}]$, there is a $6$-point $p\in S$ whose splitting field is $L[\sqrt{\alpha}]$, and thus a $6$-link $\tau\colon S \dasharrow S^{\rm op}$ (see Example~\ref{Ex36links} and Lemma~\ref{LinkSop}). By Proposition~\ref{prop:SBtoSBMfs}, there is a birational morphism $\widehat{B}\to \PP^n$, two SBMfs $Y\to \widehat{B}$ and $Y'\to \widehat{B}$ such that the generic fibres of $Y\to \PP^n$ and $Y'\to \PP^n$ are isomorphic to $S$ and $S^{\rm op}$ respectively, and a $6$-link $\chi\colon Y/\widehat{B}\dasharrow Y'/\widehat{B}$, which is a Sarkisov link of type \II, such that the action of the $\PP^n$-birational map $\chi \colon Y\dasharrow Y'$ on the generic fibres is $\tau$. The base-loci of $\chi$ and $\chi^{-1}$ are two varieties $\Gamma$, $\Gamma'$ such that the morphism $\Gamma\to \PP^n$ and $\Gamma'\to \PP^n$ are generically $6:1$ (Lemma~\ref{lem:SBMfsSarki}). Hence, the splitting field is the residue field of the point, which is the same for the base-point $p$ of $\tau$ and the base-point of $\tau^{-1}$ by Lemma~\ref{lem:splittingfieldinverse}. This residue field is the field of functions of $\Gamma$ and $\Gamma'$: the field extension associated to the morphisms $\Gamma\to \PP^n$ and  $\Gamma'\to \PP^n$ is $L[\sqrt{\alpha}]/K$.

Consider the double covering $\eta\colon \Delta\to \PP^n$ and the morphism $\mu\colon \PP^n\to \PP^n/\langle \gamma \rangle$, where $\gamma$ is as above. Since $\CC(\PP^n)=\CC(\frac{x_1}{x_0},\cdots,\frac{x_n}{x_0})$, we obtain $\CC(\PP^n/\langle \gamma \rangle)=\CC(\frac{x_1}{x_0},\cdots,\frac{x_{n-1}}{x_0})$. Setting $t_1=(\frac{x_{n}}{x_0})^3, t_2=\frac{x_1}{x_0},\ldots,t_n=\frac{x_{n-1}}{x_0}$, we obtain $\CC(\PP^n/\langle \gamma \rangle)=\CC(t_1,\ldots,t_n)=K$, $\CC(\PP^n)=\CC(\sqrt[3]{t_1},t_2,\ldots,t_n)=L$. Now choose
$$
\alpha=\frac{x_n^{6r}+q(x_0,\ldots,x_{n-1})}{x_0^{6r}}=t_1^{2r} +q(t_2,\ldots,t_n)
$$
Then $\CC(\Delta)=L[\sqrt{\alpha}]$. As $x_n^{6r}+q$ is irreducible, $\alpha$ is not a square in $L$. The Sarkisov link described as above is then such that the field extensions $\Gamma\to \PP^n$ and $\Gamma'\to \PP^n$ are isomorphic to the one given by $\Delta\to \PP^n/\langle \gamma\rangle$. This gives the existence of birational maps $\psi,\psi'$. Finally, $\varphi$ is given by the identification of $K=\CC(t_1,\ldots,t_n)$ with $\CC(\PP^n)$ and $\CC(\PP^n/\langle \gamma\rangle)$, where $t_i=\frac{x_i}{x_0}$ in the first case, and $t_1=(\frac{x_{n}}{x_0})^3, t_2=\frac{x_1}{x_0},\ldots,t_n=\frac{x_{n-1}}{x_0}$ in the second.
\end{proof}

In what follows, we will need the following classical result, whose essential part is known as the Matsumura-Monsky theorem.

\begin{prop}\label{prop:Matsumura-Monsky}
Let $d\geqslant 3$ and let $n\geqslant 1$ be integers. Letting $\CC[x_0,\ldots,x_n]_d$ be the vector space of homogeneous polynomials of degree $d$ in $n+1$ variables, view $\PP(\CC[x_0,\ldots,x_n]_d)$ as $\PP^{N-1}$ with $N=\binom{d+n}{d}$. Then the set $\mathcal{H}_{n,d}$ corresponding to smooth hypersurfaces in $\PP^n$ of degree $d$ is open and dense in this projective space. Furthermore, the set
\[
\mathcal{U}_{n,d}=\{H\in \mathcal{H}_{n,d}\mid \alpha(H)\not=H \text{ for each }\alpha\in \Aut(\PP^{n})\setminus \{\mathrm{id}\}\}
\]
of smooth hypersurfaces admitting no linear automorphisms is an open subset of $\mathcal{H}_{n,d}$, that is empty if and only if $(n,d)\in \{(1,3),(1,4),(2,3)\}$.
\end{prop}
\begin{proof}
The fact that $\mathcal{U}_{n,d}$ is open in $\mathcal{H}_{n,d}$ follows from \cite[Lemma 11.8.5]{KatzSarnak}.  The remaining non-trivial part of the statement is essentially due to Matsumura and Monsky \cite{MatsumuraMonsky} who proved it for $n\geqslant 3$, $d\geqslant 3$. The case $n=2$, $d\geqslant 4$ can be found e.g. in \cite{ChangPlaneCurves}. If $(n,d)=(2,3)$, then $\mathcal{U}_{n,d}$ is empty, as any smooth cubic curve can be put in the Hessian form and admits a non-trivial finite group of linear automorphisms.

The case $n=1$ is easy to handle. For $d\in \{3,4\}$, we get three or four points in $\mathbb{P}^1$, that can be exchanged by some automorphism of order $2$ (transposition in the case of $3$ points and $2\times 2$-cycle for four points), so $\mathcal{U}_{1,d}=\varnothing$ for such $d$.
On the other hand, for $d\geqslant 5$ one can construct a set $\Sigma=\{a_1,\ldots,a_d\}$ of $d$ points on $\PP^1$ which is not preserved by any projective transformation. For this, we require that all pairs of $4$ points from this set have different cross-ratios. Having two cross-ratios being different gives an open condition on the set of $d$-uples on $\PP^1$, so it is possible to find such a set. Then, any automorphism $\tau$ of $\PP^1$ will preserve every $4$-element subset of $\Sigma$. Since $d\geqslant 3$, $\tau$ can be identified with an element of $\Sym(\Sigma)\simeq\Sym_d$, and hence it belongs to $\cap_{\iota}\Sym_4^\iota$, where the intersection is taken over all embeddings $\iota\colon\Sym_4\hookrightarrow\Sym_d$. Since every $\Sym_4^\iota$ can be identified with the stabilizer of $d-4$ points from $\{a_1,\ldots,a_d\}$, this intersection is trivial (immediate for $d=5$, and follows by induction for $d>5$).
\end{proof}

\begin{proof}[Proof of Theorem \ref{thm: free product}]
Let $n=m-2\geqslant 2$ and consider the groupoid $\BirMori(X)$, where $\pi\colon  X\to B$ is the SBMfs constructed in Proposition~\ref{prop: explicit SB fibration} with birational morphism $B\to\PP^n$. Since $X$ is birational to $\PP^m=\PP^{n+2}$, we have $\Bir_\CC(\PP^{m})\simeq\Bir(X)$. We will apply Theorem~\ref{Theorem:SBMfs} and need for this some $6$-links of covering genus at least $g$, where $g$ is as in Theorem~\ref{Theorem:SBMfs}.

Recall that the generic fibre of $X\to\PP^n$ is the non-trivial Severi-Brauer surface $S$ of Corollary~$\ref{cor: non-trivial SB over function field}$, defined over $K=\CC(\mathbb{A}^n)=\CC(t_1,\ldots,t_n)$. For each integer $d\geqslant 1$ and each homogeneous polynomial $q\in \CC[x_0,\ldots,x_{n-1}]$ of degree $6d$ that defines a smooth hypersurface of $\PP^{n-1}$, Proposition~\ref{prop:existence6linksC} provides a link $\chi_q\colon Y_q/\widehat{B}_q\dasharrow Y_q'/\widehat{B}_q$. The base-loci $\Gamma_q\subseteq Y_q$ and $\Gamma'_q\subseteq Y'_q$ of $\chi_q$ and $\chi_q^{-1}$ fit into the commutative diagram \eqref{eq: base loci of 6-link} with the corresponding $\eta_q\colon\Delta_q\to\PP^n$ and $\sigma_q\colon \widehat{B}_q\to \PP^n$. As $q$ defines a smooth hypersurface of $\PP^{n-1}$, the polynomial $q(x_0,\ldots,x_{n-1})+x_n^{6d}$ is irreducible and defines a smooth hypersurface of $\PP^n$. We fix some integer $d$ large enough such that the covering genus of $\Delta_q$ is at least $g$ for each $q$, which is possible by Lemma~\ref{lem: cyclic covering covgen}.

Consider the sets $\mathcal{H}_{n-1,6d}$ and $\mathcal{U}_{n-1,6d}$ defined in Proposition~\ref{prop:Matsumura-Monsky}. The group $\Aut(\PP^{n-1})\simeq\PGL_n(\CC)$ acts on both. Choosing $d$ large enough so that the dimension of $\mathcal{U}_{n-1,6d}/\Aut(\PP^{n-1})$ is positive, we find some set $R\subseteq \mathcal{U}_{n-1,6d}$, of cardinality equal to the one of $\CC$, such that no non-trivial element of $\Aut(\PP^{n-1})$ sends any element of $R$ onto another one.

Take $q,s\in R$ and consider the associated links $\chi_q\colon Y_q\dasharrow Y_q'$ and $\chi_s\colon Y_s\dasharrow Y_s'$. Let us show that if $\chi\in \{\chi_q,\chi_q^{-1}\}$ is equivalent to $\chi_s$, then $q=s$ and $\chi=\chi_q$. Indeed, if $\chi_q$ or $\chi_q^{-1}$ is equivalent to $\chi_s$, there is a commutative diagram
	\[
	\xymatrix@R=6pt@C=20pt{
		\Gamma_q\ar[d]\ar@{-->}[rr]^{\widehat{\nu}} &&\Gamma_s\ar[d] \\
		\PP^n \ar@{-->}[rr]^{\theta} &&  \PP^n,
	} \text{ or }
	\xymatrix@R=6pt@C=20pt{
		\Gamma_q'\ar[d]\ar@{-->}[rr]^{\widehat{\nu}} &&\Gamma_s\ar[d] \\
		\PP^n \ar@{-->}[rr]^{\theta} &&  \PP^n
	}\]
where $\widehat{\nu}$ and $\theta$ are birational maps. By Proposition~\ref{prop:existence6linksC}, this induces a birational map
	\[
	\xymatrix@R=6pt@C=20pt{
		\Delta_q\ar[d]\ar@{-->}[rr]^{\nu} &&\Delta_s\ar[d] \\
		\PP^n/\langle \gamma \rangle \ar@{-->}[rr]^{\theta'} &&  \PP^n/\langle \gamma \rangle,
	}\]
where $\theta'=\varphi\circ\theta\circ\varphi^{-1}$ is obtained by conjugating $\theta$ by the birational map $\varphi\colon \PP^n\dasharrow \PP^n/\langle \gamma \rangle$ of Proposition~\ref{prop:existence6linksC}, which does not depend on $q$ or $s$.	By Lemma~\ref{lem: cyclic covering covgen}, the canonical divisors of $\Delta_q$ and $\Delta_s$ are ample, so $\nu$ is an isomorphism $\Delta_q\iso \Delta_s$ by Lemma~\ref{lem: birational map of canonically polarized}. Moreover, the double coverings $\eta_q\colon\Delta_q\to \PP^n$ and $\eta_s\colon\Delta_s\to \PP^n$ are given by a fraction of the canonical divisor (again by Lemma~\ref{lem: cyclic covering covgen}), so we obtain a commutative diagram
	\[
	\xymatrix@R=6pt@C=20pt{
		\Delta_q\ar[d]_{\eta_q}\ar[rr]^{\nu} &&\Delta_s\ar[d]^{\eta_s} \\
		\PP^n \ar[d]_{\mu}\ar[rr]^{\alpha}&& \PP^n\ar[d]^{\mu}\\
		\PP^n/\langle \gamma \rangle \ar@{-->}[rr]^{\theta'} &&  \PP^n/\langle \gamma \rangle,
	}\]
where $\alpha\in \Aut(\PP^n)$ sends the hypersurface $q+x_n^{6d}=0$ onto the hypersurface $s+x_n^{6d}=0$. As $\alpha$ is compatible with the quotient $\mu\colon \PP^n\to \PP^{n}/\langle \gamma\rangle$, it has to commute with $\gamma$, and thus be of the form $[x_0:\cdots:x_n]\mapsto [M(x_0,\ldots,x_{n-1}):x_n]$, for some $M\in \GL_{n}(\CC)$. The class of $M$ in $\PGL_{n}(\CC)=\Aut(\PP^{n-1})$ sends $q=0$ onto $s=0$. As $q,s\in R$, we find that $q=s$ and that $M$ is a homothety. It remains to show that $\chi=\chi_q$ (to see that $\chi_q$ and $\chi_q^{-1}$ are not equivalent).
Remembering that $\alpha\in \Aut(\PP^n)$ sends the hypersurface $q+x_n^{6d}=0$ onto the hypersurface $s+x_n^{6d}=0$, we find $\alpha=[x_0:\cdots:x_n]\mapsto [x_0:\cdots:x_{n-1}:\zeta x_n]$, where $\zeta^{6d}=1$. The action of $\theta'$ and of $\theta$ are then uniquely determined by $\zeta$. The action on $K=\CC(t_1,\ldots,t_n)$ then fixes $t_2,\ldots,t_n$ and sends $t_1$ onto $\zeta^3t_1$ (see Proposition~\ref{prop:existence6linksC}). It remains to see that under this base-change, the generic fibres of $Y_q\to\PP^n$ and $Y_q'\to\PP^n$ are not isomorphic. Recall that $S$ is given as in Corollary~$\ref{cor: non-trivial SB over function field}$, defined over $K=\CC(t_1,\ldots,t_n)$ and is the  surface $S=S_\xi$ of Lemma~\ref{lem:SBdegree3} associated to the field extension $L/K$, with, $L=K[\sqrt[3]{\lambda}]$, $\lambda=t_1$ and $\xi=t_2$. Moreover, $S^{\rm op}$ is $S_{\xi^{-1}}$ (Lemma~\ref{Lem:SxiOp}). The base-change replaces $\xi$ with some non-zero scalar multiple of $\xi$ and then does not change the isomorphism class of the surface over $K$ (Lemma~\ref{lem:SBdegree3}). As $S$ is not isomorphic to $S^{\rm op}$, we have proven that $\chi_q$ is not equivalent to $\chi_q^{-1}$.

We now finish the proof, by applying Theorem~\ref{Theorem:SBMfs} to the set $\mathcal{M}_6$ of Sarkisov links $\chi_q$, with $q\in R$. As we have proven, these links $\chi_q$ all have covering genus at least $g$ and no two different elements are equivalent and no link is equivalent to its inverse. %
Moreover, for each $q$, the generic fibre of $Y_q\to \PP^n$ is isomorphic to $S$, the generic fibre of $X/\PP^n$ (Proposition~\ref{prop:existence6linksC}). Therefore, the generic fibre $\widetilde{S}$ of the SBMfs $X/B$ is birational to the generic fibre of $Y_q/\widehat{B}_q$, after a base change as in Definition~\ref{Def:inducedBy}, and so $\chi_q$ induces an element in $\BirMori(\widetilde{S})$.
By Theorem~\ref{Theorem:SBMfs} there exists a surjective group homomorphism \[
\Bir(X)\to  \bigast_{\mathcal{M}_6\setminus\chi_{q_0}} \mathbb{Z}=\mathcal{F}(\mathcal{M}_6\setminus{\chi_{q_0}})\simeq \mathcal{F}(\CC)
\]
for any $q_0\in R$.
\end{proof}

\begin{proof}[Proof of Corollary~$\ref{CorHopfian}$]
Since every birational self-map of $\PP_\CC^n$ is given by a finite number of polynomials of a fixed degree, we have $\lvert \Bir_\CC(\PP^n)\rvert=\lvert\CC\rvert$. Thus there is a surjective homomorphism $\Upsilon\colon \mathcal{F}(\CC)\twoheadrightarrow\Bir_{\CC}(\PP^n)$ given by a presentation of $\Bir_{\CC}(\PP^n)$ in which the set of generators is the whole $\Bir_{\CC}(\PP^n)$. Now take a surjective homomorphism $\Psi: \Bir_{\CC}(\PP^n)\twoheadrightarrow \mathcal{F}(\CC)$ provided by Theorem \ref{thm: free product}. We get a surjective endomorphism $\Upsilon\circ\Psi: \Bir_{\CC}(\PP^n)\to\Bir_{\CC}(\PP^n)$ that is not injective (as neither $\Psi$ nor $\Upsilon$ is).
\end{proof}

\section{The abelianisation of higher Cremona groups: 3-torsion part}

This final section is devoted to the proof of Theorem~\ref{3Torsion}. Our strategy is similar to the one we applied in the proof of Theorem~\ref{thm: free product}. Namely, we need to provide enough non-equivalent $3$-links of large covering genus, which are also not equivalent to their inverses; then we apply Theorem~\ref{Theorem:SBMfs}. The technical details will just be more involved.

\subsection{Birational transformations of order 3 on Severi-Brauer surface bundles}

First, we need a version of Proposition~\ref{prop:existence6linksC} for $3$-links.

\begin{prop}\label{prop:existence3linksC}
Let $n\geqslant 2$. Consider the non-trivial Severi-Brauer surface $S$ of Corollary~$\ref{cor: non-trivial SB over function field}$, defined over $K=\CC(t_1,\ldots,t_n)=\CC(\PP^n)$. Let $a,b\in \CC[t_1,\ldots,t_n]$ be  such that $t_2b^3-a^3$ is not a cube in $L=K[\sqrt[3]{t_1}]$. Consider the variety
\[\Delta=\{(s,t_1,\ldots,t_n)\in \AA^{n+1}\mid  t_1s^3+a(t_1,\ldots,t_n)^3-t_2b(t_1,\ldots,t_n)^3=0\}\]
that admits a generically $3:1$ morphism $\eta\colon \Delta \to \PP^n$, $(s,t_1,\ldots,t_n)\mapsto [1:t_1:\ldots:t_n]$.

Then, there is a birational morphism $\sigma\colon \widehat{B}\to \PP^n$, two SBMfs $\pi\colon Y\to \widehat{B}$ and $\pi'\colon Y'\to \widehat{B}$, such that the generic fibres of $Y\to \PP^n$ and $Y'\to \PP^n$ are isomorphic to $S$ and $S^{\rm op}$ respectively, and a $3$-link
\[
	\xymatrix@R=6pt@C=20pt{
		Y\ar@{-->}[rr]^{\chi}\ar[rd]_{\pi} &&Y'\ar[ld]^{\pi'} \\
		& \widehat{B},
	}\]
such that the base-loci of $\chi$ and $\chi^{-1}$ are two varieties $\Gamma\subseteq Y$, $\Gamma'\subseteq Y'$, birational over $\PP^n$, together with a commutative diagram

\begin{equation}\label{eq: base loci of 3-link}
	\xymatrix@R=6pt@C=20pt{
		\Gamma\ar[dr]_{\pi}\ar@{-->}[rr]^{\psi} &&  \Gamma'\ar[dl]^{\pi'}\ar@{-->}[r]^{\psi'} & \Delta\ar[ddll]^{\eta} \\
		& \widehat{B} \ar[d]_{\sigma}\\
		& \PP^n,
	}
\end{equation}
for some birational maps $\psi,\psi'$. Moreover, there exists an element $\phi\in \Bir(Y/\PP^n)$ of order $3$, which admits a decomposition into Sarkisov links $\chi_s\circ\cdots \chi_1\circ \chi$, where $\chi_1,\ldots,\chi_s$ are Sarkisov links between SBMfs, being either isomorphisms between the generic fibres or of covering genus $0$.
\end{prop}
\begin{proof}

	Recall that the non-trivial Severi-Brauer surface $S$ of Corollary~$\ref{cor: non-trivial SB over function field}$ is defined over $K=\CC(\mathbb{A}^n)=\CC(t_1,\ldots,t_n)$ and is the surface  $S=S_\xi$ of Lemma~\ref{lem:SBdegree3} associated to the field extension $L/K$, with, $L=K[\sqrt[3]{\lambda}]$, $\lambda=t_1$ and $\xi=t_2$. As $t_2b^3-a^3$ is not a cube in $L$, neither is $\mu=(t_2b^3-a^3)(27b^3t_1)^{-1}$. However, $\xi-27\lambda\mu={a^3}{b^{-3}}$ is a cube in $K$. By Proposition~\ref{prop: existence of links over function fields}\ref{exi3}, there is an element $\varphi\in \Bir_K(S)$ of order $3$, that is equal to $\tau_2^{-1}\circ \tau_1$ where $\tau_1$ and $\tau_2^{-1}$ are $3$-links with splitting fields $L$ and $K[\sqrt[3]{\mu}]$, respectively (or vice versa). Replacing $\varphi$ with $\varphi^{-1}$ if needed, we may assume that $L$ is the splitting field of $\tau_2$ and $K[\sqrt[3]{\mu}]$ is the splitting field of $\tau_1$.

	For each $i\in \{1,2\}$, we apply Proposition~\ref{prop:SBtoSBMfs}\ref{LinkStoSBMfs} to $\tau_i$ and obtain a birational morphism $\widehat{B}_i\to \PP^n$, two SMBfs $Y_i\to \widehat{B}_i$ and $Y'_i\to \widehat{B}_i$, so that the  generic fibres of $Y_i\to  \PP^n$ and $Y'_i\to \PP^n$ are isomorphic to $S$ and $S^{\rm op}$  respectively and a $d$-link $\chi_i\colon Y_i/\widehat{B}_i\dasharrow Y'_i/\widehat{B}_i$ such that the action of the $\PP^n$-birational map $\chi_i \colon Y_i\dasharrow Y'_i$ on the generic fibres is $\tau_i$. We write $\chi=\chi_1$, $Y=Y_1$, $Y'=Y_1'$, $\widehat{B}=\widehat{B}_1$. The element $\varphi\in \Bir_K(S)$ corresponds to an element  $\phi\in\Bir(Y/\widehat{B})$ of order $3$. The action on the generic fibres of $\phi \circ \chi^{-1}\colon Y'\dasharrow Y$ being the same as the one of $\chi_2^{-1}$ (namely, this is an action of $\tau_2^{-1}=(\tau_2^{-1}\circ\tau_1)\circ\tau_1^{-1}$), we may choose two birational maps $\alpha\colon Y'\dasharrow Y_2'$ and $\beta \colon Y_2\dasharrow Y$, inducing isomorphisms on the generic fibres and obtain
	$\phi=\beta \circ \chi_2^{-1}\circ  \alpha \circ \chi.$
	We then decompose $\alpha$ and $\beta$ into Sarkisov links over $\PP^n$ and obtain links which are isomorphisms on generic fibres. It remains to show that the covering genus of $\chi_2$ is zero and to describe the base-locus of $\chi$ and the diagram~\eqref{eq: base loci of 3-link}.

	The base-loci of $\chi$ and $\chi^{-1}$ are two varieties $\Gamma$ and $\Gamma'$ such that the morphisms $\Gamma\to \PP^n$ and $\Gamma'\to \PP^n$ are generically $3:1$, see Lemma~\ref{lem:SBMfsSarki}. As in the proof of Proposition~\ref{prop:existence6linksC}, we notice that the function fields $\CC(\Gamma)$ and $\CC(\Gamma')$ coincide with the splitting field of $\tau_1$, i.e.  $K[\sqrt[3]{\mu}]$, which is the field extension of $K$ associated to the morphisms $\Gamma\to \PP^n$ and  $\Gamma'\to \PP^n$. The description of the morphism $\eta\colon\Delta\to \PP^n$ gives that the corresponding field extension is $K[\sqrt[3]{\mu'}]/K$ where $\mu'={(a^3-t_2b)}{t_1^{-1}}=27\mu b^3$, so both extensions of degree $3$ coincide, and this gives Diagram~\eqref{eq: base loci of 3-link}. For $\chi_2$, the same works, with $K[\sqrt[3]{\mu}]/K$ replaced with $L/K$. As $L=\CC(\sqrt[3]{t_1},t_2,\ldots,t_n)$, the field extension $L/K$ corresponds to the $3:1$ morphism $\AA^n\to \AA^n$, $(s_1,t_2,\ldots,t_n)\mapsto (s_1^3,t_2,\ldots,t_n)$.  We find that the base-locus of $\chi_2$ is a rational variety, so the covering genus is zero.
\end{proof}

\subsection{Birational models of the centers of 3-links}

We will use Proposition~\ref{prop:existence3linksC} for some polynomials $a$ and $b$ of large degree. To show that  the covering genus of the associated link $\chi$ can be chosen arbitrarily large and to show that the link $\chi$ may be chosen to be not equivalent to its inverse, we need to study a birational model of $\Delta$ that is smooth (note that $\Delta$ is singular at the points where $s=a=b=0$). The model obtained is the variety $\Omega$ described in Proposition~\ref{Lem:ModelDeltaSmooth}. The proof that it is smooth relies on the  technical lemmas~\ref{LemABsmoothBlowUp} and~\ref{Lem:SmoothAffHyp}.

\begin{lem}\label{LemABsmoothBlowUp}
Let $n,d\geqslant 2$ be integers and let $V_{n,d}=\CC[x_0,\ldots,x_n]_d$ be the $\CC$-vector space of homogeneous polynomials of degree $d$ in $x_0,\ldots,x_n$. Consider the subset $U_{n,d}\subseteq \PP(V_{n,d} \times V_{n,d})$ of classes $[A:B]$ of pairs $(A,B)\in V_{n,d}\times V_{n,d}$ satisfying the following conditions:
\begin{enumerate}
\item\label{UABGammasmooth}
$F=\{([x_0:\ldots:x_n])\in \PP^n\mid  A(x)=B(x)=0\}$ is smooth and of dimension $n-2$.
\item\label{UABZsmooth}
$Z=\{([x_0:\ldots:x_n])\in F\mid x_n=0\}$ is smooth and of dimension $n-3$ $(\varnothing$ if $n=2)$.
\item\label{UABXsmooth}
$P=\{([x_0:\ldots:x_n],[u:v])\in \PP^n\times \PP^1\mid A(x)v=B(x)u\}$
is smooth, of dimension $n$.
\item\label{UABQsmooth}
$Q=\{([x_0:\ldots:x_n],[u:v])\in P \mid  v^3x_0=x_1u^3\}$
is smooth, of dimension $n-1$.
\item\label{UABQxsmooth}
$R=\{([x_0:\ldots:x_n],[u:v])\in Q\mid x_n=0\}$ is smooth, of dimension $n-2$.
\end{enumerate}
Then $U_{n,d}$ is open and dense in $ \PP(V_{n,d} \times V_{n,d})\simeq\PP^{2N-1}$ for $N=\binom{n+d}{n}$.
\end{lem}
\begin{proof}
For each $n\geqslant 1$ and each $d\geqslant 2$, we denote by
$S_{n,d}\subseteq  \PP(V_{n,d} \times V_{n,d})$ the set of classes $[A:B]$ satisfying \ref{UABGammasmooth}-\ref{UABZsmooth}-\ref{UABXsmooth}, which is open by Bertini's theorem (or simply by looking at the matrices of derivatives of the equations). Furthermore, $S_{n,d}$ is dense, or equivalently non-empty, since $[A:B]=[\sum_{i=0}^n x_i^d: \sum_{i=0}^n \lambda_i x_i^d]\in S_{n,d}$ for $\lambda_0,\ldots,\lambda_n\in \CC^*$ pairwise different.

For each $n\geqslant 1$ and each $d\geqslant 2$, we define $T_{n,d}\subseteq \PP(V_{n,d} \times V_{n,d})$ to be the open set of classes $[A:B]$ satisfying \ref{UABQsmooth}.
We now prove that $T_{n,d}$ is dense, or equivalently that $T_{n,d}$ is not empty. For this, we prove that $[A:B]=[x_1^d+\sum_{i=2}^n a_i x_i^d :x_0^d+\sum_{i=2}^n b_i x_i^d]\in T_{n,d}$, for all $b_2,\ldots,b_n,a_2,\ldots,b_n\in \CC^*$ such that $\{b_ia_i^{-1}\mid i=2,\ldots,n\}$ are pairwise different and not roots of unity.
For this choice, $Q$ is given by
\[
Q=\{([x_0:\cdots:x_n],[u:v])\in \PP^n\times \PP^1\mid x_0^du-x_1^dv+\sum_{i=2}^n (b_iu-a_i v)x_i^d=0, v^3x_0=x_1u^3\}.
\]
If $n=1$, then $Q$ is the complete intersection of two curves of $ \PP^1\times \PP^1$ of bidegree $(d,1)$ and $(1,3)$ respectively, given by $x_0^du=x_1^dv$ and $v^3x_0=x_1u^3$. It is then smooth  as it defines exactly $3d+1$ points, namely $\{([\xi:1],[1:\xi^d]) \mid \xi \in \CC, \ \xi^{3d+1}=1\}$.

We suppose now $n\geqslant 2$ and assume for contradiction that some point of $Q$ is singular. Exchanging $u$ with $v$, $x_0$ with $x_1$ and $A$ with $B$, we may assume that this point satisfies $u=1$, so $x_1=v^3x_0$. It thus gives a singular point $([x_0:x_2:\cdots:x_n],v)\in \PP^{n-1}\times \AA^1$ of the hypersurface given by
\[
x_0^d(1-v^{3d+1})+\sum_{i=2}^n (b_i-a_i v)x_i^d=0.
\]
If the singular point satisfies $x_i=0$ for each $i\in \{2,\ldots,n\}$, then $x_0\not=0$ so $1-v^{3d+1}=0$, and in the affine chart $x_0=1$, the derivative with respect to $v$ is not zero, as the polynomial $1-v^{3d+1}$ has only simple roots. We then consider the case where $x_i\not=0$ for some $i\in \{2,\ldots,n\}$, consider the affine chart $x_i=1$ and obtain the polynomial equation
\[
x_0^d(1-v^{3d+1})+(b_i-a_iv)+\sum_{j\in \{2,\ldots,n\}\setminus \{i\}} (b_j-a_j v)x_j^d=0.
\]
Since the partial derivatives with respect to $x_j$, $j\in \{2,\ldots,d\}\setminus \{i\}$ are zero, we find that $(b_j-a_j v)x_j=0$. Similarly, the derivative with respect to $x_0$ gives $x_0(1-v^{3d+1})=0$, so $b_i-a_iv=0$ and $v=b_ia_i^{-1}$. By assumption, $b_j-a_jv\not=0$ for each $j\in \{2,\ldots,n\}\setminus \{i\}$, so $x_j=0$ and $1-v^{3d+1}\not=0$, whence $x_0=0$. The derivative with respect to $v$ is then equal to $-a_i\not=0$, which is a contradiction.

So, we proved that $T_{n,d}\subseteq \PP( V_{n,d} \times V_{n,d})$ is open and dense for each $n\geqslant 1$ and each $d\geqslant 2$. Assume that $n\geqslant 2$ and $d\geqslant 2$, let $W_{n,d}\subseteq \PP( V_{n,d} \times V_{n,d})$ be the dense open set corresponding to pairs $[A:B]$ where $x_n$ does not divide both $A$ and $B$,  and consider the surjective morphism \[\begin{array}{cccc}
\pi\colon & W_{n,d}&\to& \PP(V_{n-1,d} \times V_{n-1,d})\\
& [A:B]&\mapsto &[A(x_0,\ldots,x_{n-1},0):B(x_0,\ldots,x_{n-1},0)].\end{array}\]
We obtain that $\pi^{-1}(T_{n-1,d})$ is the subset of $\PP(V_{n,d}\times V_{n,d})$ consisting of classes $[A:B]$ for which \ref{UABQxsmooth} is satisfied, which is therefore open and dense in $\PP(V_{n,d} \times V_{n,d})$. Therefore $U_{n,d}=S_{n,d}\cap T_{n,d}\cap \pi^{-1}(T_{n-1,d})$; this is a dense open subset of $\PP(V_{n,d}\times V_{n,d})$.
\end{proof}

\begin{lem}\label{Lem:SmoothAffHyp}
Let $X$ be a smooth irreducible complex affine variety, let $f,g\in \mathcal{O}(X)$ be regular functions on $X$, such that the zero loci  $V_X(f)$ and $V_X(f,g)$ are smooth, of codimension $1$ and $2$ respectively. Then, for each integer $m\geqslant 2$, the affine variety \[W=\{(x,t)\in X\times \AA^1\mid t^mg=f \}\] is smooth, of the same dimension as $X$.
\end{lem}

\begin{proof}
We view $X$ as a closed subset of $\AA^n_\CC$, defined by equations $h_1,\ldots,h_r\in \CC[x_1,\ldots,x_n]$, and view $W$ as a closed subset of $\AA^{n+1}_\CC$, defined by $h_1,\ldots,h_r, x_{n+1}^mg-f$, where $f,g\in \CC[x_1,\ldots,x_n]$. Let $p=(p_1,\ldots,p_{n+1})\in W$. We check that this point is smooth, i.e.~the gradient $\nabla_p (s)=(\frac{\partial s}{x_1}(p),\ldots,\frac{\partial s}{x_{n+1}}(p))\in \CC^{n+1}$ of $s=x_{n+1}^mg-f$ is not in the $\CC$-vector space
$
V=\langle \nabla_p (h_1),\ldots,\nabla_p (h_r)\rangle,
$
viewed as a subspace of $\CC^{n+1}$. Note that $\frac{\partial h_i}{\partial x_{n+1}}(p)=0$ for each $i\in \{1,\ldots,r\}$ and $\frac{\partial s}{\partial x_{n+1}}(p)=mp_{n+1}^{m-1}g(p)$. Hence, if $p_{n+1}\not=0$ and $g(p)\ne 0$, we are done. Assume that $p_{n+1}=0$. Then $f(p)=0$ and $\frac{\partial s}{\partial x_{i}} (p)=-\frac{\partial f}{\partial x_{i}} (p)$ for each $i$. Since $(p_1,\ldots,p_n)$ is a smooth point of $V_X(f)$, the vector $\nabla_p(f)=-\nabla_p(s)$ does not belong to $V$, and the result follows. The remaining case is $p_{n+1}\not=0$ and $g(p)=0$, which yields $f(p)=0$. As $V_X(f,g)$ is smooth of codimension $2$ in $X$, only the trivial linear combination of $\nabla_p (f)$ and $\nabla_p (g)$ belongs to $V$. Since $\frac{\partial s}{\partial x_{i}} (p)=p_{n+1}^m\cdot \frac{\partial g}{\partial x_{i}} (p)-\frac{\partial f}{\partial x_{i}} (p)$ for each $i\in \{1,\ldots,n\}$ and $p_{n+1}\ne 0$, we conclude that $\nabla_p(s)\not\in V$.
\end{proof}

\begin{mydef}[Rational scroll, see \cite{ReidChapters}]\label{Fnd}
Let $n\geqslant 1$ and let $d\geqslant 0$. We define $\mathcal{F}_d^n$ to be the quotient of $ (\AA^{n}\setminus \{0\})\times (\AA^2\setminus \{0\})$ by the action of $\mathbb{G}_m^2$ given by
\[\begin{array}{ccc}
\mathbb{G}_m^2 \times(\AA^n\setminus \{0\})\times (\AA^{2}\setminus \{0\}) & \to &(\AA^n\setminus \{0\})\times (\AA^{2}\setminus \{0\})\\
((\lambda,\mu), (x_0,\ldots,x_{n-1},y_0,y_1))&\mapsto&
(\mu x_0,\ldots,\mu x_{n-1},\lambda\mu^{-d} y_0, \lambda y_1).\end{array}\]
The class of $(x_0,\ldots,x_{n-1},y_0,y_1)$ will be denoted $[x_0:\cdots:x_{n-1};y_0:y_1]$. The morphism
\[
\mathcal{F}_d^n\to \PP^{n-1},\ [x_0:\cdots:x_{n-1};y_0:y_1]\to [x_0:\cdots:x_{n-1}]
\]
identifies $\mathcal{F}_d^n$ with the $\PP^1$-bundle $\PP(\mathcal{O}_{\PP^{n-1}}\bigoplus \mathcal{O}_{\PP^{n-1}}(d))\to \PP^{n-1}$. In particular,  $\mathcal{F}_d^n$ is a smooth toric variety of dimension $n$. For $n=2$, we recover the classical Hirzebruch surfaces $\mathcal{F}_d^2=\mathbb{F}_d$.
\end{mydef}

\begin{lem}\label{Lem:ModelDeltaSmooth}
Let $n,d\geqslant 2$  be  integers and let $A,B\in \CC[x_0,\ldots,x_n]$ be homogeneous polynomials of degree $d$ such that $[A:B]$ satisfies the conditions of Lemma~$\ref{LemABsmoothBlowUp}$.
Then, the following variety
\[
\Omega=\left\{\left ([x_0:\cdots:x_n;y_0:y_1],[u:v:w] \right )\in \mathcal{F}_d^{n+1}\times \PP^2\left| \begin{array}{l}w^3x_n+x_1u^3-x_0v^3=0,\\ uy_1=wy_0A(x),vy_1=wy_0B(x),\\ A(x)v=B(x)u.\end{array} \right\}\right.
\]
is smooth of dimension $n$, and its isomorphism class only depends on the class $[A:B]\in \PP(V_{n,d} \times V_{n,d})$ and not on the choice of a representative $(A,B)\in V_{n,d}\times V_{n,d}$.
\end{lem}
\begin{proof} We first observe that sending $y_1$ onto $y_1\xi $ for some $\xi\in \CC^*$ sends $\Omega$ on the same variety where we replace $A,B$ with $\xi A,\xi B$, respectively. Hence, the isomorphism class of $\Omega$ only depends on $[A:B]\in \PP(V_{n,d} \times V_{n,d})$.
As $[A:B]$ satisfy the conditions of Lemma~$\ref{LemABsmoothBlowUp}$, the varieties $P,Q,R$ of Lemma~$\ref{LemABsmoothBlowUp}$ are smooth, of dimension $n,n-1,n-2$ respectively. Let us consider the open subset $\Omega_u$ of $\Omega$ where $u=1$. Here we obtain $y_1=wy_0A(x)$. In particular, $y_0\not=0$ so we may thus assume $y_0=1$ and $y_1=wA(x)$. This gives an isomorphism from $\Omega_u$ to
\[
\Omega'=\left\{\left ([x_0:\cdots:x_n],(v,w) \right)\in \PP^n\times \AA^2\mid -v^3x_0+w^3x_n+x_1=0, A(x)v=B(x) \right\}.
\]
For each $i\in \{0,\ldots,n\}$, denote by $\Omega_i\subseteq \Omega'$ the affine open subset where $x_i=1$. Consider the smooth affine open subsets $P_i\subseteq P$, $Q_i\subseteq Q$ and $R_i\subseteq R$ where $x_i=1$ and $u=1$, and take
$
f=x_0v^3-x_1,\ g=x_n\in \mathcal{O}(P_i)$. Then, $\Omega_i$ is isomorphic to $\{(x,w)\in P_i\times \AA^1 \mid w^3g=f\}$ and is thus smooth by  Lemma~\ref{Lem:SmoothAffHyp}, as $P_i$, $V_{P_i}(f)\simeq Q_i$ and $V_{P_i}(f,g)\simeq R_i$ are smooth affine open subsets. For the open subset $\Omega_v$ where $v=1$, we obtain the same proof, using the smooth affine open subsets of $P,Q,R$ with $x_i=1$ and $v=1$ (we may also simply exchange $u$ and $v$, $x_0$ with $x_1$ and $A$ with $B$).

Now let us consider the open subset $\Omega_w\subseteq \Omega$ where $w=1$, $y_1=1$. We want to show that this is smooth, and only need to consider the points where $u=v=0$, as the other points are on the previously considered charts. The fact that $[u:v:w]=[0:0:1]$ implies that $x_n=0$, $y_0A(x)=0$ and $y_0B(x)=0$. On $\Omega_w$, one has $u=y_0A(x)$ and $v=y_0B(x)$, so $\Omega_w$ is isomorphic to
\[
\{([x_0:\cdots:x_n], y_0)\in \PP^n\times \AA^1 \mid x_n+x_1(y_0A(x))^3-x_0(y_0B(x))^3=0\}.
\]
Take any $i\in \{0,\ldots,n-1\}$, restrict to the chart $x_i\not=0$, and obtain an affine hypersurface of $\AA^{n+1}$ given by one equation whose derivative with respect to $x_n$ is not zero, for each point where $x_n=y_0A(x)=y_0B(x)=0$. As the other points satisfy $u\not=0$ or $v\not=0$, our previous considerations imply that $\Omega_w$ is smooth. It remains to consider the open subset $\Omega_w'\subseteq \Omega$ where $w=1$, $y_0=1$, which is isomorphic to
\[
\{([x_0:\cdots:x_n],(y_1,u,v))\in \PP^n\times \AA^3 \mid x_n+x_1u^3-x_0v^3=0,\\ uy_1=A(x),vy_1=B(x) \}.
\]
We only need to show that points where $u=v=y_1=0$ on this chart are smooth. Note that $x_n=A(x)=B(x)=0$. Since the class $[A:B]$ satisfies Condition~\ref{UABZsmooth} of Lemma~\ref{LemABsmoothBlowUp}, the points we consider exist only when $n\geqslant 3$, and in this case the variety $Z$ from the lemma is smooth of dimension $n-3$. Considering the affine chart $x_i=1$, $i\in \{0,\ldots,n-1\}$, and computing the partial derivatives of the three polynomials $x_n+x_1u^3-x_0v^3, uy_1-A(x), vy_1-B(x)$, we get that for each point, where $u=v=y_1=x_n=A(x)=B(x)=0$, the derivatives are the same as those of $x_n, A,B$, respectively. Since $Z$ is smooth, we are done.
\end{proof}

\subsection{$3$-links of large covering genus and the proof of Theorem~\ref{3Torsion}}

In this section, we construct Sarkisov 3-links of large covering genus on Severi-Brauer surface bundles and finally prove Theorem~\ref{3Torsion}. Before that, let us recall some standard notions from higher-dimensional birational geometry, that will appear in the proofs. Let $X$ be a $\QQ$-factorial terminal algebraic variety. Recall that a \emph{flop} is a commutative diagram
\[
\xymatrix@R=6pt@C=20pt{
		X\ar@{-->}[rr]^{\chi}\ar[rd]_{\pi} &&X^{+}\ar[ld]^{\pi^{+}} \\
		& Z,
}\]
where $\pi$ and $\pi^{+}$ are small birational contractions (in particular, their exceptional sets are of codimension at least 2), $K_X$ and $K_{X^+}$ are $\pi$-trivial and $\pi^{+}$-trivial, respectively, and $\rho(X/Z)=\rho(X^{+}/Z)=1$. Given a divisor $D\in\Pic(X)$, let us denote its strict transform $\chi_*D$ by $\widetilde{D}$. The small map $\chi$ identifies the divisors on $X$ and $X^{+}$ and induces an isomorphism $\Pic(X)\iso\Pic(X^{+})$ that sends $K_X$ onto $K_{X^+}$. By \emph{the flopping locus} we mean the exceptional locus with respect to the small contraction~$\pi$.
\begin{rem}\label{rem: flop intersection number}
	Note that for every irreducible curve $C$ which is not contained in the flopping locus, one has $C\cdot K_X=\widetilde{C}\cdot K_{X^+}$. Indeed, for such a curve $C$ one has $\pi(C)=\pi^{+}(\widetilde{C})$. Since the divisor $K_Z$ is Cartier (see e.g.~\cite[Theorem 8-1-3]{Matsuki}), we can apply the projection formula to get
	\[
	C\cdot K_{X}=C\cdot\pi^* K_Z=\pi(C)\cdot K_Z=\pi^{+}(C)\cdot K_Z=\widetilde{C}\cdot(\pi^+)^*K_Z=\widetilde{C}\cdot K_{X^+}.
	\]
\end{rem}

\begin{lem}\label{lem: properties of Omega}
Let $n,d$  be  integers with $n\geqslant 2$ and $d-n\geqslant 4$, and let $U_{n,d}\subseteq \PP(V_{n,d} \times V_{n,d})$ be the open and dense subset of Lemma~$\ref{LemABsmoothBlowUp}$. For each $[A:B]\in U_{n,d}$, we denote by $\Omega=\Omega_{[A:B]}\subseteq \mathcal{F}_d^{n+1}\times \PP^2$ the variety of Lemma~$\ref{Lem:ModelDeltaSmooth}$ associated to $[A:B]$. Then, the following hold:
\begin{enumerate}
\item\label{cgWAB}
$\cg(\Omega)\geqslant \frac{d-n}{2}+1$;
\item\label{CanonicalDivisorWbignef}
The canonical divisor $K_{\Omega}$ of $\Omega$ is nef and big.
\item\label{WPn31}
The morphism
\[
\eta=\eta_{[A:B]}\colon \Omega\to \PP^n,\ \ ([x_0:\ldots:x_n;y_0:y_1],[u:v:w])\mapsto [x_0:\ldots:x_n]
\]
is generically $3:1$. Every irreducible curve $\mathcal{C}\subseteq \Omega$ not contracted by $\eta$ satisfies $\mathcal{C}\cdot K_{\Omega}\geqslant d-n\geqslant 4$. Moreover, the union
\[
\mathcal{E}_{[A:B]}=\bigcup_{\mathcal{C}\subseteq \Omega, \eta(\mathcal{C})=\pt} \mathcal{C}=\eta^{-1}(V_{\PP^n}(A,B)\cup V_{\PP^n}(x_n, x_0A^3-x_1B^3))
\]
of curves contracted by $\eta$ is covered by irreducible curves $\mathcal{C}$ with  $\mathcal{C}\cdot K_{\Omega}\leqslant 3$.
\item\label{flopWAB}
If $[C:D]\in U_{n,d}$ and if we have a commutative diagram
\begin{equation}\label{eq: birational map big and nef}
	\xymatrix@R=6pt@C=20pt{
		\Omega_{[A:B]}\ar[d]_{\eta_{[A:B]}}\ar@{-->}[rr]^{\psi} &&\Omega_{[C:D]}\ar[d]^{\eta_{[C:D]}} \\
		\PP^n \ar@{-->}[rr]^{\theta} &&  \PP^n,
	}
\end{equation}
where $\psi,\theta$ are birational, then $\psi$ is a pseudo-isomorphism and $\theta$ is an automorphism that sends $V_{\PP^n}(A,B)$ onto $V_{\PP^n}(C,D)$.
\end{enumerate}
\end{lem}

\begin{proof}
In what follows, we write $\mathcal{F}=\mathcal{F}_d^{n+1}$, $\Omega=\Omega_{[A:B]}$ and $\eta=\eta_{[A:B]}$ to simplify the notation. The Picard group of $\mathcal{F}$ is generated by the classes $\mathcal{F}_x$ and $\mathcal{F}_y$ of the divisors given by $x_0=0$ and $y_0=0$, respectively. Moreover, $x_i=0$ is equivalent to $\mathcal{F}_x$ for each $i$ and $y_{1}=0$ is equivalent to $d\mathcal{F}_x+\mathcal{F}_y$. The canonical divisor of this toric variety of dimension $n+1$ is then given by the sum of all coordinate divisors with multiplicity $-1$ and equals $K_{\mathcal{F}}=-(n+d+1)\mathcal{F}_x-2\mathcal{F}_y$. Consider the variety
\[
\widehat{\mathcal{F}}=\left\{\left ([x_0:\cdots:x_n;y_0:y_1],[u:v:w] \right )\in \mathcal{F}_d^{n+1}\times \PP^2\left| \begin{array}{l}uy_1=wy_0A(x),vy_1=wy_0B(x),\\ A(x)v=B(x)u.\end{array} \right\}\right.
\]
and observe that the projection $\pi\colon \widehat{\mathcal{F}}\to  \mathcal{F}$ is the blow-up of the subvariety $F\subseteq \mathcal{F}$ given by $A(x)=B(x)=0$ and $y_1=0$, with inverse being the birational map
\[
\mathcal{F}\dasharrow \widehat{\mathcal{F}}, ([x_0:\cdots:x_n;y_0:y_1] )\mapsto ([x_0:\cdots:x_n;y_0:y_1],[A(x)y_0:B(x)y_0:y_1]).
\]
As $F$ is smooth of codimension $3$ in $\mathcal{F}$ (Condition \ref{UABGammasmooth} of Lemma~\ref{LemABsmoothBlowUp}), we get $K_{\widehat{\mathcal{F}}}=\pi^*K_{\mathcal{F}}+2E$, where $E=F\times \PP^2\subseteq \widehat{\mathcal{F}}$ is the exceptional divisor. The Picard group of $\widehat{\mathcal{F}}$ is generated by the classes $\widehat{\mathcal{F}}_x$, $\widehat{\mathcal{F}}_y$, $\widehat{\mathcal{F}}_u$ of the divisors given by $x_0=0$ and $y_0=0$ and $u=0$, respectively. Using the morphism $\pi$, we find
\[
\widehat{\mathcal{F}}_x=\pi^*\mathcal{F}_x, \quad \widehat{\mathcal{F}}_y=\pi^*\mathcal{F}_y, \quad \widehat{\mathcal{F}}_u=d\pi^*\mathcal{F}_x+\pi^*\mathcal{F}_y-E,
\]
where the last equality comes from the fact that the rational map
\[\mathcal{F}\dasharrow \PP^2,\ \ \ [x_0:\ldots :x_n;y_0:y_1]\mapsto [y_0A(x):y_0B(x):y_1]
\]
is given by a linear system of hypersurfaces equivalent to $d\mathcal{F}_x+\mathcal{F}_y$ passing through $F$. Hence,
\[
K_{\widehat{\mathcal{F}}}=\pi^*K_{\mathcal{F}}+2E=-(n+d+1)\pi^*\mathcal{F}_x-2\pi^*\mathcal{F}_y+2E
=(d-n-1)\widehat{\mathcal{F}}_x-2\widehat{\mathcal{F}}_u.
\]
The hypersurface $\Omega\subseteq \widehat{\mathcal{F}}$ is given by $w^3x_n+x_1u^3-x_0v^3=0$ and is thus linearly equivalent to $\widehat{\mathcal{F}}_x+3\widehat{\mathcal{F}}_u$. The adjunction formula then gives
\[
K_{\Omega}=(K_{\widehat{\mathcal{F}}}+{\Omega})|_{\Omega}=((d-n)\widehat{\mathcal{F}}_x+\widehat{\mathcal{F}}_u)|_{\Omega}.
\]
We consider the variety
\[
\widehat{\Omega}=\{ ([x_0:\cdots:x_n],[u:v:w]  )\in \PP^n\times \PP^2 \mid w^3x_n+x_1u^3-x_0v^3=0, A(x)v=B(x)u \}.
\]
The morphism
\[
\rho\colon\Omega\to \widehat{\Omega},\ \ ([x_0:\cdots:x_n;y_0:y_1],[u:v:w])\mapsto ([x_0:\cdots:x_n],[u:v:w])
\]
is birational and $K_{\Omega}$ is the pull-back of an ample divisor, so $K_{\Omega}$ is big and nef \cite[2.5]{KollarMoriBirationalGeometry}. This achieves the proof of \ref{CanonicalDivisorWbignef}.

We now prove~\ref{cgWAB} and \ref{WPn31}. As $\Omega\to \widehat{\Omega}$ is birational and $\widehat{\Omega}\to \PP^n$ is generically $3:1$, the morphism  $\eta\colon \Omega\to \PP^n$ is generically $3:1$. For each irreducible curve $\mathcal{C}\subseteq \Omega$ one finds $K_{\Omega}\cdot \mathcal{C} \geqslant d-n\geqslant 4$ unless $\mathcal{C}\cdot \widehat{\mathcal{F}}_x=0$. This latter condition is equivalent to the fact that $\mathcal{C}$ is contracted by $\eta$. Since $\eta$  is generically $3:1$, the set $\mathcal{E}_{[A:B]}$, defined as the union of curves contracted by $\eta$, is not dense in~$\Omega$. By Lemma~\ref{KXgeacg}, the covering genus of $\Omega$ satisfies $\cg(\Omega)\geqslant \frac{d-n}{2}+1$. This gives \ref{cgWAB}. To finish the proof of  \ref{WPn31}, it remains to see that $\mathcal{E}_{[A:B]}$ is equal to  $\eta^{-1}(V_{\PP^n}(A,B)\cup V_{\PP^n}(x_n, x_0A^3-x_1B^3))$ and is covered by irreducible curves $\mathcal{C}\subseteq \Omega$ with  $\mathcal{C}\cdot K_{\Omega}\leqslant 3$.
For this, we take a point $p=[x_0:\cdots:x_n]\in \PP^n$  and study its preimage $\eta^{-1}(p)$. For each irreducible curve $\mathcal{C}$ contained in $\eta^{-1}(p)$ one has $\mathcal{C}\cdot \widehat{\mathcal{F}}_x=0$, and thus $\mathcal{C}\cdot K_{\Omega}=\mathcal{C}\cdot \widehat{\mathcal{F}}_u$.

Suppose first that $A(p)=B(p)=0$. Then
\[
\eta^{-1}(p)=\{([x_0:\cdots:x_n;y_0:y_1],[u:v:w])\mid w^3x_n+x_1u^3-x_0v^3=0, uy_1=vy_1=0\}
\]
is the union of a rational curve $\mathcal{C}$ where $[u:v:w]=[0:0:1]$ and $y_0,y_1$ are free, that satisfies $\mathcal{C}\cdot K_\Omega=0$, and of $\{([x_0:\cdots:x_n; 1:0],[u:v:w])\mid w^3x_n+x_1u^3-x_0v^3=0\}$. This latter set is a finite union of irreducible curves $\mathcal{C}$ with $\mathcal{C}\cdot K_\Omega\leqslant 3$ if $(x_n,x_1,x_0)\not=(0,0,0)$. If $(x_n,x_1,x_0)=(0,0,0)$, then it is a plane, which is a union of lines $\mathcal{C}$ with $\mathcal{C}\cdot K_\Omega= 1$.

Suppose that $x_n=0$ and $x_0A(p)^3=x_1B(p)^3$. We moreover assume that $(A(p),B(p))\not=(0,0)$.  Then, the preimage of $p$ in $\widehat{\Omega}$ is $\{(p,[u:v:w])\in \{p\}\times \PP^2\mid A(p)v=B(p)u\}$ and is thus a line in $\{p\}\times \PP^2\simeq \PP^2$. The fibre $\eta^{-1}(p)$ is the preimage of this line in $\Omega$. As $(u,wA,wB,v)$ are not all zero, $\eta^{-1}(p)$ is a finite union of irreducible curves $\mathcal{C}$, each of them satisfying $\mathcal{C}\cdot K_{\Omega}\leqslant 1$.

It remains to consider the case where $(A(p),B(p))\not=(0,0)$, $(x_n,x_0A(p)^3-x_1B(p)^3)\not=(0,0)$. This implies that $(x_0,x_1,x_n)\not=(0,0,0)$. Then, the preimage in $\widehat{\Omega}$ is $
\{ ([x_0:\cdots:x_n],[u:v:w]  ) \mid w^3x_n+x_1u^3-x_0v^3=0, A(x)v=B(x)u \}$ and consists of at most $3$ points as it is the intersection of a line and a cubic not containing the line. The preimage of any of these points in $\Omega$ is only one point, as  $(u, wA,v,wB)$ are not all zero. This achieves the proof of \ref{WPn31}.

To prove \ref{flopWAB}, we assume that there is a commutative diagram \eqref{eq: birational map big and nef}. As was shown in \ref{CanonicalDivisorWbignef}, the canonical divisors of $\Omega_{[A:B]}$ and $\Omega_{[C:D]}$ are both nef. Therefore, by \cite[Theorem 1]{KawamataFlopsMMP}, the birational map $\psi$ can be decomposed into
a sequence of flops (recall that $\Omega_{[A:B]}$ and $\Omega_{[C:D]}$ are in fact smooth) and is thus an isomorphism $U\iso V$, where $U\subseteq \Omega_{[A:B]}$ and $V\subseteq \Omega_{[C:D]}$ are open subsets, such that the closed subsets $\Omega_{[A:B]}\setminus U$ and $\Omega_{[C:D]}\setminus V$ have codimension at least $2$ and are covered by curves trivial against the canonical divisors.  For each irreducible closed curve $\mathcal{C}\subseteq \Omega_{[A:B]}$ with $\mathcal{C}\cap U\not=\varnothing$, the strict transform $\widetilde{\mathcal{C}}=\psi_*(\mathcal{C})$ satisfies $\mathcal{C}\cdot K_{\Omega_{[A:B]}}=\widetilde{\mathcal{C}}\cdot K_{\Omega_{[C:D]}}$, by Remark~\ref{rem: flop intersection number}. By~\ref{WPn31}, the birational map $\psi$ restricts to an isomorphism
\[
\Omega_{[A:B]}\setminus \mathcal{E}_{[A:B]}\iso \Omega_{[C:D]}\setminus \mathcal{E}_{[C:D]},
\]
and thus $\theta$ restricts to an isomorphism from
$\PP^n\setminus (V_{\PP^n}(A,B)\cup V_{\PP^n}(x_n, x_0A^3-x_1B^3))$ to  $\PP^n\setminus (V_{\PP^n}(C,D)\cup V_{\PP^n}(x_n, x_0C^3-x_1D^3)).$
As both are open subsets with complements of codimension at least $2$, we find that $\theta\in \Aut(\PP^n)$. Indeed, any birational map of $\PP^n$ can be written using homogeneous polynomials of the same degree $d\geqslant 1$ without common factor and contracts the determinant of the Jacobian of the map, given by a polynomial of degree $n(d-1)$. As $\theta$ is a pseudo-isomorphism, then $d=1$, which implies that it is an automorphism. Moreover, looking at the degrees of the closed subvarieties we have here, $\theta$ sends $V_{\PP^n}(A,B)$ onto $V_{\PP^n}(C,D)$ and sends $V_{\PP^n}(x_n, x_0A^3-x_1B^3)$ onto $V_{\PP^n}(x_n, x_0C^3-x_1D^3)$.
\end{proof}

The following result follows from a much more general Theorem 1.3 in the unpublished preprint \cite{AutCompleteIntersections}, which generalises Proposition~\ref{prop:Matsumura-Monsky} for some complete intersections; see \cite{BenoistCompleteIntersection} for related discussion. As we may take our degree $d$ large, we give a self-contained proof in the particular case which will be relevant for us.

\begin{lem}\label{Lemm:CompleteInt}
Let $n\geqslant 2$. There exists an integer $d_0\geqslant 3$ such that for each $d\geqslant d_0$, there is a set $R_{n,d}\subseteq U_{n,d}\subseteq \PP(V_{n,d} \times V_{n,d})$ $($with $U_{n,d}$,  $V_{n,d}$ as in Lemma~$\ref{LemABsmoothBlowUp})$ which has the cardinality of $\CC$ and satisfies the following properties:
\begin{enumerate}
\item\label{ci2}
For each $[A:B]\in R_{n,d}$, the zero set $V_{\PP^n}(A,B)$ is smooth of dimension $n-2$ and every automorphism of $\PP^n$ that preserves $V_{\PP^n}(A,B)$ is trivial.
\item\label{ci3}
For any two distinct points $[A:B], [C:D]\in R_{n,d}$, there is no automorphism of $\PP^n$ sending $V_{\PP^n}(A,B)$ onto $V_{\PP^n}(C,D)$.
\end{enumerate}
\end{lem}
\begin{proof}
As usual we identify $V_{n,d}$ with the affine space $\AA^N$ where $N=\binom{n+d}{n}$, and  consider the canonical morphism $\pi\colon V_{n,d}\setminus \{0\}\to \PP(V_{n,d})$ that sends an element $P$ onto its class $[P]$. Choose $d_0$ such that $N\geqslant (n+1)^2+2$ for all $d\geqslant d_0$.
Choose $A\in V_{n,d}\setminus \{0\}$ such that $\pi(A)=[A]$ is an element of $\mathcal{U}_{n,d}$, the open subset of $\PP(V_{n,d})$ defined in Proposition~\ref{prop:Matsumura-Monsky}, and such that the open set $\widehat{\mathcal{W}}=\{B\in V_{n,d}\mid [A:B]\in U_{n,d}\}\subseteq V_{n,d}$ is not empty. As $\pi$ is surjective, the image $\pi(\widehat{\mathcal{W}})$ contains a dense open subset $\mathcal{W}$ of $\PP(V_{n,d})$.

Consider the action of $\PGL_{n+1}(\CC)$ on $\PP(V_{n,d})$ by changes of variables. Denote by $F\subseteq \PP(V_{n,d})$ the closure of the union of lines through $[A]$ and any other point of the orbit  $\PGL_{n+1}(\CC)\cdot [A]$. We claim that $\dim(F)\leqslant 1+\dim\PGL_{n+1}(\CC)=(n+1)^2<\dim\PP(V_{n,d})=N-1$ and that for each line $\ell\subseteq \PP(V_{n,d})$ through $[A]$, we have either $\ell\subseteq F$ or $F\cap \ell=\{[A]\}$. To see this, we look at the blow-up $\nu\colon T\to \PP(V_{n,d})$ of $[A]$ and obtain a $\PP^1$-bundle $\theta\colon T\to \PP^{N-2}$ whose fibres are the strict transforms of lines of $\PP(V_{n,d})$ through $[A]$. Then, $F$ is the image by $\nu$ of $\theta^{-1}(F_0)$, where $F_0=\overline{\theta(\nu^{-1}(\PGL_{n+1}(\CC)\cdot [A]\setminus \{[A]\}))}$. Hence, $F_0$ satisfies $\dim(F_0)\leqslant \dim(\PGL_{n+1}(\CC))$ and thus $\dim F\leqslant\dim \theta^{-1}(F_0)= \dim(F_0)+1\leqslant 1+\dim(\PGL_{n+1}(\CC))$ as we desired. Moreover, $\theta^{-1}(F_0)$ is a union of fibres of $\theta$, so $F$ is a union of lines through $[A]$, hence the claim. As $\dim(F)<\dim \PP(V_{n,d})$, we may replace $\mathcal{W}$ with $\mathcal{W}\cap (\PP(V_{n,d})\setminus F)$ and assume that $\mathcal{W}\cap F=\varnothing$.

We now consider the projection $p_A\colon \PP(V_{n,d})\dasharrow \PP^{N-2}$ away from $[A]$. As it is surjective, $p_A(\mathcal{W})$ contains again a dense open subset of $\PP^{N-2}$. As $N-2\geqslant 1$, this dense subset contains a set of the same cardinality as $\CC$, and we can lift each of its elements to an element of $\mathcal{W}$ and then to an element of  $\widehat{\mathcal{W}}$. We obtain a subset $ \mathcal{J}\subseteq \widehat{\mathcal{W}}\subseteq V_{n,d}$ with the same cardinality as $\CC$ such that for two distinct elements $B,B'\in\mathcal{J}$ the points $[B],[B']\in\mathcal{W}\subseteq \PP(V_{n,d})$ are not collinear with $[A]$. Let $R_{n,d}$ be the set of classes $[A:B]$ with $B\in \mathcal{J}$. As $\mathcal{J}\subseteq \widehat{\mathcal{W}}$, we have $R_{n,d}\subseteq U_{n,d}$.

For each $[A:B]\in R_{n,d}$, an automorphism $\gamma\in \PGL_{n+1}(\CC)$ that preserves $V_{\PP^n}(A,B)$ also preserves  the pencil of hypersurfaces $\lambda A+\mu B=0$, $[\lambda:\mu]\in\PP^1$, as this one is the unique pencil of hypersurfaces of degree $d$ through $V_{\PP^n}(A,B)$. Hence, $\gamma$ preserves the line $\ell_{[A:B]}\subseteq \PP(V_{d,n})$ spanned by $[A]$ and $[B]$. As $[B]\in \mathcal{W}$, the line $\ell_{[A:B]}$ is not contained in $F$ and thus $\ell_{[A:B]}\cap F=\{[A]\}$. The hypersurface $V_{\PP^n}(A)$ is then fixed by $\gamma$. Since $[A]\in\mathcal{U}_{n,d}$, we conclude that $\gamma=\id$. This proves \ref{ci2}.

Similarly, assume that $\gamma\in \Aut(\PP^n)$ sends $V_{\PP^n}(A,B)$ onto $V_{\PP^n}(C,D)$, with $[A:B]$, $[C:D]\in R_{n,d}$. Note that by construction of $R_{n,d}$, we may assume $C=A$. Then the line $\ell_{[A:B]}$ spanned by $[A]$ and $[B]$ is sent onto the line $\ell_{[A:D]}$ spanned by $[A]$ and $[D]$. As before, the lines intersect $F$ only at $[A]$, so the hypersurface $V_{\PP^n}(A)$ is fixed by $\gamma$. But then $\gamma$ is trivial, so $\ell_{[A:B]}=\ell_{[A:D]}$. This gives $p_A([B])=p_A([D])$ and thus implies $B=D$ by our choice of $R_{n,d}$. This proves \ref{ci3}.
\end{proof}

\begin{proof}[Proof of  Theorem~\ref{3Torsion}]
As in the proof of Theorem \ref{thm: free product}, we write $n=m-2\geqslant 2$ and consider the groupoid $\BirMori(X)$, where $\pi\colon  X\to B$ is the rational SBMfs constructed in Proposition~\ref{prop: explicit SB fibration} with a birational morphism $B\to\PP^n$, so $\Bir_\CC(\PP^{m})\simeq\Bir(X)$. We apply Lemma~\ref{Lemm:CompleteInt} and find, for $d$ large enough, a subset $R_{n,d}\subseteq U_{n,d}\subseteq \PP(V_{n,d} \times V_{n,d})$ satisfying the properties mentioned in this lemma. For each element  $p\in R_{n,d}$ we fix $(A,B)\in V_{n,d}\times V_{n,d}$ such that $[A:B]=p\in R_{n,d}$. The variety $\Omega=\Omega_{[A:B]}$ of Lemma~$\ref{Lem:ModelDeltaSmooth}$ is smooth of dimension $n$. Moreover, it satisfies $\cg(\Omega)\geqslant \frac{d-n}{2}+1$ by Lemma~\ref{lem: properties of Omega}\ref{cgWAB}. We choose $d$ large enough such that the covering genus of $\Omega$ is bigger than $g$, where $g$ is as in Theorem~\ref{Theorem:SBMfs}. The projection of $\Omega$ on $\PP^n\times \PP^2$ gives a birational morphism to
\[
\widehat{\Omega}_{[A:B]}=\{ ([x_0:\cdots:x_n],[u:v:w]  )\in \PP^n\times \PP^2 \mid w^3x_n+x_1u^3-x_0v^3=0, A(x)v=B(x)u \}.
\]
Set $a=A(t_2,1,t_3,\ldots,t_n,t_1)$, $b=B(t_2,1,t_3,\ldots,t_n,t_1)\in\CC[t_1,\ldots,t_n]$ and consider the variety
\[
\Delta_{(A,B)}=\{(s,t_1,\ldots,t_n)\in \AA^{n+1}\mid  t_1s^3+a(t_1,\ldots,t_n)^3-t_2b(t_1,\ldots,t_n)^3=0\}
\]
of Proposition~\ref{prop:existence3linksC}. Note that its equation depends on the choice of $(A,B)\in V_{n,d}\times V_{d,n}$ and not only on $p=[A:B]$, while the isomorphism class of $\Delta_{(A,B)}$ is independent of this choice (Lemma~\ref{Lem:ModelDeltaSmooth}). Note that
$\Delta_{(A,B)}$ is birational to $\widehat{\Omega}_{[A:B]}$ via the birational map
\[
\Delta_{(A,B)}\dasharrow \widehat{\Omega}_{[A:B]}, (s,t_1,\ldots,t_n)\mapsto ([t_2:1:t_3:\cdots :t_n:t_1],[a:b:s]).
\]
Using the morphism $\eta_{(A,B)}\colon\Delta_{(A,B)}\to\PP^n$, $(s,t_1,\ldots,t_n)\mapsto [1:t_1:\ldots:t_n]$ of Proposition~\ref{prop:existence3linksC}, we obtain a commutative diagram
	\[
	 \xymatrix@R=8pt@C=40pt{
	\Delta_{(A,B)}\ar@{->}[rd]_{\eta_{(A,B)}}\ar@{-->}[r]& \widehat{\Omega}_{[A:B]}\ar[d]& \Omega_{[A:B]}\ar[l]\ar@{->}[dl]^{\epsilon_{[A:B]}}\\
	 &\PP^n,
	 }
	\]
	where $\epsilon_{[A:B]}$ sends $\left ([x_0:\cdots:x_n;y_0:y_1],[u:v:w] \right )$ onto $[x_1:x_n:x_0:x_2:\cdots:x_{n-1}]$.

Let us show that $t_2b^3-a^3$ is not a cube in $L=K[\sqrt[3]{t_1}]$. Otherwise, the polynomial  $P=a(t_1^3,t_2,\ldots,t_n)^3-t_2b(t_1^3,t_2,\ldots,t_n)^3\in \CC[t_1,\ldots,t_n]$ would be a cube, and thus the pull-back of $\Delta_{(A,B)}$ under the base change $t_1\to t_1^3$, i.e.~the variety
\[
\Delta'=\{(s,t_1,\ldots,t_n)\in \AA^{n+1}\mid  t_1^3s^3+P(t_1,\ldots,t_n)=0\},
\]
would be the union of three distinct rational varieties, impossible as $\cg(\Omega)\geqslant \frac{d-n}{2}+1$.

So, we can apply Proposition~\ref{prop:existence3linksC} and obtain a $3$-link $\chi_{(A,B)}\colon Y_{(A,B)}\dasharrow Y_{(A,B)}'$ associated to $(A,B)$. The base-loci $\Gamma_{(A,B)}\subseteq Y_{(A,B)}$ and $\Gamma'_{(A,B)}\subseteq Y'_{(A,B)}$ of $\chi_{(A,B)}$ and $\chi_{(A,B)}^{-1}$ fit into the commutative diagram \eqref{eq: base loci of 3-link}. Let us check that $\chi_{(A,B)}$ is not equivalent to its inverse, and that it is not equivalent to $\chi_{(C,D)}$ or $\chi_{(C,D)}^{-1}$ for any $[C:D]\in R_{n,d}$ different from $[A:B]$. Indeed, otherwise by Proposition~\ref{prop:existence3linksC}, Lemma~\ref{Lem:ModelDeltaSmooth} and Lemma~\ref{lem: properties of Omega}, we have a commutative diagram
	\[
	 \xymatrix@R=6pt@C=25pt{
	& \Omega_{[A:B]}\ar[dddd]^{\epsilon_{[A:B]}}\ar@{-->}[rrrr]^{\psi} &&&& \Omega_{[C:D]}\ar[dddd]^{\epsilon_{[C:D]}}	\\
	 && \Gamma'_{(A,B)}\ar@{-->}[dd]\ar@{-->}[rr]\ar[dddl] &&\Gamma'_{(C,D)}\ar[dddr]\ar@{-->}[dd] && \\
	  \Delta_{(A,B)}\ar@{-->}[uur]\ar@{-->}[rru]\ar@{-->}[rrd]\ar@{->}[ddr]_{\eta_{(A,B)}}&&&&&& \Delta_{(C,D)}\ar@{-->}[uul]\ar@{-->}[llu]\ar@{-->}[dll]\ar@{->}[ddl]^{\eta_{(C,D)}}\\
	 && \Gamma_{(A,B)}\ar[dl]\ar@{-->}[rr] &&\Gamma_{(C,D)}\ar[dr]  \\
	 &\PP^n\ar@{-->}[rrrr]^{\theta}&	 &&  & \PP^n.
	 }
	\]
By Lemma~\ref{lem: properties of Omega}\ref{flopWAB}, the birational map $\psi$ is a pseudo-isomorphism and $\theta$ is an automorphism of $\PP^n$ that sends $V_{\PP^n}(A,B)$ to $V_{\PP^n}(C,D)$. However, by the construction of the set $R_{n,d}$, we must have $[A:B]=[C:D]$ (Lemma~\ref{Lemm:CompleteInt}\ref{ci3}), whence $(A,B)=(C,D)$. It remains to show that $\chi_{(A,B)}$ and $\chi_{(A,B)}^{-1}$ are not equivalent. Otherwise, again by the construction of $R_{n,d}$ (Lemma~\ref{Lemm:CompleteInt}\ref{ci2}), we find that $\theta$ is the identity. Hence, the birational map in the equivalence of links induces the identity on the base, see Definition~\ref{Def:equiLink}. But as the generic fibres  $S$ and $S^{\rm op}$ of $Y_{(A,B)}/\PP^n$ and $Y_{(A,B)}'/\PP^n$ are not isomorphic, we conclude that $\chi_{(A,B)}$ is not equivalent to $\chi_{(A,B)}^{-1}$.

Consider the set $\mathcal{M}_3$ of Sarkisov links $\chi_{(A,B)}$ for $[A:B]\in R_{n,d}$. They all have covering genus at least $g$ and no two different elements are equivalent and no link is equivalent to its inverse. For each pair $(A,B)$, Proposition~\ref{prop:existence3linksC} implies that the generic fibre of $Y_{(A,B)}\to \PP^n$ is isomorphic to $S$, the generic fibre of $X\to \PP^n$; hence we obtain as in the proof of Theorem~\ref{thm: free product} above, that the link $\chi_{(A,B)}$ induces an element in $\BirMori(S)$. The cardinality of $\mathcal{M}_3$ equals the one of $R_{n,d}$, which is the cardinality of $\CC$.
Applying Theorem~\ref{Theorem:SBMfs} yields a surjective groupoid homomorphism $\BirMori(X)\to  \bigoplus_{\mathcal{M}_3} \ZZ/3\ZZ,$ whose restriction is a group homomorphism
\[
\rho\colon \Bir(X)\to  \bigoplus_{\mathcal{M}_3} \ZZ/3\ZZ.
\]
For each $(A,B)$ with $[A:B]\in R_{n,d}$, Proposition~\ref{prop:existence3linksC} yields an element $\psi_{[A:B]}\in \Bir(Y_{(A,B)}/\PP^n)$ of order $3$, which admits a decomposition into Sarkisov links $\chi_s\circ\cdots \circ\chi_1\circ \chi_{(A,B)}$, where $\chi_1,\ldots,\chi_s$ are Sarkisov links between SBMfs, being either isomorphisms between the generic fibres or of covering genus $0$. We take any birational map from $Y_{(A,B)}$ to $X$ and conjugate $\psi_{[A:B]}$ to obtain an element $\psi'_{(A,B)}\in \Bir(X)$ of order $3$, whose image under $\rho$ corresponds to the generator of  $\bigoplus_{\mathcal{M}_3} \ZZ/3\ZZ$ corresponding to $(A,B)$. We proved that the group homomorphism $\rho\colon \Bir(X)\to  \bigoplus_{\mathcal{M}_3} \ZZ/3\ZZ$ is surjective. This group homomomorphism factors through the abelianisation. Moreover, as any of the generators of $\bigoplus_{\mathcal{M}_3} \ZZ/3\ZZ$ is the image of an element of order $3$, the image of these elements in the abelianisation generate a subgroup of the abelianisation isomorphic to $\bigoplus_{\mathcal{M}_3} \ZZ/3\ZZ$ and thus to $\bigoplus_{\CC} \ZZ/3\ZZ$. This achieves the proof.
\end{proof}

\def\bibindent{2.5em}

\bibliographystyle{alphadin}
\bibliography{biblio}

\newcommand{\etalchar}[1]{$^{#1}$}
\begin{thebibliography}{BDPE{\etalchar{+}}17}


\providecommand{\url}[1]{\texttt{#1}}
\expandafter\ifx\csname urlstyle\endcsname\relax
  \providecommand{\doi}[1]{doi: #1}\else
  \providecommand{\doi}{doi: \begingroup \urlstyle{rm}\Url}\fi

\bibitem[Art82]{Artin}
\textsc{Artin}, M.:
\newblock \emph{Brauer-{Severi} varieties. ({Notes} by {A}. {Verschoren})}.
\newblock Brauer groups in ring theory and algebraic geometry, {Proc}.,
  {Antwerp} 1981, {Lect}. {Notes} {Math}. 917, 194-210 (1982)., 1982

\bibitem[BBLG18]{BadrBars2}
\textsc{Badr}, Eslam ; \textsc{Bars}, Francesc  ;
  \textsc{Lorenzo~Garc{\'{\i}}a}, Elisa:
\newblock The {Picard} group of {Brauer}-{Severi} varieties.
\newblock {In: }\emph{Open Math.} 16 (2018), S. 1196--1203.
\newblock \url{http://dx.doi.org/10.1515/math-2018-0101}. --
\newblock DOI 10.1515/math--2018--0101. --
\newblock ISSN 2391--5455

\bibitem[BBLG19]{BadrBars1}
\textsc{Badr}, Eslam ; \textsc{Bars}, Francesc  ;
  \textsc{Lorenzo~Garc{\'{\i}}a}, Elisa:
\newblock On twists of smooth plane curves.
\newblock {In: }\emph{Math. Comput.} 88 (2019), Nr. 315, S. 421--438.
\newblock \url{http://dx.doi.org/10.1090/mcom/3317}. --
\newblock DOI 10.1090/mcom/3317. --
\newblock ISSN 0025--5718

\bibitem[BDPE{\etalchar{+}}17]{Bastianelli}
\textsc{Bastianelli}, Francesco ; \textsc{De~Poi}, Pietro ; \textsc{Ein},
  Lawrence ; \textsc{Lazarsfeld}, Robert  ; \textsc{Ullery}, Brooke:
\newblock Measures of irrationality for hypersurfaces of large degree.
\newblock {In: }\emph{Compos. Math.} 153 (2017), Nr. 11, S. 2368--2393.
\newblock \url{http://dx.doi.org/10.1112/S0010437X17007436}. --
\newblock DOI 10.1112/S0010437X17007436. --
\newblock ISSN 0010--437X

\bibitem[Ben13]{BenoistCompleteIntersection}
\textsc{Benoist}, Olivier:
\newblock S\'{e}paration et propri\'{e}t\'{e} de {D}eligne-{M}umford des champs
  de modules d'intersections compl\`etes lisses.
\newblock {In: }\emph{J. Lond. Math. Soc., II. Ser.} 87 (2013), Nr. 1, S.
  138--156.
\newblock \url{http://dx.doi.org/10.1112/jlms/jds045}. --
\newblock DOI 10.1112/jlms/jds045. --
\newblock ISSN 0024--6107

\bibitem[Bir21]{Birkar}
\textsc{Birkar}, Caucher:
\newblock Singularities of linear systems and boundedness of {Fano} varieties.
\newblock {In: }\emph{Ann. Math. (2)} 193 (2021), Nr. 2, S. 347--405.
\newblock \url{http://dx.doi.org/10.4007/annals.2021.193.2.1}. --
\newblock DOI 10.4007/annals.2021.193.2.1. --
\newblock ISSN 0003--486X

\bibitem[BLZ21]{BLZ}
\textsc{{Blanc}}, J{\'e}r{\'e}my ; \textsc{{Lamy}}, St{\'e}phane  ;
  \textsc{{Zimmermann}}, Susanna:
\newblock {Quotients of higher dimensional Cremona groups}.
\newblock {In: }\emph{Acta Math} 226 (2021), Nr. 2, S. 211--318

\bibitem[BY20]{BlancYasinsky}
\textsc{Blanc}, J{\'e}r{\'e}my ; \textsc{Yasinsky}, Egor:
\newblock Quotients of groups of birational transformations of cubic del
  {Pezzo} fibrations.
\newblock {In: }\emph{J. {\'E}c. Polytech., Math.} 7 (2020), S. 1089--1112.
\newblock \url{http://dx.doi.org/10.5802/jep.136}. --
\newblock DOI 10.5802/jep.136. --
\newblock ISSN 2429--7100

\bibitem[CD13]{CerveauDeserti}
\textsc{Cerveau}, Dominique ; \textsc{D{\'e}serti}, Julie:
\newblock \emph{Cours Sp{\'e}c. (Paris)}. Bd.~19: {\emph{Transformations
  birationnelles de petit degr{\'e}}}.
\newblock Paris: Soci{\'e}t{\'e} Math{\'e}matique de France (SMF), 2013. --
\newblock ISBN 978--2--85629--770--4

\bibitem[Cha78]{ChangPlaneCurves}
\textsc{Chang}, H.~C.:
\newblock On plane algebraic curves.
\newblock {In: }\emph{Chin. J. Math.} 6 (1978), S. 185--189. --
\newblock ISSN 0379--7570

\bibitem[CL13]{CantatLamy}
\textsc{Cantat}, Serge ; \textsc{Lamy}, St{\'e}phane:
\newblock Normal subgroups in the {Cremona} group.
\newblock {In: }\emph{Acta Math.} 210 (2013), Nr. 1, S. 31--94.
\newblock \url{http://dx.doi.org/10.1007/s11511-013-0090-1}. --
\newblock DOI 10.1007/s11511--013--0090--1. --
\newblock ISSN 0001--5962

\bibitem[CPZ15]{AutCompleteIntersections}
\textsc{Chen}, Xi ; \textsc{Pan}, Xuanyu  ; \textsc{Zhang}, Dingxin:
\newblock \emph{Automorphism and Cohomology II: Complete intersections}.
\newblock \url{https://arxiv.org/abs/1511.07906}.
\newblock \,Version:\,2015

\bibitem[D{\'e}s07]{Deserti}
\textsc{D{\'e}serti}, Julie:
\newblock The {Cremona} group is {Hopfian}.
\newblock {In: }\emph{C. R., Math., Acad. Sci. Paris} 344 (2007), Nr. 3, S.
  153--156.
\newblock \url{http://dx.doi.org/10.1016/j.crma.2006.12.005}. --
\newblock DOI 10.1016/j.crma.2006.12.005. --
\newblock ISSN 1631--073X

\bibitem[DGO16]{DGO17}
\textsc{Dahmani}, F. ; \textsc{Guirardel}, V.  ; \textsc{Osin}, D.:
\newblock \emph{Mem. Am. Math. Soc.}. Bd. 1156: {\emph{Hyperbolically embedded
  subgroups and rotating families in groups acting on hyperbolic spaces}}.
\newblock Providence, RI: American Mathematical Society (AMS), 2016.
\newblock \url{http://dx.doi.org/10.1090/memo/1156}.
\newblock \url{http://dx.doi.org/10.1090/memo/1156}. --
\newblock ISBN 978--1--4704--2194--6; 978--1--4704--3601--8

\bibitem[Enr95]{Enriques}
\textsc{Enriques}, F.:
\newblock Conferenze di Geometria: fundamenti di una geometria iperspaziale.
\newblock {In: }\emph{Bologna}  (1895)

\bibitem[GS17]{Gille}
\textsc{Gille}, Philippe ; \textsc{Szamuely}, Tam{\'a}s:
\newblock \emph{Camb. Stud. Adv. Math.}. Bd. 165: {\emph{Central simple
  algebras and {Galois} cohomology}}.
\newblock 2nd revised and updated edition.
\newblock Cambridge: Cambridge University Press, 2017.
\newblock \url{http://dx.doi.org/10.1017/9781316661277}.
\newblock \url{http://dx.doi.org/10.1017/9781316661277}. --
\newblock ISBN 978--1--107--15637--1

\bibitem[Har81]{Harris}
\textsc{Harris}, Joe:
\newblock A bound on the geometric genus of projective varieties.
\newblock {In: }\emph{Ann. Sc. Norm. Super. Pisa, Cl. Sci., IV. Ser.} 8 (1981),
  S. 35--68. --
\newblock ISSN 0391--173X

\bibitem[HM13]{HMcK}
\textsc{Hacon}, Christopher~D. ; \textsc{McKernan}, James:
\newblock The {Sarkisov} program.
\newblock {In: }\emph{J. Algebr. Geom.} 22 (2013), Nr. 2, S. 389--405.
\newblock \url{http://dx.doi.org/10.1090/S1056-3911-2012-00599-2}. --
\newblock DOI 10.1090/S1056--3911--2012--00599--2. --
\newblock ISSN 1056--3911

\bibitem[{Isk}96]{Isk1996}
\textsc{{Iskovskikh}}, V.~A.:
\newblock {Factorization of birational maps of rational surfaces from the
  viewpoint of Mori theory}.
\newblock {In: }\emph{{Russ. Math. Surv.}} 51 (1996), Nr. 4, S. 585--652

\bibitem[Jah00]{Jahnel}
\textsc{Jahnel}, J\"org:
\newblock \emph{The Brauer-Severi Variety Associated with a Central Simple
  Algebra: A Survey}.
\newblock 2000. --
\newblock preprint available:
  \url{https://www.math.uni-bielefeld.de/LAG/man/052.html}

\bibitem[Kaw08]{KawamataFlopsMMP}
\textsc{Kawamata}, Yujiro:
\newblock Flops connect minimal models.
\newblock {In: }\emph{Publ. Res. Inst. Math. Sci.} 44 (2008), Nr. 2, S.
  419--423.
\newblock \url{http://dx.doi.org/10.2977/prims/1210167332}. --
\newblock DOI 10.2977/prims/1210167332. --
\newblock ISSN 0034--5318

\bibitem[KM08]{KollarMoriBirationalGeometry}
\textsc{Koll{\'a}r}, J{\'a}nos ; \textsc{Mori}, Shigefumi:
\newblock \emph{Camb. Tracts Math.}. Bd. 134: {\emph{Birational geometry of
  algebraic varieties. {With} the collaboration of {C}. {H}. {Clemens} and {A}.
  {Corti}}}.
\newblock Paperback reprint of the hardback edition 1998.
\newblock Cambridge: Cambridge University Press, 2008. --
\newblock ISBN 978--0--521--06022--6

\bibitem[Kol95]{Kollar_rational}
\textsc{Koll{\'a}r}, J{\'a}nos:
\newblock \emph{Ergeb. Math. Grenzgeb., 3. Folge}. Bd.~32: {\emph{Rational
  curves on algebraic varieties}}.
\newblock Berlin: Springer-Verlag, 1995. --
\newblock ISBN 3--540--60168--6

\bibitem[Kol16]{Kollar_SB}
\textsc{Koll{á}r}, J{á}nos:
\newblock \emph{Severi-Brauer varieties; a geometric treatment}.
\newblock \url{http://dx.doi.org/10.48550/ARXIV.1606.04368}.
\newblock \,Version:\,2016

\bibitem[KS99]{KatzSarnak}
\textsc{Katz}, Nicholas~M. ; \textsc{Sarnak}, Peter:
\newblock \emph{American Mathematical Society Colloquium Publications}. Bd.~45:
  {\emph{Random matrices, {F}robenius eigenvalues, and monodromy}}.
\newblock American Mathematical Society, Providence, RI, 1999. --
\newblock  xii+419 S.
\newblock \url{http://dx.doi.org/10.1090/coll/045}.
\newblock \url{http://dx.doi.org/10.1090/coll/045}. --
\newblock ISBN 0--8218--1017--0

\bibitem[KT19]{KreschTschinkelModels}
\textsc{Kresch}, Andrew ; \textsc{Tschinkel}, Yuri:
\newblock Models of {Brauer}-{Severi} surface bundles.
\newblock {In: }\emph{Mosc. Math. J.} 19 (2019), Nr. 3, 549--595.
\newblock \url{www.mathjournals.org/mmj/2019-019-003/2019-019-003-005.html}. --
\newblock ISSN 1609--3321

\bibitem[KT20]{KreschTschinkel}
\textsc{Kresch}, Andrew ; \textsc{Tschinkel}, Yuri:
\newblock Stable rationality of {Brauer}-{Severi} surface bundles.
\newblock {In: }\emph{Manuscr. Math.} 161 (2020), Nr. 1-2, S. 1--14.
\newblock \url{http://dx.doi.org/10.1007/s00229-018-1087-z}. --
\newblock DOI 10.1007/s00229--018--1087--z. --
\newblock ISSN 0025--2611

\bibitem[Lan02]{lang-algebra}
\textsc{Lang}, Serge:
\newblock \emph{Grad. Texts Math.}. Bd. 211: {\emph{Algebra.}}
\newblock 3rd revised ed.
\newblock New York, NY: Springer, 2002. --
\newblock ISBN 0--387--95385--X

\bibitem[LM21]{Lazarsfeld}
\textsc{Lazarsfeld}, Robert ; \textsc{Martin}, Olivier:
\newblock \emph{Measures of association between algebraic varieties}.
\newblock \url{http://dx.doi.org/10.48550/ARXIV.2112.00785}.
\newblock \,Version:\,2021

\bibitem[Lon16]{Lonjou}
\textsc{Lonjou}, Anne:
\newblock Non-simplicity of the {Cremona} group over any field.
\newblock {In: }\emph{Ann. Inst. Fourier} 66 (2016), Nr. 5, S. 2021--2046.
\newblock \url{http://dx.doi.org/10.5802/aif.3056}. --
\newblock DOI 10.5802/aif.3056. --
\newblock ISSN 0373--0956

\bibitem[LS24a]{LamySchneider}
\textsc{Lamy}, St\'{e}phane ; \textsc{Schneider}, Julia:
\newblock \emph{Generating the plane {C}remona groups by involutions}.
\newblock \url{http://dx.doi.org/10.14231/ag-2024-004}.
\newblock \,Version:\,2024

\bibitem[LS24b]{ShinderLin}
\textsc{Lin}, Hsueh-Yung ; \textsc{Shinder}, Evgeny:
\newblock Motivic invariants of birational maps.
\newblock {In: }\emph{Ann. Math. (2)} 199 (2024), Nr. 1, S. 445--478.
\newblock \url{http://dx.doi.org/10.4007/annals.2024.199.1.6}. --
\newblock DOI 10.4007/annals.2024.199.1.6. --
\newblock ISSN 0003--486X

\bibitem[LZ20]{LamyZimmermann}
\textsc{Lamy}, St{\'e}phane ; \textsc{Zimmermann}, Susanna:
\newblock Signature morphisms from the {Cremona} group over a non-closed field.
\newblock {In: }\emph{J. Eur. Math. Soc. (JEMS)} 22 (2020), Nr. 10, S.
  3133--3173.
\newblock \url{http://dx.doi.org/10.4171/JEMS/983}. --
\newblock DOI 10.4171/JEMS/983. --
\newblock ISSN 1435--9855

\bibitem[Mae96]{MaedaMaps}
\textsc{Maeda}, Takashi:
\newblock Birational maps of standard projective plane bundles over algebraic
  surfaces.
\newblock {In: }\emph{Ryukyu Math. J.} 9 (1996), S. 5--35. --
\newblock ISSN 1344--008X

\bibitem[Mae97]{MaedaModels}
\textsc{Maeda}, Takashi:
\newblock On standard projective plane bundles.
\newblock {In: }\emph{J. Algebra} 197 (1997), Nr. 1, S. 14--48.
\newblock \url{http://dx.doi.org/10.1006/jabr.1997.7120}. --
\newblock DOI 10.1006/jabr.1997.7120. --
\newblock ISSN 0021--8693

\bibitem[Mat02]{Matsuki}
\textsc{Matsuki}, Kenji:
\newblock \emph{Introduction to the {Mori} program}.
\newblock New York, NY: Springer, 2002 (Universitext). --
\newblock ISBN 0--387--98465--8

\bibitem[MM64]{MatsumuraMonsky}
\textsc{Matsumura}, Hideyuki ; \textsc{Monsky}, Paul:
\newblock On the automorphisms of hypersurfaces.
\newblock {In: }\emph{J. Math. Kyoto Univ.} 3 (1963/64), 347--361.
\newblock \url{http://dx.doi.org/10.1215/kjm/1250524785}. --
\newblock DOI 10.1215/kjm/1250524785. --
\newblock ISSN 0023--608X

\bibitem[Pie82]{Pierce}
\textsc{Pierce}, Richard~S.:
\newblock \emph{Grad. Texts Math.}. Bd.~88: {\emph{Associative algebras}}.
\newblock Springer, Cham, 1982.
\newblock ISSN 0072--5285

\bibitem[Rei97]{ReidChapters}
\textsc{Reid}, Miles:
\newblock Chapters on algebraic surfaces.
\newblock {In: }\emph{Complex algebraic geometry. Lectures of a summer program,
  Park City, UT, 1993}.
\newblock Providence, RI: American Mathematical Society, 1997. --
\newblock ISBN 0--8218--0432--4, S. 5--159

\bibitem[Sch22]{Schneider}
\textsc{Schneider}, Julia:
\newblock Relations in the {C}remona group over a perfect field.
\newblock {In: }\emph{Ann. Inst. Fourier (Grenoble)} 72 (2022), Nr. 1, 1--42.
\newblock \url{http://dx.doi.org/10.5802/aif.3463}. --
\newblock DOI 10.5802/aif.3463. --
\newblock ISSN 0373--0956

\bibitem[Ser79]{SerreLocalFields}
\textsc{Serre}, Jean-Pierre:
\newblock \emph{Grad. Texts Math.}. Bd.~67: {\emph{Local fields. {Translated}
  from the {French} by {Marvin} {Jay} {Greenberg}}}.
\newblock Springer, Cham, 1979.
\newblock ISSN 0072--5285

\bibitem[Ser97]{Serre}
\textsc{Serre}, Jean-Pierre:
\newblock \emph{Galois cohomology. {Transl}. from the {French} by {Patrick}
  {Ion}}.
\newblock Berlin: Springer, 1997. --
\newblock ISBN 3--540--61990--9

\bibitem[Shr20]{Shramov1}
\textsc{Shramov}, C.~A.:
\newblock Birational automorphisms of {Severi}-{Brauer} surfaces.
\newblock {In: }\emph{Sb. Math.} 211 (2020), Nr. 3, S. 466--480.
\newblock \url{http://dx.doi.org/10.1070/SM9304}. --
\newblock DOI 10.1070/SM9304. --
\newblock ISSN 1064--5616

\bibitem[Shr21]{Shramov2}
\textsc{Shramov}, Constantin:
\newblock Finite groups acting on {Severi}-{Brauer} surfaces.
\newblock {In: }\emph{Eur. J. Math.} 7 (2021), Nr. 2, S. 591--612.
\newblock \url{http://dx.doi.org/10.1007/s40879-020-00448-3}. --
\newblock DOI 10.1007/s40879--020--00448--3. --
\newblock ISSN 2199--675X

\bibitem[Sko55]{Skolem}
\textsc{Skolem}, Th.:
\newblock Einige {Bemerkungen} {\"u}ber die {Auffindung} der rationalen
  {Punkte} auf gewissen algebraischen {Gebilden}.
\newblock {In: }\emph{Math. Z.} 63 (1955), S. 295--312.
\newblock \url{http://dx.doi.org/10.1007/BF01187940}. --
\newblock DOI 10.1007/BF01187940. --
\newblock ISSN 0025--5874

\bibitem[Sta]{Stacks}
\textsc{Stacks}:
\newblock The Stacks project.
\newblock {In:
  }\emph{\href{https://stacks.math.columbia.edu/}{https://stacks.math.columbia.edu/}}

\bibitem[Ste89]{Stern}
\textsc{Stern}, Leonid:
\newblock On the norm groups of global fields.
\newblock {In: }\emph{J. Number Theory} 32 (1989), Nr. 2, S. 203--219.
\newblock \url{http://dx.doi.org/10.1016/0022-314X(89)90026-7}. --
\newblock DOI 10.1016/0022--314X(89)90026--7. --
\newblock ISSN 0022--314X

\bibitem[SV21]{ShramovVologodsky}
\textsc{Shramov}, Constantin ; \textsc{Vologodsky}, Vadim:
\newblock Boundedness for finite subgroups of linear algebraic groups.
\newblock {In: }\emph{Trans. Amer. Math. Soc.} 374 (2021), S. 9029--9046

\bibitem[Tre21]{TrepalinSB}
\textsc{Trepalin}, A.~S.:
\newblock Quotients of {Severi}-{Brauer} surfaces.
\newblock {In: }\emph{Dokl. Math.} 104 (2021), Nr. 3, S. 390--393.
\newblock \url{http://dx.doi.org/10.1134/S106456242106017X}. --
\newblock DOI 10.1134/S106456242106017X. --
\newblock ISSN 1064--5624

\bibitem[Uen75]{UenoClassificationTheory}
\textsc{Ueno}, Kenji:
\newblock \emph{Lect. Notes Math.}. Bd. 439: {\emph{Classification theory of
  algebraic varieties and compact complex spaces. {Notes} written in
  collaboration with {P}. {Cherenack}}}.
\newblock Springer, Cham, 1975

\bibitem[Vik22]{Vikulova}
\textsc{Vikulova}, Anastasia~V.:
\newblock \emph{Birational automorphism groups of Severi-Brauer surfaces over
  the field of rational numbers}.
\newblock \url{http://dx.doi.org/10.48550/ARXIV.2211.11456}.
\newblock \,Version:\,2022

\bibitem[Voi22]{Voisin}
\textsc{Voisin}, Claire:
\newblock \emph{On fibrations and measures of irrationality of hyper-K\"ahler
  manifolds}.
\newblock \url{http://dx.doi.org/10.48550/ARXIV.2201.07037}.
\newblock \,Version:\,2022

\bibitem[Wei89]{Weinstein89}
\textsc{Weinstein}, Felix:
\newblock \emph{On birational automorphisms of Severi--Brauer surfaces}.
\newblock 1989

\bibitem[Wei22]{Weinstein22}
\textsc{Weinstein}, Felix:
\newblock {On birational automorphisms of Severi-Brauer surfaces}.
\newblock {In: }\emph{{Communications in Mathematics}} {Volume 30 (2022), Issue
  1} (2022), M{\^^b a}rz.
\newblock \url{http://dx.doi.org/10.46298/cm.9040}. --
\newblock DOI 10.46298/cm.9040

\bibitem[Zik23]{Zikas}
\textsc{Zikas}, Sokratis:
\newblock Rigid birational involutions of {{\(\mathbb{P}^3\)}} and cubic
  threefolds.
\newblock {In: }\emph{J. {\'E}c. Polytech., Math.} 10 (2023), S. 233--252.
\newblock \url{http://dx.doi.org/10.5802/jep.217}. --
\newblock DOI 10.5802/jep.217. --
\newblock ISSN 2429--7100

\bibitem[Zim18]{Zimmermann}
\textsc{Zimmermann}, Susanna:
\newblock The abelianization of the real {Cremona} group.
\newblock {In: }\emph{Duke Math. J.} 167 (2018), Nr. 2, S. 211--267.
\newblock \url{http://dx.doi.org/10.1215/00127094-2017-0028}. --
\newblock DOI 10.1215/00127094--2017--0028. --
\newblock ISSN 0012--7094

\end{thebibliography}

\end{document}